\documentclass{amsart}
\usepackage{amsmath,amssymb,amsthm}
\usepackage{enumitem}
\usepackage{fixltx2e}
\usepackage{mathtools}
\usepackage{tikz}
\usepackage{tikz-cd}
\usepackage{xparse}
\usepackage{xspace}
\usepackage{mathtools}
\usepackage{dsfont}
\usepackage[all]{xy}
\usepackage{color}
\usepackage{verbatim}
\usepackage{hyperref}
\usepackage{mathrsfs}
\usepackage{microtype}
\usepackage{bm}

\usepackage{scalerel}
\usepackage{stackengine,wasysym}

\numberwithin{equation}{section}

\let\cal\mathcal
\def\Ascr{{\cal A}}
\def\Bscr{{\cal B}}
\def\Cscr{{\cal C}}

\def\Escr{{\cal E}}
\def\Fscr{{\cal F}}

\def\Kscr{{\cal K}}
\def\Lscr{{\cal L}}
\def\Mscr{{\cal M}}
\def\Nscr{{\cal N}}
\def\Oscr{{\cal O}}

\def\Qscr{{\cal Q}}
\def\Rscr{{\cal R}}
\def\Sscr{{\cal S}}

\let\blb\mathbb

\def\QQ{{\blb Q}}
\def\GG{{\blb G}}
\def \PP{{\blb P}}

\def \ZZ{{\blb Z}}

\def \NN{{\blb N}}

\def\id{\text{id}}

\def\Mod{\operatorname{Mod}}

\def\Gr{\operatorname{Gr}}

\def\gr{\operatorname{gr}}

\def\rad{\operatorname {rad}}

\def\gr{\operatorname {gr}}
\def\Spec{\operatorname {Spec}}
\def\Rep{\operatorname {Rep}}
\def\GL{\operatorname {GL}}

\def\Ext{\operatorname {Ext}}

\def\End{\operatorname {End}}
\def\RHom{\operatorname {RHom}}
\def\ulRHom{\underline{\RHom}}

\def\Sl{\operatorname {SL}}
\def\Gl{\operatorname {GL}}

\def\im{\operatorname {im}}

\def\ker{\operatorname {ker}}

\def\Tor{\operatorname {Tor}}
\def\End{\operatorname {End}}

\def\id{{\operatorname {id}}}

\def\Tot{\operatorname {Tot}}

\def\r{\rightarrow}

\def\u{\uparrow}

\def\GL{\operatorname {GL}}
\DeclareMathOperator{\Proj}{Proj}

\DeclareMathOperator{\Ind}{Ind}

\DeclareMathOperator{\coh}{coh}

\DeclareMathOperator{\Ob}{Ob}
\DeclareMathOperator{\Sym}{Sym}
\DeclareMathOperator{\Fr}{Fr}
\DeclareMathOperator{\codim}{codim}

\DeclareMathOperator{\SL}{SL}
\DeclareMathOperator{\tH}{\widehat{H}}
\DeclareMathOperator{\HHH}{H}

\theoremstyle{definition}

\newtheorem{lemma}{Lemma}[section]
\newtheorem{proposition}[lemma]{Proposition}
\newtheorem{theorem}[lemma]{Theorem}
\newtheorem{corollary}[lemma]{Corollary}

\newtheorem{definition}[lemma]{Definition}
\newtheorem{conjecture}[lemma]{Conjecture}

\newtheorem{remark}[lemma]{Remark}

\DeclareMathOperator\Hom{Hom}

\DeclareMathOperator{\Dist}{Dist}

\def\Cjk{{C^\bullet_{jk}}}
\def\gCjk{{\tilde{C}^\bullet_{jk}}}
\def\Djk{{D^\bullet_{jk}}}
\def\oCjk{{\bar{C}}^\bullet_{jk}}
\def\oDjk{{\bar{D}}^\bullet_{jk}}
\def\tCj{{{C^\bullet_{j}}}}
\def\gCj{{{\tilde{C}^\bullet_{j}}}}

\def\Xs{X^{\mathbf{s}}}
\def\tH{\hat{H}}
\def\refl{\operatorname{ref}}
\def\dur{*}
\def\Ljkt{C_{jk}^{(t)}}
\def\Ljt{C_{j}^{(t)}}
\def\sLjkt{L_{jk}^{(t)}}

\def\Djkt{D_{jk}^{(t)}}

\def\genC{\tilde{C}}
\def\bC{C}
\def\gC{\tilde{C}}
\def\un{{}}

\def\assume{\max\{n-2,3\}}
\def\W{E}

\makeatletter
\newcommand*\bigcdot{\mathpalette\bigcdot@{.5}}
\newcommand*\bigcdot@[2]{\mathbin{\vcenter{\hbox{\scalebox{#2}{$\m@th#1\bullet$}}}}}
\input ffrt.aux
\makeatother

\makeatletter
\newcommand{\pushright}[1]{\ifmeasuring@#1\else\omit\hfill$\displaystyle#1$\fi\ignorespaces}
\newcommand{\pushleft}[1]{\ifmeasuring@#1\else\omit$\displaystyle#1$\hfill\fi\ignorespaces}
\makeatother

\definecolor{ruta2}{rgb}{0.409, 0.459, 0.208}
\usepackage{xcolor}
\hypersetup{
    colorlinks,
    linkcolor={red!50!black},
    citecolor={blue!50!black},
    urlcolor={blue!80!black}
}

\mathchardef\mhyphen="2D

\newcounter{todocounter}
\DeclareDocumentCommand\addreference{g}{\stepcounter{todocounter}\todo[color = blue!30, fancyline]{\thetodocounter. Add reference\IfNoValueF{#1}{: #1}}\xspace}
\DeclareDocumentCommand\checkthis{g}{\stepcounter{todocounter}\todo[color = red!50, fancyline]{\thetodocounter. Check this\IfNoValueF{#1}{: #1}}\xspace}
\DeclareDocumentCommand\fixthis{g}{\stepcounter{todocounter}\todo[color = orange!50, fancyline]{\thetodocounter. Fix this\IfNoValueF{#1}{: #1}}\xspace}
\DeclareDocumentCommand\expand{g}{\stepcounter{todocounter}\todo[color = green!50, fancyline]{\thetodocounter. Expand\IfNoValueF{#1}{: #1}}\xspace}

\makeatletter
\newcommand{\mylabel}[2]{#2\def\@currentlabel{#2}\label{#1}}
\makeatother

\title[]{The Frobenius morphism in invariant theory II}
\author{Theo Raedschelders}
\address{(Theo Raedschelders)\newline School of Mathematics and Statistics, University of Glasgow, Glasgow G12 8QQ, 
United Kingdom\newline {\tt Theo.Raedschelders@glasgow.ac.uk}}
\thanks{The first author is supported by an EPSRC postdoctoral fellowship EP/R005214/1.}
\author{\v{S}pela \v{S}penko}
\thanks{The second author is a FWO $[$PEGASUS$]^2$ Marie Sk\l odowska-Curie fellow at the Free University of Brussels
(funded by the European Union Horizon 2020 research and innovation
programme under the Marie Sk\l odowska-Curie grant agreement
No 665501 with the Research Foundation Flanders (FWO)). During part of this work she was also a postdoc with Sue Sierra at the University of Edinburgh}
\address{(\v{S}pela \v{S}penko)\newline Departement Wiskunde, Vrije Universiteit Brussel, 
Pleinlaan $2$, B-1050 Elsene, Belgium\newline {\tt spela.spenko@vub.be}}
\author{Michel Van den Bergh}
\address{(Michel Van den Bergh)\newline Departement WNI, Universiteit Hasselt, Universitaire Campus \\
B-3590 Diepenbeek, Belgium \newline {\tt michel.vandenbergh@uhasselt.be}}

\thanks{The third author is a senior researcher at the Research Foundation Flanders (FWO).  While working on this project he was supported by
the FWO grant G0D8616N: ``Hochschild cohomology and deformation theory of triangulated categories''.}

\let\oldmarginpar\marginpar
\def\marginpar#1{\oldmarginpar{\raggedright \tiny #1}}
\def\cone{\operatorname{cone}}
\def\pdim{\operatorname{pdim}}
\def\RInd{\operatorname{RInd}}
\begin{document}

\begin{abstract}
Let $R$ be the homogeneous coordinate ring of the Grassmannian $\GG=\Gr(2,n)$  defined over an algebraically closed field $k$ of characteristic $p \geq \assume$. In this paper we give a description of the decomposition of $R$, considered as graded $R^{p^r}$-module, for $r \geq 2$. This is a companion paper to \cite{FFRT1}, where the case $r=1$ was treated, and taken together, our results imply that $R$ has finite F-representation type (FFRT). Though it is expected that all rings of invariants for reductive groups have FFRT, ours is the first non-trivial example of such a ring for a group which is not linearly reductive. As a corollary, we show that the ring of differential operators $D_k(R)$ is simple, that $\GG$ has global finite F-representation type (GFFRT) and that $R$ provides a noncommutative resolution for $R^{p^r}$.
\end{abstract}

\maketitle
\tableofcontents

\section{Introduction}
\subsection{Setting the scene}
It is hard to exaggerate the importance of the Frobenius morphism in algebraic geometry and commutative algebra. Indeed, even the most hard-line characteristic zero aficionado can only 
stand and marvel at the Deligne-Illusie proof of the Hodge-to-de Rham degeneration theorem \cite{MR894379}, the plethora of powerful cohomological vanishing results obtained using Frobenius splitting \cite{MR2107324}, or the profound simplifications the theory of tight closure has brought to commutative algebra \cite{MR1017784}. Nevertheless, many basic questions concerning the Frobenius morphism for specific rings and varieties remain unanswered. 

In this paper, we focus on Grassmannians of two-dimensional quotients, and in particular the corresponding Pl\"ucker coordinate rings, which are well known to admit a description as invariant rings for $\SL_2$. The question we are concerned with was first raised in \cite{MR1444312}, and concerns a representation theoretic property of commutative rings and varieties, dubbed finite F-representation type (FFRT), which we now discuss. 

Assume $k$ is an algebraically closed field of positive characteristic $p>0$. If a commutative $k$-algebra $R$ is reduced, then the Frobenius morphism $\Fr:R \to R$ is an isomorphism onto its image, so we can identify the inclusion $R^{p^r} \hookrightarrow R$ with the extension $R \hookrightarrow R^{1/p^r}$, where $R^{1/p^r}$ denotes the overring of $p^r$-th roots of elements of $R$. If $R=k + R_1 + \cdots$ is moreover finitely generated and $\NN$-graded, then we can view and decompose $R^{1/p^r}$ in the category of $\QQ$-graded $R$-modules, which is Krull-Schmidt. The following definition was introduced in \cite{MR1444312}. 

\begin{definition}
\label{intro:def-1}
Let $R=k + R_1 + \cdots$ be a reduced finitely generated $\NN$-graded $k$-algebra. A \textit{higher Frobenius summand} is an indecomposable summand of $R^{1/p^r}$ for some $r \geq 1$. We say that R has \textit{finite F-representation type} (FFRT) if the number of isomorphism classes of higher Frobenius summands is finite (up to degree shifts).
\end{definition}

For an extended discussion of the role that this fundamental property, and the related notion of p-uniformity, plays in algebraic geometry and commutative algebra, we refer to the introduction of \cite{FFRT1}. Here we will content ourselves with reiterating that the FFRT property has been verified essentially in only the following two cases.
\begin{enumerate}
\item \label{fin-cm}If $R$ is Cohen-Macaulay and of finite representation type (e.g. if $R$ is a quadric hypersurface)
then it satisfies FFRT. 
\item \label{lin-red}One of the main results in \cite{MR1444312} states
that the FFRT property holds for invariant rings for \emph{linearly} reductive groups.
\end{enumerate}
Motivated by potential applications in the theory of rings of differential operators
(see \S\ref{intro:diff} below) one is led to define the FFRT property in characteristic zero.
\begin{definition}[FFRT in characteristic zero] 
\label{def:FFRT-char-zero}
Let $k$ denote an algebraically closed field of characteristic zero and $R=k+R_1+\cdots$ an integral finitely generated $\NN$-graded $k$-algebra. 
 We
say that $R$ satisfies FFRT if there exists a finitely generated $\ZZ$-algebra $A\subset k$, a finitely generated graded flat $A$-algebra $R_A$ such
that $k\otimes_A R_A=R$ and
such that  for every geometric point $x:\Spec l\r \Spec A$, $\operatorname{char} l>0$, we have
that $x^\ast(R_{A})$ satisfies FFRT.
\end{definition}
It is then natural to make the following conjecture:

\begin{conjecture}[FFRT for invariant rings]
\label{con:maincon}
Let $k$ be an algebraically closed field of characteristic zero and let $W$ be a finite dimensional $k$-representation for a reductive
group $G$. Then $k[W]^G$ satisfies FFRT.
\end{conjecture} 

If $G$ is linearly reductive, i.e. it is the extension of a finite group by a torus, then this conjecture follows essentially from \eqref{lin-red} (see also the calculations in \cite[\S3.2]{MR1444312}). Moreover, in this case it is shown in \cite{MR1444312} that the Frobenius summands of $k[W]^G$ are given by so-called ``modules of covariants'': $k[W]^G$-modules of the form $M(U) := (U \otimes_k k[W])^G$ for a $G$-representation $U$ \cite{MR1209911,van1991cohen}. 

\subsection{The main result}

Though the results from \cite{MR1444312} provide some evidence, Conjecture \ref{con:maincon} seems out of reach at the moment. In particular, not a single example of an invariant ring satisfying FFRT was known, for a group which is only geometrically, but not linearly reductive (and not falling under \eqref{fin-cm}). In this paper we will provide a first class of such examples.

It is classical (see \cite[\S 3.1]{weyman2003cohomology}) that the homogeneous coordinate ring of the Grassmannian $\Gr(2,n)$ can be described using invariant theory. Let $V$ be a two dimensional vector space, and let $F$ denote a vector space of dimension $n$, for $n \geq 4$. We then continue to set 
\begin{alignat*}{2}
W&=F \otimes_k V\qquad & S&=\Sym(W)\\
G &=\Sl(V)\qquad & R &=S^G
\end{alignat*}
Then $\GG=\Proj R=\Gr(2,n)$. Note in particular that $\SL(V)=\SL_2(k)$ is not linearly reductive. The following theorem is the main application of the results in \cite{FFRT1} and this paper. 

For the remainder of the introduction, we will assume that the characteristic \[p\geq \assume.\]

\begin{theorem}(\S\ref{sec:proofsmain})
\label{intro:main-app}
The invariant ring $R$ has FFRT. 
In particular, for an algebraically closed field $k$ of characteristic zero, $k[W]^G$ satisfies Conjecture \ref{con:maincon}.
\end{theorem}

What we actually achieve in this paper, is a complete description of the indecomposable summands of $R$ as graded $R^{p^r}$-module, for $r \geq 1$, up to (nonzero) 
multiplicity, for $p$ as above. The case $r=1$ was treated in \cite{FFRT1} (and is in fact characteristic free), and the results from that paper will be used here to obtain the corresponding decomposition for $r \geq 2$. Contrary to our early expectations, passing from $r=1$ to $r \geq 2$ seems to be highly non-trivial, and requires a significant amount of work.

To formulate a version of the result, we need to introduce the summands occurring. For $j\in \NN$, define the $R$-modules of covariants 
\begin{equation}
\label{intro:moc}
S\{j\}=M(S^jV)=(S^jV\otimes_k S)^G,
\end{equation}  
where $S^jV$ denotes the $j$-th symmetric power of $V$. Then $S\{j\}$
is a Cohen-Macaulay $R$-module if $0\le j\le n-3$ (see
\cite[Proposition \ref{prop:cm}]{FFRT1}). In contrast to the linearly
reductive case, it turns out we need more than just modules of
covariants to describe the higher Frobenius summands, which partly
explains the increased difficulty.

Denoting the fundamental weight of $G=\SL_2$ by $\omega$, we will
identify $S$ with the polynomial ring $k[x_1,y_1,\ldots,x_n,y_n]$,
with $x_i$, $y_i$ having weights $+\omega$ and $-\omega$
respectively. If we put $S_+=S/(y_1,\ldots,y_n)$, then the extra
summands we need can be constructed as syzygies of the modules
\begin{equation}
M_j=\Ind_{B}^G ((j\omega)\otimes_k S_+),
\end{equation} 
where $B$ denotes the Borel subgroup (with negative roots) of $G$. More precisely, the module $M_j$ has, for $1\leq j\leq n-3$, a $G\times \GL(F)$-equivariant graded resolution of the form: 
{\footnotesize
\begin{multline}
0\r S^{n-j-2}V\otimes_k\wedge^n F\otimes_k S(-n)\r \cdots \r 
V\otimes_k \wedge^{j+3}F\otimes_k S(-j-3)\r \wedge^{j+2}F \otimes_k S(-j-2)\r\\ \wedge^j F\otimes_k S(-j)
\r
 V\otimes_k \wedge^{j-1} F\otimes_k S(-j+1)\r \cdots \r 
 S^jV\otimes_k S\,[\r  M_j\r 0],
\end{multline}
}
see \S\ref{sec:standard1}. Using this resolution, we can then define the graded $R$-modules
\begin{equation}
\label{intro:kjk}
K\{j,k\}=(\Omega^{k+1}M_j)^G.
\end{equation}
Our main theorem now states that \eqref{intro:moc} and \eqref{intro:kjk} make up all the higher Frobenius summands, up to degree shift. If $M$ is a graded $R$-module then we write $M^{\Fr^r}$ for the corresponding graded $R^{p^r}$-module via the isomorphism $\Fr^r:R \to R^{p^r}$. We first recall a (degrees and multiplicities free) version of the main theorem of \cite{FFRT1}. 

\begin{theorem}\cite[Theorem \ref{thm:mainR}]{FFRT1}\label{thm:old}
The indecomposable summands of $R$ as an $R^p$-module are up to degree and nonzero multiplicity:
\[
\{K\{j,j\}^{\Fr}\mid 1\leq j\leq n-3\}\cup \{S\{l\}^{\Fr}\mid 0\leq l\leq n-3\}.
\]
\end{theorem}

For $r \geq 2$, we obtain the following theorem, which we consider to be the main result in this paper. Remember that our standing assumption in this introduction is that $p\geq \assume$.

\begin{theorem}(\S\ref{sec:proofsmain})
\label{thm:decR} 
The indecomposable summands of $R$ as an $R^{p^r}$-module for $r\geq 2$ are up to degree and nonzero multiplicity:
\begin{equation}\label{eq:Rsummands}
\{K\{j,k\}^{\Fr^r}\mid 1\leq j, k\leq n-3\}\cup \{S\{l\}^{\Fr^r}\mid 0\leq l\leq n-3\}.
\end{equation}
\end{theorem}

From these results, it is then clear that one obtains Theorem \ref{intro:main-app} as an immediate corollary.

\subsubsection{Proof sketch}

Our approach to Theorem \ref{thm:decR} is similar to the one used in \cite{FFRT1}.  We first note that decomposing $R$ as $R^{p^r}$-module
 is equivalent to decomposing $S^{G_r}$ as
$(G^{(r)}, S^{p^r})$-module, where $G_r$ denotes the $r$-th Frobenius
kernel of $G$ (see \cite[Part I, \S 9]{jantzen2007representations}), and $G^{(r)}=G/G_r$, and the same holds for $R,S^{G_r}$ replaced by $M^G,M^{G_r}$ for a reflexive $(G,S)$-module $M$. Our proof of Theorem \ref{thm:decR} then starts with Theorem
\ref{thm:old}, which (equivalently) gives the decomposition of $S^{G_1}$ and iterates the construction (by decomposing the $G_1$-invariants of the summands of $(S^{G_1})^{\Fr^{-1}}$). In a way we are lucky, since it turns out that no new summands appear  after the second iteration.
In order to perform this iteration we use a general decomposition result \cite[Theorem \ref{prop:mainffrt1}]{FFRT1} (see Theorem \ref{prop:main}) from our previous paper. The lion's share of this paper is devoted to checking that a certain formality result, which is necessary to apply this decomposition result, holds true. A more detailed outline of the proof will be discussed later on in \S\ref{sec:proofoutline}, when we have recalled all the necessary 
concepts and notation.

\subsection{Differential operators, Grassmannians and noncommutative resolutions}

We now briefly discuss several results of a seemingly different nature, which in fact follow from our main results Theorem \ref{thm:old} and Theorem \ref{thm:decR}.

\subsubsection{Differential operators}\label{intro:diff}
The authors of \cite{MR1444312} originally introduced the FFRT
property as a possible mode of attack to a question posed by Levasseur
and Stafford \cite{LSt} regarding simplicity of rings of differential
operators for invariant rings in characteristic zero. We refer to
\cite[\S\ref{sec:diff}]{FFRT1} for more references and context. One of
the main results of \cite{MR1444312} is that if $R$ is as in
Definition \ref{intro:def-1} and satisfies FFRT as well as another
natural condition (strong F-regularity) then its ring of differential
operators is simple. Hence Theorem \ref{intro:main-app} can be used to
show the
following.

\begin{theorem}\label{thm:diff}(\S\ref{sec:proofsmain})
The ring of differential operators $D_k(R)$ is a simple ring.
\end{theorem}
Using the isomorphism $\SL(2)\cong \operatorname{Sp}(2)$ one finds
that the corresponding result in characteristic zero is true by
\cite[Corollary IV.1.4]{LSt} (see also \cite[Theorem C]{VdB13}).  As
mentioned above it would be interesting to use Theorem \ref{thm:diff}
to obtain an alternative proof of the
characteristic zero result.
This would
entail establishing the relevant
base change properties (see \cite[Question
5.1.2]{MR1444312}).
\subsubsection{Grassmannians}

For a variety $X$ defined over $k$, the pushforwards $\Fr^r_*\Oscr_X$ of the structure sheaf (and other line bundles $\Lscr$) seem to carry subtle information about $X$. Let us give some examples:
\begin{enumerate}
\item In \cite{MR3428948}, Achinger characterises smooth toric varieties in terms of splitting properties of $\Fr_*\Lscr$. A similar result is proved for ordinary abelian varieties in \cite{MR3563232}.
\item If $X$ is a projective curve, then $\Fr^r_*\Oscr_X$ carries information about the genus of the curve, see \cite[Theorem 1.3]{MR3563232}.
\item For many interesting varieties and bundles, $\Fr_*$ of a semistable bundle is again semistable \cite{MR2415312}.
\item For homogeneous varieties, $\Fr^r_*\Oscr_X$ often contains summands which are tilting bundles, see \cite[\S\ref{sec:pushforwards}]{FFRT1} for more precise statements.
\end{enumerate} 

The definition of FFRT admits an obvious global version.

\begin{definition}\cite{MR3459631}
Let $X$ denote a variety over $k$. A \textit{higher Frobenius summand} is an indecomposable summand of $\Fr^r_*\Oscr_X$ for some $r \geq 1$. We say that X has \textit{global finite F-representation type} (GFFRT) if the number of isomorphism classes of higher Frobenius summands is finite. 
\end{definition} 

There are several examples of GFFRT varieties: projective spaces and smooth quadric hypersurfaces \cite{achinger2012frobenius}, smooth toric varieties \cite{MR1752764}, $\PP^2$ blown up at four points in general position \cite{MR3459631}. It is moreover clear that for a homogeneous coordinate ring $R$ of $X$, $R$ having FFRT almost ensures that $X=\Proj(R)$ has GFFRT (one only needs to take care of the twists). In fact, if $R$ is Gorenstein, FFRT implies GFFRT (see Lemma \ref{lem:(G)FFRT}).  
As a corollary to Theorem \ref{intro:main-app}, we thus obtain the following theorem.

\begin{theorem}\label{intro:FrOscr}(\S\ref{subsec:GFFRT})
The Grassmannian $\GG=\Gr(2,n)$ has GFFRT.
\end{theorem}
In fact, also in this case we obtain a precise description of the actual vector bundles occurring as summands of $\Fr^r_*\Oscr_{\GG}$, for which we refer to \S\ref{sec:grass}.

\subsubsection{Noncommutative resolutions}

As discussed in \cite[\S\ref{sec:grass}]{FFRT1}, for many singular varieties $Y$, the Frobenius pushforward $\Fr_*\Oscr_Y$ provides a canonical \textit{noncommutative resolution} of $Y$. This means that $\Escr nd_Y(\Fr_*\Oscr_Y)$ is a sheaf of algebras on $Y$ which is locally of finite global dimension. This is true for $Y=\Spec(R)$, the affine cone over $\GG$ by \cite[Theorem \ref{thm:ncr}]{FFRT1}. In \S\ref{sec:ncrproperty}, we show that this result extends to $r >1$.

\begin{theorem}\label{thm:ncr2}(\S\ref{sec:ncrproperty}) 
Assume 
$r\geq 1$. 
Then $\Fr^{r}_*\Oscr_Y$ provides a noncommutative resolution for $Y$.
\end{theorem}

\section{Notation and conventions}\label{sec:not2}
We use mostly the same notation as in the companion paper \cite{FFRT1}. Throughout $k$ is an algebraically closed field of characteristic $p>0$. 
The assumption on $p$ will be made after introducing the necessary notation. 
\medskip

Below a $G$-module for an algebraic group $G$ is a comodule for the coordinate ring $\Oscr(G)$. A $G$-representation
is a finite dimensional $G$-module. Some ``modules'' like tilting modules are finite dimensional
by definition and so they are in fact representations. We write $G_r=\ker \Fr^{r}$ for the $r$'th Frobenius kernel \cite[Part I, \S 9]{jantzen2007representations}. We have $G/G_r\cong G^{(r)}$
where $G^{(r)}$ is the $r$-Frobenius twist of $G$. 

\medskip

If $S$ is a reduced commutative $k$-algebra and  $M$ is an $S$-module then we
  write $M^{\Fr}$ for the pullback
of $M$ under the ring isomorphism $S^p\r S:f\mapsto f^{1/p}$, so
$S^{\Fr}=S^p$. If $S=\Sym(W)$ for a $G$-representation $W$ then 
$S^{\Fr}=\Sym(W^{\Fr})$ where $W^{\Fr}=W^{(1)}$ is the usual Frobenius twist
of $W$. 

\medskip

If $U=\bigoplus_n U_n(-n)$ is a graded representation for $G$ then $U^{\Fr}=\bigoplus_n U^{\Fr}_{n}(-np)$.
Occasionally the blown up grading of $U^{\Fr}$ is not what we want and in that
case we write $U^{\overline{\Fr}}=\bigoplus_n U^{\Fr}_{n}(-n)$.

\medskip

For a reductive group $G$ we denote by $B$ its Borel subgroup, and by $H$ its 
maximal torus.\footnote{In order to avoid notation collisions a torus is denoted by the letter $H$ instead of the more standard $T$, which is reserved  for tilting modules.}   By convention, the roots of $B$ are the negative roots.
We write $X(H)$ for the weights of $H$ and $X(H)^+$ for the set of
dominant weights.  The set $X(H)$ is partially ordered by putting $\mu<\lambda$
whenever $\lambda-\mu$ is a sum of positive roots. If $\chi$ is a character of a torus $H$ 
 and $n\in \ZZ$ then we write $(n\chi)$ for the 1-dimensional $H$-representation
with character $\chi(-)^n$. 

\medskip

In most of this paper we specialize to  the following situation:
$G=\SL_2=\Sl(V)$, $\dim_k V=2$ and $S=\Sym(W)$, $R=S^G$,
where $W=V^{\oplus n}=F \otimes_k V$, $\dim_k F=n$ with $n \geq 4$. 
As in \cite{FFRT1} this is sometimes referred to as the ``standard $\Sl_2$-setting''. 

\medskip

Contrary to \cite{FFRT1}, which made no assumptions on the characteristic, we will assume  $p\geq \assume$ here, unless otherwise specified. The significance of this hypothesis is explained in \S\ref{sec:assonp}. 

\medskip

The fundamental weight of
$\SL_2$ is denoted by $\omega$. 
We denote by $S^jV$ (resp. $D^jV$) the $j$-th symmetric
(resp. divided) power representation, by $L(j)$ the simple
$\SL_2$-representation with the highest weight $j\omega$, and by $T(j)$ the
tilting module with the highest weight $j\omega$.    
We also use $(-)$ for the grading
shift but as explained in \cite[\S\ref{sec:notation}]{FFRT1} this
never leads to confusion. For preliminaries on modular representation
theory, in particular for the notion of good filtration and related
results (which we will sometimes use without further
reference), 
we refer to \cite[\S\ref{sec:prelim}]{FFRT1}.

\medskip

We use $(-)^\vee=\Hom_R(-,R)$ for the dual of a module, and $(-)^\dur=\Hom_k(-,k)$ for the dual of a  vector space.

\medskip

For a logical condition $P$ we put
\begin{equation}
\label{eq:logical}
[P]=
\begin{cases}
1&\text{if $P$ is true,}\\
0&\text{if $P$ is false.}
\end{cases}
\end{equation}
\section{Preliminaries}
\subsection{$G_r$-invariants versus higher Frobenius summands}\label{subsec:equivalence}
Recall the following definition from \cite{vspenko2015non}.
\begin{definition}
\label{def:generic1}
Let $G$  be a reductive group and let $W$ be a $G$-representation. 
Denote $X=W^\vee$. 
We say that $W$ is \emph{generic} if  
\begin{enumerate}
\item  $X$ contains a point with closed orbit and trivial stabilizer.
\item If $X^{\mathbf{s}}\subset X$ is the locus of points that satisfy (1) then $\codim
  (X-\Xs,X) \ge 2$.
\end{enumerate}
\end{definition}
Assume that $W$ is generic and put $S=\Sym(W)$, $R=S^G$.
Generalizing slightly \cite[\eqref{eq:refl}]{FFRT1}
we have inverse
symmetric monoidal equivalences 
\begin{equation}
\label{eq:reflr}
\xymatrix{
\refl(R^{p^r})\ar@/^1em/[rr]^{(S^{p^r}\otimes_{R^{p^r}}-)^{\vee\vee}}&&\refl(G^{(r)},S^{p^r})\ar@/^1em/[ll]^{(-)^{G^{(r)}}}
}
\end{equation}
between the category of reflexive $R^{p^r}$-modules
and the category of $G^{(r)}$-equivariant reflexive $S^{p^r}$-modules.

\begin{proposition}\cite[Proposition \ref{SG1}]{FFRT1} If $M$ is a reflexive $(G,S)$-module then
 $M^{G_r}$ is a reflexive $S^{p^r}$-module and moreover
$M^G$ and $M^{G_r}$ correspond to each other under the equivalences \eqref{eq:reflr}. 
\end{proposition}

\begin{corollary}\cite[Corollary \ref{cor:frob}]{FFRT1}
\label{cor:decomp}
For a reflexive $(G,S)$-module $M$ there is a 1-1 correspondence, under the equivalences \eqref{eq:reflr}, between the indecomposable summands of $M^G$ as $R^{p^r}$-module and the indecomposable summands of $M^{G_r}$ as $(G^{(r)},S^{p^r})$-module.
\end{corollary}

\subsection{Tilt-free, $\nabla$-free modules, stable syzygies}
Let $G$ be a reductive group acting on a commutative connected graded algebra $S=k+S_1+\cdots$ with good filtration.
In \cite[\S\ref{sec:projectives}]{FFRT1} we introduced certain well-behaved $(G,S)$-modules. 
A graded $(G,S)$-module $F$ is 
$\nabla$-\emph{free} if $F$ has a good filtration and $F$ is projective (hence free) as graded $S$-module, and is \emph{tilt-free} if it is of the form $T\otimes_k S$ with $T$ a graded tilting module. For equivalent definitions see loc. cit.

A tilt-free resolution (i.e. a resolution with tilt-free
modules as terms) gives rise to \emph{stable syzygies} which are well-suited for an equivariant setting.  More precisely, if $M$ is a finitely
generated graded $(G,S)$-module with good filtration and
$T\otimes_k S\to M$ is a surjection whose kernel has good filtration
(such a surjection exists in our standard $\SL_2$-setting by
\cite[Proposition \ref{prop:stable_syzygy}]{FFRT1}), then we call its
kernel a stable syzygy and write $\Omega_s M$. It is uniquely
determined up to adding $T'\otimes_k S$ with $T'$ tilting
(c.f. \S\ref{sec:stables} in loc. cit.). By $\tilde\Omega_s M$ we
denote the $(G,S)$-module obtained from $\Omega_s M$ by deleting all
summands which are $\nabla$-free.

\subsection{Some standard modules and their properties}
\label{sec:standard1}
We assume that we are in the standard $\Sl_2$-setting (\S\ref{sec:not2}). As in \cite[\S\ref{calculations}]{FFRT1} we will identify $S$ with the polynomial ring $k[x_1,y_1,\ldots,x_n,y_n]$ with $x_i$, $y_i$ having weights $+\omega$ and $-\omega$ respectively (see \S\ref{sec:not2}). Put $S_+=S/(y_1,\ldots,y_n)$. 
Here we state some facts about the modules $M_j=\Ind_{B}^G ((j\omega)\otimes_k S_+)$ that we shall use  in the sequel. We refer to \cite[\S\ref{sec:resol}]{FFRT1} for the proofs. 
We will also discuss some additional useful properties of the syzygies of $M_j$.

The module $M_j$ has for $1\leq j\leq n-3$ a $G\times \GL(F)$-equivariant graded resolution of the form: 
{\footnotesize
\begin{multline}
\label{biresolution01}
0\r D^{n-j-2}V\otimes_k\wedge^n F\otimes_k S(-n)\r \cdots \r 
D^1 V\otimes_k \wedge^{j+3}F\otimes_k S(-j-3)\r \wedge^{j+2}F \otimes_k S(-j-2)\r\\ \wedge^j F\otimes_k S(-j)
\r
 V\otimes_k \wedge^{j-1} F\otimes_k S(-j+1)\r \cdots \r 
 S^jV\otimes_k S\,[\r  M_j\r 0].
\end{multline}
}

\begin{remark}\label{rem:DtoSinresM}
Under our standing assumption $p\geq \assume$ (see \S\ref{sec:not2}), we have $T(l)=D^lV= S^lV$ for $0\leq l\leq n-3$ (see \cite[Proposition \ref{dotyhenke}]{FFRT1}). Hence the $D^lV$, $S^lV$ that appear in \eqref{biresolution01}
may be replaced by $T(l)$.
\end{remark}

The resolution \eqref{biresolution01} is induced from the $B$-equivariant Koszul resolution of $(j\omega)\otimes_k S_+$:
{\footnotesize
\begin{multline}
\label{biresolution11}
0\r ((j-n)\omega) \otimes_k \wedge^n F\otimes_k S(-n)\r \cdots\r
(-3\omega)\otimes_k \wedge^{j+3}F\otimes_k S(-j-3)\r\\
(-2\omega)\otimes_k \wedge^{j+2}F\otimes_k S(-j-2)\r 
(- \omega)\otimes_k \wedge^{j+1}F\otimes_k S(-j-1)\r
\wedge^{j}F \otimes_k S(-j)\r\\
(\omega)\otimes_k\wedge^{j-1} F\otimes_k S(-j+1)\r \cdots \r 
(j\omega)\otimes_k S\,[\r  (j\omega)\otimes_k S_+\r 0].
\end{multline}
}

We set 
\begin{align}
\label{def:kjk}
K_{jk}^\un&=\Omega^{k+1}M_j,\\
L_{jk}&=\Omega^{k+1}((j\omega)\otimes_k S_+),
\end{align}
where the syzygies are computed using \eqref{biresolution01}\eqref{biresolution11}.
In the proof of \cite[Proposition \ref{prop:resMj}]{FFRT1}, \eqref{biresolution01} was derived from \eqref{biresolution11} using the relations (without labeling $K_{jk}$)
\begin{equation}\label{eq:LK}
R\Ind_{B}^G L_{jk}=
\begin{cases}
\Ind_{B}^G L_{jk}=K_{jk}^\un&\text{if $k\leq j$,}\\
R^1\Ind_{B}^G L_{jk}[-1]=K_{j,k-1}^\un[-1]&\text{if $k\geq j+1$.}
\end{cases}
\end{equation}
Following \cite{FFRT1} we also put
\begin{equation}\label{Kj}
K_j=\Omega^{j+1}M_j(j+2)=K_{jj}(j+2)
\end{equation}
where the shift by $(j+2)$ serves as an appropriate normalization.
The properties of the modules $K_{jk}$ resemble those of  $K_j$ which were discussed in \cite[Proposition \ref{prop:Kjproperties}]{FFRT1}.
\begin{proposition}\label{prop:propKjk}
Let $1\leq j,k\leq n-3$.
\begin{enumerate}
\item\label{lem:Kjkdual}
As $(G,S)$-modules:  
$K_{jk}^\vee\cong K_{n-j-2,n-k-2}(n)\otimes_k \wedge^nF^*$. 
\item\label{lem:Kjkgoofy}$K_{jk}$ has a good filtration.
\item\label{item:Kjkinde} $K_{jk}$ is indecomposable.
\item\label{item:pdim}$\pdim K_{jk}=n-k-2$.
\item\label{item:0}
Let $K_{jk}^n$ be a shift of $K_{jk}$, such that $(K_{jk}^n)_0\neq 0$ and $(K_{jk}^n)_{<0}=0$. Then 
\[
K_{jk}^n=\begin{cases}
K_{jk}(k)&\text{if $k<j$,}\\
K_{jk}(k+2)&\text{otherwise.}
\end{cases}
\]
and
\[(K_{jk}^n)_0=\begin{cases}
\wedge^{k+1}F\otimes_k S^{j-k-1}V&\text{if $k<j$,}\\
\wedge^{k+2}F\otimes_k S^{k-j}V&\text{otherwise,}
\end{cases}
\]
\item\label{item:stabletilde}$\tilde{\Omega}^{k+1}_sM_j=\Omega^{k+1}_sM_j\cong K_{jk}$. Furthermore $\tilde{\Omega}^{l+1}_sM_l=0$ for $l\ge n-2$.
\end{enumerate}
\end{proposition}

\begin{proof}
\begin{enumerate}
\item 
By the discussion in \cite[\S \ref{proof:Mjproperties4}]{FFRT1}, the dual of \eqref{biresolution01}   is the same as the corresponding resolution of $M_{n-j-2}$ with the functor $(-)(n)\otimes_k \wedge^n F^\dur$ applied to it. The claim thus easily follows.
\item This follows by decreasing induction on $k\leq n-3$, using that the  cokernel of an injective map between modules with good filtration has good filtration. 
\item Same as the proof of \cite[Proposition 12.3(4)]{FFRT1}.
\item Clear from \eqref{biresolution01}.
\item Clear from \eqref{biresolution01}.
\item Since \eqref{biresolution01} is tilt-free (see Remark \ref{rem:DtoSinresM}), the first claim follows by \eqref{item:Kjkinde}. That $\tilde{\Omega}^{l+1}_sM_l=0$ was already shown in \cite[Proposition 12.3(5)]{FFRT1}.
\qedhere\end{enumerate}
\end{proof}

The following result is an immediate corollary of Proposition \ref{prop:propKjk}\eqref{item:pdim}\eqref{item:0}. 
\begin{corollary}
All $(G,S)$-modules $K_{jk}$ for $1\leq j,k\leq n-3$ are non-isomorphic and non-free as $S$-modules.
\end{corollary}

\subsection{Preliminaries on $B_1$-cohomology}\label{subsec:prel}
We assume $G=\SL_2$ with the Borel subgroup $B$ (with negative roots) and maximal torus $H$. Also, set $U:=\rad(B)$. 
Here we give a slight generalization of \cite[\S10.2]{FFRT1}. 
 We will use the stable category $\underline{\Mod}(B_1)$ of $B_1$-modules with $\Hom$-spaces denoted by $\underline{\Hom}_{B_1}$, and Tate cohomology for $B_1$ defined by $\widehat{H}^\bullet(B_1,M):=\underline{\Ext}^\bullet_{B_1}(k,M)$ for which we refer the reader to \cite[Appendix B]{FFRT1}. 

For a bounded $B$-equivariant complex $M$,  
\begin{equation}
  \underline{\RHom}_{B_1}(k,M)=\allowbreak\underline{\RHom}_{U_1}(k,M)^{H_1}
  \end{equation} 
  (recall $H_1$ is linearly reductive) in $D(B^{(1)})$,
   and $\underline{\RHom}_{U_1}(k,M)$ is computed by 
the total  complex of the double complex
\begin{equation}
\label{complexa}
\cdots \xrightarrow{F}M\otimes_k (-2(p-1)\omega)\xrightarrow{F^{p-1}} M\xrightarrow{F} M\otimes_k (2\omega)\xrightarrow{F^{p-1}} M\otimes_k (2p\omega)\xrightarrow{F}\cdots 
\end{equation}
with the second $M$ occurring in horizontal degree zero. Here we used a complete resolution of $k$ as a $\Dist(U_1)$-module, where $\Dist(U_1)=k[F]/(F^p)$ is the distribution algebra of $U_1$, and $F$ has weight $-2\omega$ (see \cite[\S II.12]{jantzen2007representations}).  We  obtain the following periodicity result. 
\begin{lemma} \label{lem:periodicity} There exists a natural isomorphism between 
$\underline{\RHom}_{B_1}(k,-)[2]$ and $\underline{\RHom}_{B_1}(k,-) \otimes_k (2p\omega)$ viewed as functors $D(B)\r D(B^{(1)})$.
\end{lemma}

We will often use the cohomology and formality in $D(B^{(1)})$ of the complexes $\underline{\RHom}_{B_1}(k,L(i)\otimes_k (a\omega))$ for $0\leq i\leq p-2$, which can be easily checked as in the proof of  \cite[Lemma 10.2]{FFRT1}, using the complex \eqref{complexa}.

\begin{lemma}\cite[Lemma \ref{B1cohomology0}]{FFRT1}\label{B1cohomology02}
If $i\in \{0,\ldots, p-2\}$ then $\underline{\RHom}_{B_1}(k,L(i)\otimes_k (a\omega))$ is formal as object in $D(B^{(1)})$. Moreover we have as $B^{(1)}$-representations:
\begin{small}
\begin{equation}
\label{TateB1formulaA}
\tH^l(B_1,L(i)\otimes_k (a\omega))=
\begin{cases}
((lp-(i-a))\omega)&\text{if $i\equiv a\,(p)$ and $l \equiv 0\,(2)$,}\\
((lp+(i-(p-2-a)))\omega)&\text{if $i\equiv p-2-a\,(p)$ and $l \equiv 1\,(2)$,}\\
0&\text{otherwise.}
\end{cases}
\end{equation}
\end{small}
\end{lemma}
In particular, Lemma \ref{B1cohomology02} also  yields the formality of  $\underline{\RHom}_{B_1}(k,(j\omega))$
(which is particularly easy to see since all maps in \eqref{complexa} are zero) and \eqref{TateB1formulaA} computes the Tate cohomology of $(j\omega)$.

If $M$ is a $B$-module, then the ordinary $B_1$-cohomology (i.e.\ non-Tate) of $M$ 
can be computed by replacing \eqref{complexa} by its truncation in degree zero. For use below we record the following consequence
\begin{equation}
\label{eq:B1char}
H^i(B_1,(j\omega))
=
\begin{cases}
\hat{H}^i(B_1,(j\omega))&\text{if $i\ge 0$},\\
0&\text{if $i< 0$}.
\end{cases}
\end{equation}

\section{Higher Frobenius kernel invariants}\label{sec:introHFr}
First recall the decomposition  of $S^{G_1}$  which we obtained in~\cite{FFRT1} and which is valid in arbitrary characteristic. 
\begin{theorem}\label{thm:decFr1}
The indecomposable summands of  the graded $(G^{(1)},S^p)$-module $S^{G_1}$ are up to (nonzero if $p\geq n-2$) multiplicity  and grading shift: 
\[
\{T(0)^{\Fr}\otimes_k S^p,T(1)^{\Fr}\otimes_k S^p,\ldots,T(n-3)^{\Fr}\otimes_k S^p,K_1^{\Fr},K_2^{\Fr}\ldots,K_{n-3}^{\Fr}\}.
\]
\end{theorem}
Specializing to $p\ge \assume$ (our standing assumption) we will obtain in this paper a corresponding decomposition of $S^{G_r}$ 
 for $r\ge 2$.
\begin{theorem}\label{thm:higherFr} 
The indecomposable summands of the graded $(G^{(r)},S^{p^r})$-module $S^{G_r}$ for $r\geq 2$ are up to (nonzero) multiplicity and grading shift:
\begin{equation}\label{eq:summandsSr}
\{K_{jk}^{\Fr^r}\mid 1\leq j,k\leq n-3\}\cup \{(S^lV)^{\Fr^r}\otimes_k S^{p^r}\mid 0\leq l\leq n-3\}.
\end{equation}
\end{theorem}

\subsection{Proof outline}
\label{sec:proofoutline}
Our proof of Theorem \ref{thm:higherFr} starts with Theorem
\ref{thm:decFr1} which gives the decomposition of $S^{G_1}$ and then
iterates the construction. More precisely, instead of directly decomposing
$S^{G_r}$ we decompose the $G_1$-invariants of the summands occurring in
the decomposition of $(S^{G_{r-1}})^{\Fr^{-(r-1)}}$. Starting with
$r=2$ we need to first decompose $K_i^{G_1}$ and
$(T(i)\otimes_k S)^{G_1}$. It turns out that their non-tilt-free summands are among the $K_{jk}^{\Fr}$ 
 so to continue we need to
decompose $K_{jk}^{G_1}$.  We will show that its summands are again
(up to tilt-free summands) among the $K_{lm}^{\Fr}$, which completes
the iteration after the second step.
Summarizing, in order to prove Theorem \ref{thm:higherFr} we need to 
find the indecomposable summands of $K_{jk}^{G_1}$ and $(T(j)\otimes_k S)^{G_1}$ as $(G^{(1)},S^p)$-modules.
To this end we will use the following general result.
\begin{theorem}\cite[Theorem \ref{prop:mainffrt1}]{FFRT1}
\label{prop:main} 
Assume that we are in the standard $\SL_2$-setting (\S\ref{sec:not2}) and assume in addition the following
\begin{enumerate}
\item \label{aa} $M^\bullet$ is a complex of finitely generated $\nabla$-free modules
  concentrated in degree $\ge 0$.
\item \label{bb} The $G$-representations $\HHH^i(M^\bullet)$ have good filtrations for all $i$.
\item \label{rr} The reflexive $(G,S)$-module $\HHH^0(M^\bullet)=Z^0(M^\bullet)$ has the property that $\Hom_S(\HHH^0(M^\bullet),S)$ has a good filtration.
\item \label{dd} $\tau_{\ge 1}M^\bullet$ is formal in the derived category of $(G,S)$-modules.
\end{enumerate}
Then we have as graded $(G,S)$-modules
\begin{equation}
\label{eq:maindecompositionformula}
\HHH^0(M^\bullet)\cong T\otimes_k S\oplus \bigoplus_{j\ge 1} \tilde{\Omega}^{j+1}_s \HHH^j(M^{\bullet}),
\end{equation}
where $T$ is a graded tilting module.
\end{theorem}

  We will use Theorem \ref{prop:main} in order to
find the summands of $K_{jk}^{G_1}$ and $(T(j)\otimes_k S)^{G_1}$ as
$(G^{(1)},S^p)$-modules.  To do this we have to find complexes
$M^{\bullet}$, $N^{\bullet}$ satisfying \eqref{aa}-\eqref{dd} such
that $\HHH^0(M^\bullet)=K_{jk}^{G_1}$,
$\HHH^0(N^\bullet)=(T(j)\otimes_k S)^{G_1}$.  Let $C(k)$ be a left
projective $G_1$-resolution of~$k$ consisting of $G$-tilting modules
(see \cite[\S\ref{sec:proofoffinaldecompsub}]{FFRT1}).  In the case $j=0$ we used the complex
$N^{\bullet}=\Hom_{G_1}(C(k),S)$ in \cite[\S\ref{sec:proofoffinaldecompsub}]{FFRT1}, and by
slightly modifying the proof in loc.\ cit.\ it follows that in general
we can take $N^{\bullet}=\Hom_{G_1}(C(k),T(j)\otimes_k
S)$.
We can however not take $M^\bullet=\Hom_{G_1}(C(k),K_{jk})$
since then the freeness condition in \eqref{aa} would not be satisfied.  We will therefore look
for a complex $\genC^\bullet$ of free $S$-modules such that
$M^{\bullet}=\Hom_{G_1}(C(k),\genC^{\bullet})$ has the
desired properties. 

A natural choice for $\genC^\bullet$ is to choose a (shifted)
truncation of the complex \eqref{biresolution01}; i.e., in order to
obtain $\HHH^0(M^\bullet)=K_{jk}^{G_1}$ we should choose $\genC^\bullet$ to be the stupid truncation $\sigma_{\geq -k}$ of \eqref{biresolution01} shifted by $[-k]$. 
However, it seems more convenient to work first on the level of
$B$-modules and use a truncation of the complex \eqref{biresolution11}.
Then the truncations at the $G$-level are obtained by applying
$\RInd_{B^{(1)}}^{G^{(1)}}$ to the truncations at the $B$-level.

\subsection{Notes on the assumption {\boldmath $p\geq \assume$}}\label{sec:assonp}
While the decomposition of $S^{G_1}$ in~\cite{FFRT1} is characteristic-free, we were not able to establish the decomposition of $S^{G_r}$ for $r\geq 2$ in the same generality.

The advantages of the assumption $p\geq \assume$ are the following:
\begin{enumerate}
\item\label{tiltfreeC}
The resolution \eqref{biresolution01} of $M_j$ is tilt-free. 
Therefore, in the above construction the components of $\tilde{C}^\bullet$ and $M^\bullet$ have good filtrations, which is needed to verify Theorem \ref{prop:main}(\ref{aa},\ref{bb}). 

\item\label{item:stable}
Since \eqref{biresolution01} is a tilt-free resolution of $M_j$ 
 we can use it to compute stable syzygies 
 of $\tilde{\Omega}_s^*M_l$ (noting that $H^j(M^\bullet)$ for $j\geq 1$ turns out to be a sum of $M_l$ for $l$ in a suitable interval). 

 \item\label{quickdeg}
The assumption  
 makes the cohomology of our chosen $M^\bullet$ easy to estimate and moreover the corresponding double complex spectral sequence is sparse and quickly degenerates, which leads to the formality statement \eqref{dd} in Theorem \ref{prop:main}. 
 \end{enumerate}

 While \eqref{tiltfreeC} still applies for the decomposition of
 $K_{j}^{G_1}$ in small characteristic (as in this case the truncation
 of $\tilde{C}^\bullet$ is $\nabla$-free), the summands would be given
 by syzgyies of the tilt-free resolution of $M_l$ (under the
 assumption that the complex $\tau_{\ge 1}M^\bullet$ is formal and that the 
 cohomology is still given by suitable $M_l$), causing
 \eqref{item:stable} to fail and we would need to use (truncations of)
 tilt-free resolutions of $M_l$ in the next iteration step.  However,
 the current proof of even the formality of the complexes $\tau_{\ge 1}M^\bullet$ used for 
 decomposition of $K_j^{G_1}$ appears to heavily depend on
 \eqref{quickdeg}.

\medskip

Nevertheless, despite all these hurdles, preliminary computations for $p=2$, suggest that the FFRT property may hold for general $p$ in our
standard setting.

\section{Decomposition of Frobenius kernel invariants for tilt-free modules}
In contrast to our standing hypothesis on $p$ \emph{in this section the characteristic $p$ of $k$ is arbitrary}. We obtain the decomposition of the
graded $(G^{(1)},S^p)$-module $(T(j)\otimes_k S)^{G_1}$ for
$j\in \NN$. This result is interesting in its own right but as explained in \S\ref{sec:proofoutline} it is also an important step in the ``iteration part'' of the proof of Theorem \ref{thm:higherFr}. 

\begin{proposition}\label{prop:mocdecomposition}
The indecomposable summands of the graded $(G^{(1)},S^p)$-module $(T(j)\otimes_k S)^{G_1}$ are up to (possibly zero) 
 multiplicity and grading shift:
\begin{enumerate}
\item 
for $j<p-1$ 
\begin{equation}
\label{eq:nontiltfreecase}
\{T(0)^{\Fr}\otimes_k S^p,T(1)^{\Fr}\otimes_k S^p,\ldots,T(m)^{\Fr}\otimes_k S^p,K_1^{\Fr},K_2^{\Fr}\ldots,K_{n-3}^{\Fr}\}, 
\end{equation}  
\item 
for $j\geq p-1$
\begin{equation}
\label{eq:tiltfreecase}
\{T(0)^{\Fr}\otimes_k S^p,\ldots,T(m)^{\Fr}\otimes_k S^p\},
\end{equation}
\end{enumerate}
where $m=\lfloor((n-2)(p-1)+j)/p\rfloor$, 

Moreover, in (1) all summands appear with nonzero multiplicity while in (2) the summand $T(l)^{\Fr}\otimes_k S^p$ appears with nonzero multiplicity if $j_2\leq l\leq m$,  
where we write $j$ as $j_1+pj_2$ for ($0\leq j_1\leq p-2\wedge j_2=0)\vee (p-1\leq j_1\leq 2p-2)$. 
\end{proposition}

We will need several preparatory lemmas before embarking upon the proof of Proposition \ref{prop:mocdecomposition}, which are analogues of \cite[\eqref{eq1subsub},\eqref{eq:4seq}]{FFRT1}. 
To state the lemmas we first need to introduce some extra notation. Let $\Nscr^j$ be the set of $n+1$-tuples $t=(q_t,t_1,\ldots,t_n)\in \NN\times [p-2]^n$ for which $L(q_t)$ is a direct summand of the ``fusion'' tensor product (the non-$G_1$-projective part of the tensor product, see \cite[\S\ref{tensor}]{FFRT1}) $L(j)\underline{\otimes} L(t_1) \underline{\otimes} L(t_2)  
\otimes\allowbreak \cdots \underline{\otimes} L(t_n)$. 
Let $n_t$ be the multiplicity of the corresponding summand $L(q_t)$ and put $d_t=\sum t_i$. 

\begin{lemma}\label{predzadnja}
For $j<p-1$ and $l\geq 1$ there is an isomorphism of graded $(G^{(1)},S^p)$-modules 
\begin{equation}
\label{eq1subsubj}
H^l(G_1,T(j)\otimes_k S)\cong \bigoplus_{\substack{t\in \Nscr^j,q_t=0, l \equiv 0 (2) \text{ or}\\ q_t=p-2, l \equiv 1 (2)}}M_{l}^{\Fr}(-d_t)^{\oplus n_t}
\end{equation} 
Moreover, $H^l(G_1,T(j)\otimes_k S)$ is nonzero.
\end{lemma}

\begin{proof}
We proceed as in the proof of \cite[Propositions
\ref{mainprop0}, \ref{mainprop}]{FFRT1}. Recall that in loc.\ cit.\ we first use $B_1$-cohomology and then relate it to $G_1$-cohomology via
the Andersen-Jantzen spectral sequence (see \cite[\eqref{identity2}]{FFRT1}).
Now
\begin{equation}
\label{eq:tatefusion}
\tH^l(B_1,T(j)\otimes_k S)=\tH^l(B_1,T(j)\otimes_k\widetilde{M})\otimes_k S^p_+=\tH^l(B_1,L(j)\underline{\otimes} \widetilde{M})\otimes_k S^p_+,
\end{equation} 
where
\begin{equation}
\label{eq:mult}
\widetilde{M}=\bigoplus_{(i_1,\ldots,i_n)\in [p-2]^n} L(i_1)(-i_1) \underline{\otimes} L(i_2)(-i_2)\underline{\otimes} \cdots \underline{\otimes} L(i_n)(-i_n).
\end{equation}
Note that in the last equality of \eqref{eq:tatefusion} we also used that since $j<p-1$, we have $T(j)=L(j)$ (see \cite[Proposition \ref{dotyhenke}]{FFRT1}). The summands of $L(j)\underline{\otimes} \widetilde{M}$ are $G$-representations of the form $L(q_t)$ for $t\in \Nscr^j$ and by Lemma \ref{B1cohomology02} (applied with $a=0$) those summands will have nonzero $B_1$-cohomology if and only if either $l$ is even and $q=0$ or $l$ is odd and $q=p-2$.  In that case the $B_1$-cohomology is equal to $(lp\omega)$. 
Passing from $B_1$-cohomology to $G_1$-cohomology as in the proof of \cite[Proposition  \ref{mainprop}]{FFRT1} we then obtain \eqref{eq1subsubj}. 

It remains to argue that the modules are nonzero. This can be done on the level of $B_1$-cohomology. 
  From the explicit description of
$\underline{\otimes}$ (see \cite[\S\ref{tensor}]{FFRT1}) one obtains 
\begin{align}\label{eq:L(j)tnzL(j)}
L(j)\underline{\otimes}L(j)&=L(0)\oplus \cdots \\
L(j)\underline{\otimes} L(p-2-j)&=L(p-2)\oplus\cdots.
\end{align}
Fusion tensoring these identities with $L(0)^{\underline{\otimes} n-1}$ we find 
that $L(0)$ and $L(p-2)$ do indeed occur with nonzero multiplicity in $L(j)\underline{\otimes} \widetilde{M}$ and thus $\tH^l(B_1,T(j)\otimes_k S)$ equals $(lp\omega)\otimes_k S^p_+$ 
up to nonzero multiplicity and degrees. Thus, indeed at least one $n_t$ for $t$ satisfying the conditions in \eqref{eq1subsubj} is nonzero.
\end{proof}

\begin{lemma}\label{zadnja}
There is an exact sequence of graded $G^{(1)}$-modules
\begin{multline}\label{eq:tor}
0\r \oplus_{i=1}^\infty \Tor^{S^p}_{i+1}(H^i(G,T(j)\otimes_k S),k)\r 
(T(j)\otimes_k S)^{G_1}/(T(j)\otimes_k S)^{G_1}S^p_{>0}\\\r (T(j)\otimes_k S/S^p_{>0}S)^{G_1}\r \oplus_{i=1}^\infty \Tor^{S^p}_i(H^i(G,T(j)\otimes_k S),k)\r 0.
\end{multline}
Moreover, the left- and the right-most summands are direct sums of copies of the trivial $G^{(1)}$-representation living in degrees $\geq p$.
\end{lemma}

\begin{proof}
The proof of the first claim is the same as in \cite[Lemma \ref{lem:4seq}]{FFRT1}. The second claim is proved by arguing as in the proof of \cite[Lemma \ref{lem:tor}]{FFRT1} using Lemma \ref{predzadnja} (i.e. by Lemma \ref{predzadnja} one needs to compute $\Tor^{S^p}_{l}(M_i^{\Fr},k)$, $l\in \{i,i+1\}$, which can be calculated  by employing the free $(G,S)$-resolutions  of $M_i$ (see \eqref{biresolution01} for $0\leq i\leq n-3$ and the analogous ones induced from \eqref{biresolution11} for $i\geq n-2$)). 
\end{proof}

\subsection{Proof of Proposition \ref{prop:mocdecomposition}}
As discussed in \S\ref{sec:proofoutline} we will apply Theorem \ref{prop:main} with  $M^\bullet=\Hom_{G_1}(C(k),\allowbreak T(j)\otimes_k S)$ and with $C(k)$ (as mentioned in \S\ref{sec:proofoutline}) a left
projective $G_1$-resolution of~$k$ consisting of $G$-tilting modules
(see \cite[\S\ref{sec:proofoffinaldecompsub}]{FFRT1}). 
To this end we need to verify the conditions (1)-(4). 
The proof goes along similar lines as the proof of \cite[Theorem \ref{mainth}]{FFRT1} in \cite[\S \ref{sec:mainproof}]{FFRT1} therefore we only indicate the necessary modifications. We use several times that the tensor product of modules with good filtration has a good filtration by \cite{MR1054234} (see \cite[Proposition \ref{tensorgoodfiltration}]{FFRT1}).

\begin{enumerate}
\item Similar to \cite[\S \ref{sec:proofoffinaldecompsub}\eqref{mainffrt1-3}]{FFRT1}. We can apply \cite[Lemma \ref{free}]{FFRT1} with $T(j)\otimes_k S$ instead of $S$ since $P\otimes_k T(j)$ is a projective $G_1$-representation with good filtration if the same holds for $P$. 
\item The same as \cite[\S \ref{sec:proofoffinaldecompsub}\eqref{mainffrt1-4}]{FFRT1}, using that $T(j)\otimes_k S$ has a good filtration.
\item  
Similar to \cite[\S \ref{sec:proofoffinaldecompsub}\eqref{mainffrt1-5}]{FFRT1}. 
By \eqref{Sdual} in loc. cit. we have an isomorphism $\Hom_{S^p}((T(j)\otimes_k S)^{G_1},S^p)
\cong \allowbreak (T(j)\otimes_k S)^{G_1}(2n(p-1))$.
Thus, the result follows by \cite[Corollary 2.2]{MR1202803} 
(see \cite[Theorem 4.13]{FFRT1}).
\item  Similar to \cite[\S \ref{sec:proofoffinaldecompsub}\eqref{mainffrt1-6}]{FFRT1}.
We follow the proof of \cite[Theorem \ref{formalityy}]{FFRT1}. 
 Tensoring the filtration \eqref{eq:filtffrt1} on $S$ in loc. cit. with $T(j)$, we obtain the filtration $(F'_iT_j)_{i=-1}^n$ of $T_j:=T(j)\otimes_k S$, 
such that $F'_iT_j/F'_{i-1}T_j$ is a projective $B_1$-representation for $i<n$ and $T_j/F'_{n-1}T_j=T(j)\otimes_k\widetilde{M}\otimes_k S^p_+$. We can then proceed as in loc. cit.
\end{enumerate}
As a result, we obtain the decomposition
\begin{equation}
\label{eq:result}
(T(j) \otimes S)^{G_1} \cong T^{\overline{\Fr}} \otimes_k S^p \oplus \bigoplus_{l\ge 1} \tilde{\Omega}^{l+1}_s \HHH^l(G_1,T(j)\otimes_k S)^{\Fr},
\end{equation}
where $T^{\overline{\Fr}} \otimes_k S^p$ denotes the tilt-free summand (see
\S\ref{sec:not2} for $\overline{\Fr}$).

We first discuss the non-tilt-free summands of $(T(j)\otimes_k S)^{G_1}$. To this end, we have to compute the $(G^{(1)},S^p)$-modules $H^l(G_1,T(j)\otimes_k S)=\tH^l(G_1,T(j)\otimes_k S)$ for $l\ge 1$ (see \cite[\eqref{eq:nontate}]{FFRT1}). We split this up into two cases, depending on the value of $j$.

For $j\ge p-1$, $T(j)$ is $G_1$-projective (see \cite[Proposition \ref{cor:tiltingkernel}]{FFRT1}) and hence
the same holds for $T(j)\otimes_k S$.  So 
\begin{equation}
\label{eq:coh-zero}
\tH^l(G_1,T(j)\otimes_k S)=0
\end{equation} 
for all $l$, and by \eqref{eq:result} there are no non-tilt-free summands in $(T(j) \otimes_k S)^{G_1}$, partially proving \eqref{eq:tiltfreecase}. 

For $j<p-1$ we use Lemma \ref{predzadnja} to obtain    
\begin{equation}
\bigoplus_{l\ge 1} \tilde{\Omega}^{l+1}_s \HHH^l(G_1,T(j)\otimes_k S)^{\Fr}=
\bigoplus_{l=1}^{n-3}\bigoplus_{\substack{t\in \Nscr^j,q_t=0, l \equiv 0 (2) \text{ or}\\ q_t=p-2, l \equiv 1 (2)}} K_l^{\Fr}(-p(l+2)-d_t)^{\oplus n_t},
\end{equation} 
using that $\tilde{\Omega}^{l+1}_{s} M_l=0$ for $l\geq n-2$. Moreover, up to (nonzero) multiplicities we obtain all $K_l^{\Fr}$ for $1\leq l\leq n-3$, which in particular live in degrees $\geq p$, as non-tilt-free summands in $(T(j) \otimes_k S)^{G_1}$.

\medskip

For the tilt-free summands of $(T(j)\otimes_k S)^{G_1}$ we can proceed
as in \cite[\S\ref{subsec:compT},\S\ref{sec:combdesc}]{FFRT1}.  Define
$N=(T(j)\otimes_k S/S^p_{>0}S)^{G_1}$. We will again consider two cases, according to the value of $j$.
For $j\ge p-1$, by combining Lemma \ref{zadnja} with \eqref{eq:coh-zero}, we find that
$T^{\overline{\Fr}}=N$. 
For $j<p-1$, we claim that $T^{\overline{\Fr}}$ and $N$ have the same indecomposable summands as $G^{(1)}$-representations. 
 To see this we proceed as in
\cite[\S\ref{subsec:compT}]{FFRT1}.  
We first obtain an isomorphism as $G^{(1)}$-modules (analogous to \cite[\eqref{finaldecomp2}]{FFRT1}):
\begin{multline}\label{eq:5.12}
(T(j) \otimes_k S)^{G_1}/S_{>0}^p(T(j) \otimes_k S)^{G_1} \\ \cong  \left(\bigoplus_{l=1}^{n-3} \bigoplus_{\begin{smallmatrix}t\in \Nscr^j,q_t=0,l\equiv 0\,(2)\text{ or }\\ q_t=p-2,l\equiv 1\,(2)\end{smallmatrix}} 
\wedge^{l+2}F(-p(l+2)-d_t)^{\oplus n_t}\right)\oplus T^{\overline{\Fr}},
\end{multline}
where we used Proposition \ref{prop:propKjk}\eqref{item:0}. In particular, the non-tilting part lives in degrees $\geq p$. 
We now need to compare 
$(T(j)\otimes_k S)^{G_1}/S^p_{>0}(T(j)\otimes_k S)^{G_1}$ to $N$. 
From Lemma \ref{zadnja} (and \eqref{eq:5.12}) it follows that 
$T^{\overline{\Fr}}$ can differ from $N$ only by copies of the trivial $G^{(1)}$-representation which
live in degrees $\geq p$. Since $(T(j)\otimes_k T(j)(-j))^{G_1}$
occurs in $N$ as $T(j)(-j)$ appears in $S/S^p_{>0}S$ by
\cite[\eqref{eq:roleofsigns}]{FFRT1}, $N$ contains a copy of trivial $G^{(1)}$-representation (see
\eqref{eq:L(j)tnzL(j)}) in degree $j\leq p-2<p$, which then also needs
to belong to $T^{\overline{\Fr}}$. Thus, $T^{\overline{\Fr}}$ and $N$
coincide up to nonzero multiplicities.

Moreover, $N$ is a tilting $G^{(1)}$-module (c.f. \S\ref{subsec:compT} in loc. cit.), which is a direct sum of  $T(i)^{\Fr}$ for $i\leq \lfloor(n(p-1)+j-2p+2)/p\rfloor=m$ by \cite[Proposition \ref{prop:G1inv}]{FFRT1} as the  weights  of $T(j)\otimes_k S/S^p_{>0}S$ are $\leq n(p-1)+j$. This establishes at least the upper bound on the tilt-free summands in \eqref{eq:nontiltfreecase}\eqref{eq:tiltfreecase}.

\medskip We need to see which $T(l)^{\Fr}$ for $0\leq l\leq m$
actually appear in $T^{\overline{\Fr}}$ and hence in 
\eqref{eq:nontiltfreecase}\eqref{eq:tiltfreecase}. 
By the previous paragraph we know that $N$ and $T^{\overline{\Fr}}$ contain the same indecomposable tilting representations. 
Therefore we now study the $T(l)^{\Fr}$ for $0\leq l\leq m$ that appear  
in\footnote{Strictly speaking it would not be necessary to include the case $q=0$ and $j\le p-2$ 
as it has already been dealt with in the previous paragraph.} $N$.

First we claim that
$T(q)$ appears in $S/S^p_{>0}S$ for all $0\leq q\leq n(p-1)$. From
\cite[\eqref{eq:roleofsigns}]{FFRT1}, one sees that it suffices to
check that every $0\leq q\leq n(p-1)$ occurs as some $q_t$ for
$t \in \Mscr$. This can be checked explicitly using \cite[Lemma
14.5]{FFRT1}.

Now let $q$ be as in the previous paragraph.  By \cite[Proposition \ref{cor:simple}]{FFRT1}, if we write $q$ in the form $q=q_1+pq_2$, where either $0 \leq q_1 \leq p-2$ and $q_2=0$ or $p-1 \leq q_1 \leq 2p-2$ and $q_2 \geq 0$, we can then decompose 
\begin{equation}
\label{eq:decomp}
T(j)\otimes_k T(q)=T(j_1)\otimes_k T(q_1)\otimes_k T(j_2)^{\Fr}\otimes_k T(q_2)^{\Fr}.
\end{equation} 
First assume $0\le j<  p-1$, in particular $j_2=0$. To establish \eqref{eq:nontiltfreecase} 
we need to explicitly produce the summands $T(l)^{\Fr}$, for all $0 \leq l \leq m$. We will do this by making specific choices for $q$. 
Set $q_1=2p-2-j_1$ ($> p-1$ since $j_1=j \in [0,p-2]$), so \eqref{eq:decomp} becomes 
\begin{equation}
\label{eq:tensor-j}
T(j) \otimes_k T(q)=T(j_1) \otimes_k T(2p-2-j_1) \otimes_k T(0)^{\Fr} \otimes_k T(q_2)^{\Fr}.
\end{equation}
As the highest weight of $T(j_1) \otimes_k T(2p-2-j_1)$ is $2p-2$, \eqref{eq:tensor-j} contains in particular $T(2p-2) \otimes_k T(q_2)^{\Fr}$ as a summand. We hence see that $N$ contains the summand
\begin{equation}\label{eq:59}
(T(2p-2) \otimes_k T(q_2)^{\Fr})^{G_1}=T(2p-2)^{G_1} \otimes_k T(q_2)^{\Fr}=T(q_2)^{\Fr},
\end{equation}
where we used \cite[Proposition \ref{prop:G1inv}]{FFRT1} for the last equality. Since $0\leq pq_2\leq \allowbreak n(p-1)-\allowbreak (2p-2-j)=(n-2)(p-1)+j$ we get all $T(l)^{\Fr}$ with $l\in [0,m]$ and hence \eqref{eq:nontiltfreecase}. 

Assume now
$j\geq p-1$. To establish \eqref{eq:tiltfreecase}, we need to explicitly produce the summands $T(l)^{\Fr}$, for all $j_2 \leq l \leq m$. We need to consider the two possible cases: $j_1=p-1$ and $j_1 >p-1$. If $j_1=p-1$, then set $q_1=p-1$. By a similar reasoning as above, we see that $T(j) \otimes_k T(q)$ contains $T(2p-2)\otimes_k T(q_2)^{\Fr} \otimes_k T(j_2)^{\Fr}$ as a summand, so $N$ contains $T(q_2)^{\Fr}\otimes_k T(j_2)^{\Fr}$ as a summand. Varying $q_2$, we therefore get all $T(l)^{\Fr}$ for $l\in [j_2,m]$. In the case $j_1>p-1$, by setting $q_1=2p-2-j_1$ and $q_2=0$ we similarly find $T(j_2)^{\Fr}$ as a summand. To obtain the other summands, setting $q_1=3p-2-j_1$, we see that $T(j) \otimes_k T(q)$ contains 
\begin{equation}
T(3p-2) \otimes_k T(j_2)^{\Fr} \otimes_k T(q_2)^{\Fr}=T(2p-2) \otimes_k T(1)^{\Fr} \otimes_k T(j_2)^{\Fr} \otimes_k T(q_2)^{\Fr},
\end{equation}
where we used \cite[Proposition \ref{cor:simple}]{FFRT1}. 
Therefore $N$ contains $T(1+j_2+q_2)^{\Fr}$ as a summand. Varying $q_2$, we also find all $T(l)^{\Fr}$ for $l\in [j_2+1,m]$.

\subsection{Examples and limiting behaviour}
\label{ex:notinterval} 
  Now we give examples of $j\ge p-1$ so that in Proposition
  \ref{prop:mocdecomposition}(2)  not all the potential summands
  $T(l)^{\Fr}\otimes_k S^p$ with $0\leq l<j_2$ appear with nonzero
  multiplicity in $(T(j)\otimes_k S)^{G_1}$.
We also give examples of $j\ge p-1$  where the set of summands is not even an interval. 

\medskip

To recapitulate: by the proof of Proposition \ref{prop:mocdecomposition}
the summands $T(l)^{\Fr}$ appearing in \eqref{eq:tiltfreecase} 
are those that appear in one of the representations
\[
(T(j)\otimes_k T(q))^{G_1}\allowbreak=\allowbreak (T(j_1)\otimes_k T(q_1))^{G_1}\otimes_k\allowbreak
(T(j_2)\otimes_k T(q_2))^{\Fr}
\]
for $0\le q\le n(p-1)$.
This is a purely combinatorial problem using \cite[Propositions \ref{prop:G1inv},
\ref{cor:simple}]{FFRT1}, although it seems not an easy one. Therefore we will only consider specific cases.

\medskip

A simple case where it is obvious that not all $0\leq l< j_2$ can appear with nonzero multiplicity  
is when $T(j_2)$ is $G_r$-projective (which is equivalent to $j_2\geq p^r-1$ by \cite[Lemma E.8]{jantzen2007representations}). Then every summand $T(l)$ of 
$((T(j_1)\otimes_k T(q_1))^{G_1})^{\Fr^{-1}}\otimes_k T(j_2)\otimes_k\allowbreak T(q_2)$ is $G_r$-projective as well and hence  $l\geq p^r-1$.

\medskip

We now give an example where the $l$'s that occur do not form an interval. Assume $p\geq n$ and 
let  $j=p-1+(3p-2)p$. 
In that case  $(T(j_1)\otimes_k T(q_1))^{G_1}=(T(p-1)\allowbreak\otimes_k T(q_1))^{G_1}$ is either $0$ or $k$ (up to multiplicities) (using \cite[Proposition \ref{prop:G1inv}]{FFRT1} together with $q_1\leq 2p-2$, and thus $p-1+q_1\leq 3p-3<p+(2p-2)$). 
If $q_1<p-1$ it is $0$ and if $q_1=p-1$ it is $k$. Since other $q_1$ for which we get $k$ would impose a more restrictive upper bound on the possible choice of $q_2$, we may assume $q_1:=p-1$.   
In that case the choices for $q_2$ are $0\leq q_2\leq c:=\lfloor ((n-1)(p-1))/p\rfloor\leq p-2$.  
We have $j_2=3p-2=2p-2+p$ and thus  $T(j_2)= T(2p-2)\otimes_k T(1)^{\Fr}$. 
It follows that 
\begin{equation}\label{eq:lde}
T(j_2)\otimes_k T(q_2)=T(2p-2)\otimes_k T(q_2)\otimes_k T(1)^{\Fr}.
\end{equation}
We claim that  $T(m)$ for $p\le 2p-2-c\leq m\leq 2p-2+c$ 
are all the 
 summands of $T(2p-2)\otimes_k T(q_2)$ for $0\leq q_2\leq c$. 
Assuming the claim, we split the interval of valid $m$'s as follows:
\[
[2p-2-c,2p-2+c]=[2p-2-c,2p-2]\cup [2p-1,2p-2+c].
\]
For $m\in [2p-2-c,2p-2]$ we use that $T(m)\otimes T(1)^{\Fr}=T(m+p)$. For $m\in  [2p-1,2p-2+c]$ we use
\begin{align*}
T(m)\otimes_k T(1)^{\Fr}&=T(m-p)\otimes_k T(1)^{\Fr}\otimes_k T(1)^{\Fr}\\
&=T(m-p)\otimes_k(T(0)^{\Fr}\oplus T(2)^{\Fr})\\
&=T(m-p)\oplus T(m+p).
\end{align*}
Combining the cases
we find that the summands of \eqref{eq:lde} 
correspond to the natural numbers in the  intervals  $[3p-2-c,3p-2]$, $[p-1,p-2+c]$, $[3p-1,3p-2+c]$, or equivalently in $[p-1,p-2+c]$ and $[3p-2-c,3p-2+c]$.  These intervals are neither consecutive nor overlapping
as (using that $c\le p-2$): $(p-2+c)+1\le 2p-3<2p-1\le \allowbreak(3p-2-c)-1$ (i.e.\ at least the weights $2p-3,2p-2,2p-1$ occur between the intervals). 

\medskip

It remains to prove the claim. Let us use $<_\oplus$ to denote the direct summand relation.  Note that as $T(1)^{\otimes q_2}$ is a direct sum of $T(q')$ for $q'\le q_2$ we have
\begin{multline}
\label{eq:reduction}
\{m\mid T(m)<_\oplus T(2p-2)\otimes_k T(q_2), 0\leq q_2\leq c\}=\\\{m\mid T(m)<_\oplus T(2p-2)\otimes_k T(1)^{\otimes q_2}, 0\leq q_2\leq c\}.
\end{multline}
The claim follows by induction on $c$ on the right hand side of \eqref{eq:reduction}
(as $c\leq p-2$) by the following identities which we verify below:
\begin{equation}
\label{eq:tiltingpieri}
T(a)\otimes_k T(1)=
\begin{cases}
T(a+1)\oplus T(a-1) & \text{if $p< a\leq 3p-3$, $a\neq 2p-1,2p$},\\
T(a+1)\oplus T(a-1)^{\oplus 2} & \text{if $a=p,2p$},\\
T(a+1) &\text{if $a=p-1,2p-1$}.
\end{cases}
\end{equation}
Note that these formulas are for $a\in [p-1,3p-3]=[p-1,2p-2]\cup [2p-1,2p-3]$. We claim it is sufficient to establish \eqref{eq:tiltingpieri} for $a\in [p-1,2p-2]$. Assume we have done
this and  $a\in [2p-1,3p-3]$. Then we use
$T(a)=T(a')\otimes_k T(1)^{\Fr}$ with $a'=a-p\in [p-1,2p-3]\bigl(\subset [p-1,2p-2]\bigr)$. Then using \eqref{eq:tiltingpieri} for $T(a')$ we find an expression $T(a)\otimes_k T(1)=\oplus T(a'\pm 1)\otimes_k T(1)^{\Fr}$, which establishes (with attention to the special case $a'=p-1$) that $a'\pm 1\in [p-1,2p-2]$. Then $T(a'\pm 1)\otimes T(1)^{\Fr}=T(a'\pm 1+p)=\oplus T(a\pm 1)$ where $a\pm 1$ is as given by the formulas so that we also
obtain the claimed result for $a\in [2p-1,3p-3]$.

So now we assume $a\in [p-1,2p-2]$. This case can be done  for example by 
combining \cite[Lemma 1.1]{MR2143497} (see \cite[Proposition \ref{dotyhenke}]{FFRT1})  which shows that $T(p-1)=S^{p-1}V$ and $T(a)$ is an extension of $S^aV$ by $S^{2p-2-a}V$,
with Pieri's formulas, which  in characteristic $p$ are filtrations \cite{MR931171}, 
\[
0\to S^{b-1}V\to S^bV\otimes_k V\to S^{b+1}V\to 0.
\]
In this way we may compute $[T(a)\otimes_k T(1):S^cV]$ for all $c$ and match it with the right-hand side of \eqref{eq:tiltingpieri}.

By a more careful inspection (as indicated in Example \ref{ex:notinterval}) one may also derive a precise formula for the summands in Proposition \ref{prop:mocdecomposition}(2) which appear with nonzero multiplicity. However, the formula (as we see it now) depends on distinguishing several cases and the computations are rather technical, thus we omit it. In the next corollary we show that, despite its imprecision, Proposition \ref{prop:mocdecomposition} is sufficient to obtain a definite result after iteration.
\begin{corollary}\label{cor:ffrtnmoc}
For $r\gg 0$ we have
\[
\{l\mid T(l)^{\Fr^r}\otimes_k S^{p^r}\text{ is a summand of $(T(j)\otimes_k S)^{G_r}$}
\}=\begin{cases}
[0,n-3]&\text{if $j\leq n-3$,}\\
[0,n-2]&\text{if $j\geq n-2$.}
\end{cases}
\]
\end{corollary}

\begin{proof}
Proposition \ref{prop:mocdecomposition} tells us that $(T(j)\otimes_k S)^{G_1}$ contains summands of the form $T(l)^{\Fr}\otimes_k S^p$ for every $l$ in the interval
$[a(j),b(j)]$ where
\begin{align*}
a(j)=j_2&=\max\left(0,\left\lfloor \frac{j+1-p}{p}\right\rfloor\right)\\
b(j)=m&=n-2+\left\lfloor\frac{j-n+2}{p}\right\rfloor
\end{align*}
and moreover there are no summands for $l>b(j)$.
Iterating this we find that $(T(j)\otimes_k S)^{G_2}$ contains summands of the form $T(l)^{\Fr^2}\otimes_k S^{p^2}$ 
for every $l$ in the set
\[
\bigcup_{j'\in [a(j),b(j)]} [a(j'),b(j')]=[a(a(j)),b(b(j))]
\]
where the asserted equality follows from the fact that $a(j),b(j)$ are increasing functions of $j$ and that $a(j)$ increases for at most $1$ per turn, which also implies  
that there are no summands for $l>b(b(j))$.

Iterating we see that 
$(T(j)\otimes_k S)^{G_r}$ contains summands of the form $T(l)^{\Fr^r}\otimes_k S^{p^r}$ for every $l$ in the interval
$
[a^{(r)}(j),b^{(r)}(j)]
$ 
(where $(-)^{(r)}$ denotes $r$-fold composition), 
and there are no summands for $l>b^{(r)}(j)$.

Now it is easy to see that $a(j)\le j$, with equality if and only if $j=0$. Hence $a^{(r)}(j)=0$ for $r\gg 0$. We conclude
that for $r\gg 0$ the summands of $(T(j)\otimes_k S)^{G_r}$ of the form $T(l)^{\Fr^r}\otimes_k S^{p^r}$ 
correspond precisely to the elements of the interval $[0,b^{(r)}(j)]$.

It now remains to determine $b^{(r)}(j)$ for $r\gg 0$. I.e.\ we have to study the limiting behaviour of 
the (sadly rather trivial) discrete dynamical system
\begin{equation}
\label{eq:dynamical}
j\mapsto n-2+\left\lfloor\frac{j-n+2}{p}\right\rfloor.
\end{equation}
Putting $j=u+(n-2)$ this may also be written as
\[
\label{eq:dynamical2}
u\mapsto \left\lfloor\frac{u}{p}\right\rfloor
\]
which clearly has limiting values $-1,0$ depending on whether we start with $u<0$ or $u\ge 0$. Hence \eqref{eq:dynamical} has limiting values $n-3$, $n-2$ depending on whether we start 
with $j\le n-3$ or $j\ge n-2$.
\end{proof}

\begin{remark}
The appearance of the interval $[0,n-2]$ in Corollary \ref{cor:ffrtnmoc} is quite pleasing. Recall from \S\ref{subsec:equivalence} that the decomposition of $(G^{(1)},S^p)$-module $(T(j)\otimes_k S)^{G_1}$ yields a decomposition of the $R^p$-module $T\{j\}:=(T(j)\otimes_k S)^G$. 
By \cite[Proposition \ref{prop:cm}]{FFRT1}, $T\{j\}$ is not a Cohen-Macaulay $R$-module for $j\ge n-2$, and hence it is also not a Cohen-Macaulay module when viewed as an $R^{p}$-module. Therefore
$T\{j\}_{R^p}$  must have at least one indecomposable summand  which is not Cohen-Macaulay and by Proposition \ref{prop:mocdecomposition} this summand must be of the form $T\{i\}^{\Fr}$ with $i\ge n-2$
(the other listed summands are Cohen-Macaulay).
By iteration we find that $T\{j\}_{R^{p^r}}$ must also have an indecomposable summand of the form $T\{i\}^{\Fr^r}$ with $i\ge n-2$.
The interval $[0,n-2]$ we obtain for the  summands as $r\r \infty$ is the smallest interval compatible with this constraint. Note also that the real interval $[0,n-2]$ is the closure of the strongly critical region \cite{van1991cohen} (see also \cite[Definition A.3]{FFRT1}). This suggests an obvious generalization
of our result for other pairs $(G,W)$.
\end{remark}

\section{Formality computations}\label{sec:formality}
In this section we embark on  decomposing 
$K_{jk}^{G_1}$, which will occupy most of the rest of the paper. 
Besides our standing assumption $p\geq \assume$ we now fix $1\leq j\leq n-3$. Recall  that in this case \eqref{biresolution01} is a tilt-free resolution of $M_j$ (see Remark \ref{rem:DtoSinresM}). 
We denote the complex \eqref{biresolution11} by $\tCj$ (with the last term $(j\omega)\otimes S$ 
in degree $0$). 
Unless otherwise specified we assume below $1\leq k\leq n-2$ and we set
\[
\Cjk=(\sigma_{\geq -k} \tCj[-k]),
\]  
where $\sigma_{\geq -k}$ denotes the stupid truncation.
The main goal of this section is to show 
\begin{proposition}\label{prop:outform}
The complex $\tau_{\geq 1} \RHom_{B_1}(k,\Cjk)$ is formal in the derived category of $(B^{(1)},S^p)$-modules. Moreover, its cohomology is a sum of characters.
\end{proposition}
In \S\ref{sec:ind-B-G}, we will then use induction and  Proposition \ref{prop:outform} to show that the complex $\tau_{\geq 1} \RHom_{G_1}(k,\gCjk)$, with $\gCjk$ a similar truncation of
\eqref{biresolution01}, is formal in $D(G^{(1)},S^p)$. This will finally allow us to apply Theorem \ref{prop:main}, see Proposition \ref{prop:Kjkdecomposition}.

The proof of Proposition \ref{prop:outform} proceeds in the following steps:
\begin{enumerate}
\item In Lemma \ref{lem:qired}, we decompose $\tau_{\geq 1} \RHom_{B_1}(k,\Cjk)$ into a direct sum of certain complexes $\tau_{\ge 1}\RHom_{B_1}(k,\Ljkt)\otimes_k S^p_+$ in $D(B^{(1)},S^p)$.  Hence, to prove Proposition \ref{prop:outform}, it suffices to prove formality of $\tau_{\ge 1}\RHom_{B_1}(k,\Ljkt)$ in $D(B^{(1)})$.
\item\label{step-3} Instead of doing this directly, in Proposition \ref{prop:Tformal}, we first prove that the complex $\ulRHom_{B_1}(k,\Ljkt)$ is formal in $D(B^{(1)})$.
\item In Proposition \ref{prop:formality}, we then use the formality from Step \eqref{step-3} to show $\tau_{\ge 1}\RHom_{B_1}(k,\Ljkt)$ is formal in in $D(B^{(1)})$, hence establishing Proposition \ref{prop:outform}.
\end{enumerate}

\subsection{Quasi-isomorphic reduction}\label{sec:qired}

In this subsection we prove that one can reduce the formality computation to more manageable complexes. 

Recall the following notation from \cite[\S \ref{sec:computingRHom}]{FFRT1}. Let $P=\Sym(V)$. Then $S=P^{\otimes n}$. We write $P=k[x,y]$ where the weights of
$x,y$ are $+\omega, -\omega$. 
It will be convenient to consider the simple representation $L(i)=S^iV$ for $i=0,\ldots,p-1$
to be embedded in $P$ as the part of degree~$i$.

With the convention $L(-1)=0$ and $t=(t_i)_i\in [0,p-1]^n$, let\footnote{To lighten the notation we will often drop ${ }^{\bullet}$ from the notation for  complexes.} 
\[\Ljt:=\left(\bigotimes_{i=1}^n \biggl(L(t_i-1)\otimes_k (-\omega)\xrightarrow{y} L(t_i)\biggr)\right)\otimes_k (j\omega)\]
 (with $\otimes_{i=1}^n L(t_i) \otimes_k (j\omega)$ in degree zero)  and denote  
 \begin{align*}
 \Ljkt&=(\sigma_{\geq -k}\Ljt)[-k],\\ 
 \Djkt&=(\sigma_{< -k}\Ljt)[-k-1].
 \end{align*}
  For $t\in [0,p-1]^{n}$, we denote $d_t=\sum_i t_i$.

We will prove the following lemma. 
\begin{lemma}\label{lem:qired}
There is  a quasi-isomorphism in the derived category $D(B^{(1)},\Gr(S^p))$ (of graded $(B^{(1)},S^p)$-modules)
\[
\tau_{\ge 1}\RHom_{B_1}(k,\Cjk)\cong \bigoplus_{t\in [0,p-1]^n} \tau_{\ge 1}\RHom_{B_1}(k,\Ljkt)\otimes_k S^p_+(-d_t).
\]
\end{lemma}

First recall the exact sequence \cite[(10.6)]{FFRT1} of graded $(B,P^p)$-modules:
\begin{equation*}
\label{finv1for1}
0\r \bigoplus_{i=0}^{p-1} L(p-1) \otimes_k  ky^i\otimes_k P^p(-i-(p-1))\r P\r \bigoplus_{i=0}^{p-2} L(i)\otimes_k P^p/y^p 
P^p(-i)\r 0
\end{equation*}
and notice its slight modification
\begin{equation*}
\label{finv1for2}
0\r \bigoplus_{i=1}^{p} L(p-1) \otimes_k  ky^i\otimes_k P^p(-i-(p-1))\r P\r \bigoplus_{i=0}^{p-1} L(i)\otimes_k P^p/y^p 
P^p(-i)\r 0.
\end{equation*}
Let us write these as
\[
0\r Q\r P\r \bar{P}\r 0,
\]
and
\[
0\r Q'\r P\r \bar{P}'\r 0.
\]
They fit together in a commutative diagram
\begin{equation}
\label{eq:together}
\xymatrix{
&&0\ar[d] &0\ar[d]\\
0\ar[r]& Q(-1)\otimes_k(-\omega)\ar[r]\ar[d]^\cong_y & P(-1)\otimes_k(-\omega)\ar[r]\ar[d]_y&\bar{P}(-1)\otimes_k(-\omega)\ar[r]\ar[d]^y & 0\\
0\ar[r]&Q'\ar[r]&P\ar[d]\ar[r]&\bar{P}'\ar[d]\ar[r]&0\\
&&P/yP\ar[r]_\cong\ar[d]&\bar{P}'/y\bar{P}\ar[d]\\
&&0&0
}
\end{equation}
In particular we obtain a $B$-equivariant quasi-isomorphism of complexes of graded $(B,S^p)$-modules
\[
(P(-1)\otimes_k (-\omega) \xrightarrow{y} P)\r (\bar{P}(-1)\otimes_k (-\omega)\xrightarrow{y}\bar{P}').
\]

Note that the complex
\[
\bar{P}(-1)\otimes_k (-\omega)\xrightarrow{y}\bar{P}'
\]
is equal to the sum of complexes
\[
 \bigoplus_{j=0}^{p-1}\biggl(L(j-1)\otimes_k (-\omega)(-j)\xrightarrow{y} L(j)(-j)\biggr)\otimes_k P^p/y^pP^p
\]
still with the convention $L(-1)=0$.

We now replace $P$ by $P_i=k[x_i,y_i]$ and define ${\bC}$ to be the tensor product of
$P_i(-1) \otimes_k (-\omega)\xrightarrow{y_i}P_i$. Similarly, define ${\bar{\bC}}$ to be the tensor product of the complexes
$\bar{P}_i(-1) \otimes_k (-\omega)\xrightarrow{y_i}\bar{P}'_i$. Note that
\begin{equation}
\label{eq:Cvar}
\bar{C}\otimes_k (j\omega) =\bigoplus_{t\in [0,p-1]^n} C^{(t)}_j(-d_t)\otimes_k S^p_+.
\end{equation}
Then by the above discussion we have quasi-isomorphisms
of complexes of graded $(B,S^p)$-modules
\begin{equation}
\label{eq:quasi}
\xymatrix{
{\bC}\ar@{->>}[r]\ar@{->>}[d]&{\bar{\bC}}\\
S_+
}
\end{equation}
and the kernel of the horizontal map consists of projective $B_1$-modules.

Now let as before
\begin{align*}
\Cjk&=(\sigma_{\ge -k} {\tCj})[-k],\\
&=(\sigma_{\ge -k} {\bC})[-k]\otimes_k (j\omega),
\end{align*}
and also put
\begin{align*}
\Djk&=(\sigma_{\le -k-1} {\bC})[-k-1]\otimes_k (j\omega),
\end{align*}
and use similar definitions for $\oCjk,\oDjk$, starting from ${\bar{\bC}}$. Using \eqref{eq:Cvar} we 
see 
\begin{align}
\label{eq:Cvar1}
\oCjk&=\bigoplus_{t\in [0,p-1]^n}C_{jk}^{(t)}(-d_t)\otimes_k S^p_+,\\\label{eq:Djkt}
\oDjk&=\bigoplus_{t\in [0,p-1]^n}D_{jk}^{(t)}(-d_t)\otimes_k S^p_+.
\end{align}

We have a morphism between distinguished triangles of complexes of graded $(B,S^p)$-modules
\begin{equation}
\label{keydiag}
\xymatrix{
\Djk\ar[r]\ar@{->>}[d]& \Cjk\ar[r]\ar@{->>}[d] &{\bC}[-k]\otimes_k(j\omega)\ar[r]\ar@{->>}[d]^{\cong}_{\eqref{eq:quasi}}& \\
\oDjk\ar[r]& \oCjk\ar[r]& {\bar{\bC}}[-k]\otimes_k(j\omega)\ar[r]&
}
\end{equation}
Using the lowest horizontal arrow in diagram \eqref{eq:together}
this yields a morphism of distinguished triangles in $D(B,\Gr(S^p))$
\begin{equation}
\label{keydiag1}
\xymatrix{
\Djk\ar[r]\ar@{->>}[d]& \Cjk\ar[r]\ar@{->>}[d] &S_+[-k]\otimes_k(j\omega)\ar[r]\ar@{=}[d]& \\
\oDjk\ar[r]& \oCjk\ar[r]& S_+[-k]\otimes_k(j\omega)\ar[r]&
}
\end{equation}
where the kernels of the columns consist of projective $B_1$-modules.

In particular,
\begin{align}\label{eq:tDjkbar}
\tH^i(B_1,\Djk)&\cong \tH^i(B_1,\oDjk)\qquad \text{for $i\in \ZZ$},\\\label{eq:tCjkbar}
\tH^i(B_1,\Cjk)&\cong \tH^i(B_1,\oCjk)\qquad \text{for $i\in\ZZ$}.
\end{align}

Moreover, if $P$ is the kernel of the leftmost vertical map then $H^i(B_1,P)=0$ for $i>0$. Hence
\begin{equation}\label{eq:Djkbar}
H^i(B_1,\Djk)\cong H^i(B_1,\oDjk)\qquad \text{for $i>0$}
\end{equation}
and hence by \eqref{keydiag1} and the 5-lemma 
\begin{equation}\label{eq:Cjkbar}
H^i(B_1,\Cjk)\cong H^i(B_1,\oCjk)\qquad \text{for $i>0$}
\end{equation}
or equivalently
\[
\tau_{\ge 1} \RHom_{B_1}(k,\Cjk)\cong \tau_{\ge 1}\RHom_{B_1}(k,\oCjk).
\]
Taking into account \eqref{eq:Cvar1} we have thus proved Lemma \ref{lem:qired}.

\subsection{Formality of $\ulRHom_{B_1}(k,\Ljkt)$}\label{sec:tateformality}
In order to prove the formality of the complexes $\tau_{\ge 1} \RHom_{B_1}(k,\Ljkt)$ (which we do in Proposition \ref{prop:formality}), we first show the formality of the related and more easily understandable complexes $\ulRHom_{B_1}(k,\Ljkt)$.  

\begin{proposition}\label{prop:Tformal}
The complex $\ulRHom_{B_1}(k,\Ljkt)$ is formal in $D(B^{(1)})$. Moreover  $\tH^l(B_1,\Ljkt)$ is a sum of characters for all $l$.
\end{proposition}
The proof will be split in a sequence of lemmas. The formality statement is proved in \S\ref{subsec:form}. The fact that $\tH^l(B_1,\Ljkt)$ is a sum of characters is a consequence of \eqref{eq:tHC} and Lemma \ref{lem:injection} below.
\subsubsection{Complex $\Ljt$}
Remember that for $t\in [0,p-1]^{n}$ we set $d_t=\sum_i t_i$. Note that we have an obvious quasi-isomorphism of complexes of
$B$-representations 
\begin{equation}
\label{eq:qisoL}
\Ljt\to ((j+d_t)\omega).
\end{equation}
We set
\[
q^0_{jt}=
\begin{cases}
(j+d_t)/p&\text{if $j+d_t\equiv 0\,(p)$},\\
\text{undefined}&\text{otherwise,}
\end{cases}
\]
\[
q^1_{jt}=
\begin{cases}
(j+d_t-p+2)/p &\text{if $j+d_t-p+2\equiv 0\,(p)$},\\
\text{undefined}&\text{otherwise.}
\end{cases}
\]
\begin{lemma} \label{rem:coll}
The numbers $q^0_{jt}$ and $q^1_{jt}$ cannot both be defined. Moreover, 
\begin{equation}
q_{jt}^0 \in
\begin{cases}
[1,n-1] & \text{ if } (j,p) \neq (1,n-2), \\ 
[1,n-2] & \text{ if } (j,p)=(1,n-2)
\end{cases}
\end{equation}
and 
\begin{equation}
q_{jt}^1 \in
\begin{cases}
[0,n-2] & \text{ if } (j,p) \neq (n-3,n-2), \\ 
[1,n-2] & \text{ if } (j,p)=(n-3,n-2)
\end{cases}
\end{equation}
and all possible integer values occur when varying $t$.
 \end{lemma}
 \begin{proof}
Since $p\ge 3$ by our standing hypothesis,  $q^0_{jt}$ and $q^1_{jt}$ cannot both be defined. Note that $d_t+j$ can attain any natural number in $[j,n(p-1)+j]$. 
 Since  $p\geq \max\{n-2,3\}$ and $0<j<n-2$ we obtain $0<j<p$ and also 
$-p<j-p+2\leq 0$ if $(j,p)\neq (n-3,n-2)$, 
  moreover $(n-2)p+p-2\leq n(p-1)+j<(n-1)p+p-2$ 
  and also $(n-1)p\leq n(p-1)+j\allowbreak < np$ if $(j,p)\neq (1,n-2)$.
  
Hence $q_{jt}^0$ can attain any natural number in $[1,n-1]$ if $(j,p)\neq (1,n-2)$ and in $[1,n-2]$ if $(j,p)=(1,n-2)$ (the latter upper bound by direct checking), and $q_{jt}^1$ can attain any natural number in $[0,n-2]$ if $(j,p)\neq (n-3,n-2)$ and in $[1,n-2]$ if $(j,p)=(n-3,n-2)$ (the latter lower bound by direct checking).
 \end{proof}

We obtain the following lemma as a consequence of Lemma \ref{B1cohomology02}. Note that we use the symbol $q^{i(2)}_{jt}$ as follows
\begin{equation*}
q^{i(2)}_{jt}=
\begin{cases}
q^{0}_{jt} & i \equiv 0\,(2), \\
q^{1}_{jt} & i \equiv 1\,(2).
\end{cases}
\end{equation*}
\begin{lemma}\label{lem:qiCM}
The quasi-isomorphism $\Ljt\to ((j+d_t)\omega)$ induces an isomorphism in $D(B^{(1)})$ 
\[\ulRHom_{B_1}(k,\Ljt)\cong \bigoplus_i\hat{H}^i(B_1,((j+d_t)\omega))[-i],\]
where moreover
\begin{equation}
\label{eq:qiCM}
\hat{H}^i(B_1,((j+d_t)\omega))=
\begin{cases}
((q^{i(2)}_{jt}+i)p\omega)&\text{if $q^{i(2)}_{jt}$ is defined,}\\
0&\text{otherwise.}
\end{cases}
\end{equation}
\end{lemma}

\subsubsection{A spectral sequence and some cohomology estimates}
We now describe the spectral sequence associated to the double complex $\Hom_{B_1}(C(k),\Ljt)$, for some complete $B_1$-resolution $C(k)$ of $k$ (see \cite[Appendix \ref{sec:appA}]{FFRT1}).
Hence we put
\[
E_0^{uv}=\Hom_{B_1}(C(k)^{-v},(\Ljt)^u)
\]
where $u$ is the column index. Thus the differentials are
$d_0^{uv}:E_0^{uv} \to E_0^{u,v+1}$ and
$d_1^{uv}:E_1^{uv} \to E_1^{u+1,v}$. To emphasise this, we will say
this is the $\uparrow\rightarrow$ spectral sequence associated to
$\Hom_{B_1}(C(k),\Ljt)$. Note that we have
\[
E_1^{uv}=\tH^v(B_1,(\Ljt)^u).
\]
The following lemma gives the structure of the $E_1$-page of this
spectral sequence.
\begin{lemma}\label{lem:tatecohcom}
The complex $\ulRHom_{B_1}(k,(\Ljt)^{u})$ is formal in $D(B^{(1)})$ and 
up to multiplicity (possibly zero) $\tH^l(B_1,(\Ljt)^{u})$ equals
\begin{equation}
\label{eq:mainformula}
\begin{cases}
(lp\omega)&\text{if $u\geq -j$,}\\
0&\text{if $u=-j-1$,}\\
((l-1)p\omega)&\text{if $u\le-j-2$}
\end{cases}
\end{equation}
as $B^{(1)}$-representation.
\end{lemma}

\begin{proof}
We may assume $-n\le u\le 0$ for otherwise $(C_j^{(t)})^u=0$. 
Recall that 
 $1\le j\le n-3$ and here we will be using the hypothesis  $p\geq \assume\geq n-2$. By \cite[Proposition \ref{cor:tiltingkernel}, Lemma \ref{rem:tnz}]{FFRT1}, the $G_1$-nonprojective summands of $(\Ljt)^{u}$  are equal to $L(i)\otimes_k((j+u)\omega)$ for $0\leq i\leq p-2$. Hence $\ulRHom_{B_1}(k,(\Ljt)^{u})$ is formal by Lemma \ref{B1cohomology02}. 
We compute 
\[
-p-1\leq -n+1\leq -n+j\leq u+j\leq j\leq n-3\le p-1,
\]
and so for $a:=j+u$ we have $a\in [-p-1,p-1]$  
 and, taking into account $0\le i\le p-2$,   $L(i)\otimes_k(a\omega)$ has nonzero Tate cohomology  only if $a$ equals $i ,p-2-i$ (if $a\ge 0$, and depending on the parity of $l$) or $i-p,-2-i$ (if $a\le -2$, and again depending on the parity of $l$) by \eqref{TateB1formulaA}. 
We now immediately check that \eqref{TateB1formulaA}
implies \eqref{eq:mainformula}.
\end{proof}
The following diagram summarizes the $E_1$-page of the $\u\r$-spectral sequence for 
$\Hom_{B_1}(C(k), \Ljt)$ 
\begin{small}
\begin{equation}
\label{eq:prelimE1}
\begin{array}{c|cccccccccccc}
(v)\\
\vdots&&\vdots&\vdots&&\vdots&\vdots&\vdots&&\vdots&\vdots&&\\
l+1&\cdots&0&l&\cdots&l&0&l+1&\cdots&l+1&0&\cdots&\\
l&\cdots&0&l-1&\cdots&l-1&0&l&\cdots&l&0&\cdots&\\
l-1&\cdots&0&l-2&\cdots&l-2&0&l-1&\cdots&l-1&0&\cdots&\\
\vdots&&\vdots&\vdots&&\vdots&\vdots&\vdots&&\vdots&\vdots&&\\
\hline
&&-n-1& -n &\cdots &-j-2&-j-1&-j&\cdots&0&1&\cdots&(u)
\end{array}
\end{equation}
\end{small}
where $r$ in the body stands for $(rp\omega)^{\oplus a}$ with $a\in \NN$.

\begin{corollary}
\label{lem:Ljktweights}
The weights of  $\tH^l(B_1,\Ljkt)$ for $l\in \ZZ$ are contained in the following interval
\begin{equation}
\label{eq:constraint1}
\begin{cases}
[l-k,l]p\omega &\text{if $k\leq j$, }\\
[l-k,l-1]p\omega &\text{if $k> j$. }
\end{cases}
\end{equation}
\end{corollary}
\begin{proof} The $E_1$ page of the $\u\r$-spectral
sequence for $\Hom_{B_1}(C(k),\Ljkt)$ 
is a horizontally truncated and shifted 
version of \eqref{eq:prelimE1}. More 
precisely, we have
\[
\tH^l(B_1,\Ljkt)=\tH^l(B_1,(\sigma_{\ge -k}\Ljt)[-k])=\tH^{l-k}(B_1,\sigma_{\ge -k}\Ljt)
\]
The weights of $\tH^{l-k}(B_1,\sigma_{\ge -k}\Ljt)$ occur in some 
$\tH^{l-k-u}(B_1,(\Ljt)^{u})$ for $-k\le u \allowbreak \le 0$. By Lemma  \ref{lem:tatecohcom} 
we see that the weights of $\tH^{l-k-u}(B_1,(\Ljt)^{u})$
satisfy on the nose $\ge (l-k-u-1)p\omega\geq (l-k-1)p\omega$. However the value
$(l-k-1)p\omega$ cannot actually be achieved since it corresponds to the case $u=0$ and
$u\le -j-2$ which is excluded by the hypotheses on $j$. So an actual lower bound is
$(l-k)p\omega$. 

The upper bound for the weights is on the nose  $\leq (l-k-u)p\omega\leq lp\omega$ which corresponds to $u=-k$
and $u\ge -j$. These conditions imply $k\le j$. So if $k>j$ an actual upper bound 
is $(l-1)p\omega$.
\end{proof}
\begin{remark}
See \S\ref{sec:extra2} for a more precise description of the cohomology of $\Ljkt$.
\end{remark}

Define 
\begin{equation}
\label{eq:Ldef}
\sLjkt=H^0(\Ljkt)=H^0(\Djkt) \qquad \text{for $1\leq k\leq
n-2$}.
\end{equation}
 Note that $\Djkt\to \sLjkt$ is quasi-isomorphism. 
\begin{corollary}\label{cor:Kweights}
The weights of $\tH^l(B_1,\sLjkt)$ for $l\in \ZZ$ are contained in the following interval 
\[
\begin{cases}
[l,l+n-k-2]p\omega &\text{if $k\leq j$},\\
[l-1,l+n-k-2]p\omega &\text{if $k>j$}.
\end{cases}
\]
\end{corollary}
\begin{proof}
Since $\tH^l(B_1,\sLjkt)\cong \tH^l(B_1,\Djkt)$ we can get  bounds for the weights by computing the $E_1$-page of the $\uparrow\rightarrow$ spectral sequence associated to the double complex $\Hom_{B_1}(C(k),\Djkt)$. In other words we need to compute $\tH^{l-u}(B_1,(\Djkt)^{u})$ for varying $u$.
Note that $(\Djkt)^u=(\Ljt)^{-k-1+u}$. Hence we may assume $-n+k+1\leq u\allowbreak\leq 0$ as the other terms are zero.

Hence from Lemma \ref{lem:tatecohcom} we obtain that  $\tH^{l-u}(B_1,(\Djkt)^{u})$ is up to multiplicities
\begin{equation}\label{eq:tHD}
\begin{cases}
((l-u)p\omega)&\text{if $-u+k+1\leq j$},\\
0&\text{if $-u+k=j$},\\
((l-u-1)p\omega)&\text{if  $-u+k-1\geq j$}.
\end{cases}
\end{equation}
Therefore the weights are always $\geq (l-u-1)p\omega\geq (l-1)p\omega$, The value  $(l-1)p\omega$ corresponds only to $u=0$ and $k>j$. 
The upper bound is $\leq (l-u)p\omega\leq (l+n-k-1)p\omega$, where the latter corresponds to the case $u=-n+k+1$, $-u+k+1\leq j$ which is impossible. Thus, the weights are  bounded above by $(l+n-k-2)p\omega$.
\end{proof}

\subsubsection{Cohomology exact sequences}
We proceed by analysing $\tH^*(B_1,\Ljkt)$. By the definition of $\Ljkt$
(see \S\ref{sec:qired},\eqref{eq:qisoL},\eqref{eq:Ldef}) we have
\[
H^i(\Ljkt)=
\begin{cases}
\sLjkt&\text{if $i=0$,}\\
((j+d_t)\omega)&\text{if $i=k$,}\\
0&\text{otherwise.}
\end{cases}
\]
Hence we have a distinguished triangle\footnote{The index shift caused by defining the third arrow as $\gamma[1]$  will be convenient later on.} in $D(B)$
\begin{equation}\label{secLCo}
\sLjkt\xrightarrow{\alpha} \Ljkt\xrightarrow{\beta} ((j+d_t)\omega)[-k]\xrightarrow{\gamma[1]}\,
\end{equation}
and we note that since $\Ljkt$ as a complex is concentrated in degrees $[0,k]$ we may take for $\alpha$, $\beta$ the obvious maps of complexes.

Applying $\ulRHom_{B_1}(k,-)$ to \eqref{secLCo} we obtain a distinguished triangle in $D(B^{(1)})$ 
\begin{equation}\label{eq:tri}
\ulRHom_{B_1}(k,\sLjkt)\xrightarrow{\alpha} \ulRHom_{B_1}(k,\Ljkt)\xrightarrow{\beta} \ulRHom_{B_1}(k,((j+d_t)\omega)[-k])
\xrightarrow{\gamma[1]}
\end{equation}
whose associated cohomology long exact sequence looks as follows: 
\begin{equation}
\label{eq:les}
\begin{tikzcd}
& & \cdots \arrow[overlay,out=-30, in=150]{dll}{\alpha_{l-1}}\\
\tH^{l-1}(B_1,\Ljkt) \rar{\beta_{l-1}} & \tH^{l-k-1}(B_1,((j+d_t)\omega)) \rar{\gamma_{l}} & \tH^{l}(B_1,\sLjkt) \arrow[overlay,out=-30, in=150]{dll}{\alpha_l}\\
\tH^l(B_1,\Ljkt) \rar{\beta_l} & \tH^{l-k}(B_1,((j+d_t)\omega)) \rar{\gamma_{l+1}} & \tH^{l+1}(B_1,\sLjkt) \arrow[overlay,out=-30, in=150]{dll}{\alpha_{l+1}} \\
\tH^{l+1}(B_1,\Ljkt) \rar{\beta_{l+1}} & \tH^{l-k+1}(B_1,((j+d_t)\omega)) \rar{\gamma_{l+2}} & \tH^{l+2}(B_1,\sLjkt) \arrow[overlay,out=-30, in=150]{dll}{\alpha_{l+2}} \\
\cdots & & 
\end{tikzcd}
\end{equation}
Careful inspection of this long exact sequence yields the following result:
\begin{lemma}\label{lem:splitH}
There are three mutually exclusive cases:
\begin{enumerate}
\item \label{gamma_l=0,beta_l=0}
For all $l$ one has $\beta_l=0$, $\gamma_l=0$. In this case there are isomorphisms
\begin{equation}
\label{eq:gamma_l=0,beta_l=0}
\tH^l(B_1,\sLjkt)\overset{\alpha_l}{\cong}\tH^l(B_1,\Ljkt).
\end{equation}
\item\label{gamma_l=0}
For all $l$ one has $\gamma_l=0$ and there is some $\beta_l\neq 0$. In this case
there are short exact sequences 
\begin{equation}\label{eq:sec1}
0\xrightarrow{} \tH^l(B_1,\sLjkt)\xrightarrow{\alpha_l} \tH^l(B_1,\Ljkt)\xrightarrow{\beta_l} \tH^{l-k}(B_1,((j+d_t)\omega))\xrightarrow{} 0.
\end{equation}
\item\label{gamma_l!=0}
For all $l$  one has $\beta_l=0$ and there is some $\gamma_l\neq 0$. In this case
there are short exact sequences
\begin{equation}\label{eq:sec2a}
0\xrightarrow{} \tH^{l-k-1}(B_1,((j+d_t)\omega))\xrightarrow{\gamma_{l}}  \tH^{l}(B_1,\sLjkt)\xrightarrow{\alpha_{l}} \tH^{l}(B_1,\Ljkt)\xrightarrow{} 0.
\end{equation}
\end{enumerate}
The case \eqref{gamma_l=0,beta_l=0} occurs if and only if $q^0_{jt}$, $q^1_{jt}$ are both undefined.
\end{lemma}
\begin{proof}
It is clear that the three cases are mutually exclusive and moreover that \eqref{eq:gamma_l=0,beta_l=0}\eqref{eq:sec1}\eqref{eq:sec2a} follow by 
\eqref{eq:les} from the indicated properties of
$(\beta_l)_l$, $(\gamma_l)_l$. We now claim that one of the cases always appears; i.e.
\begin{equation}
\label{eq:logic}
(\forall l:\gamma_l=0)\vee (\forall l:\beta_l=0).
\end{equation} 
For the rest of the proof we will use the following observations:
\begin{enumerate}
\item[(i)] \label{usei} For a given $l$, $\tH^l(B_1,((j+d_t)\omega))$ and $\tH^{l+1}(B_1,((j+d_t)\omega))$ cannot both be nonzero, because $p\neq 2$, $q_{jt}^{l(2)}$ being defined ensures $q_{jt}^{(l+1)(2)}$ is not defined (see Lemma \ref{rem:coll} and \eqref{eq:qiCM}). 

\item[(ii)] \label{useii}
For a given $l$, in \eqref{eq:les} either $\beta_l=0$ or $\gamma_{l+1}=0$, which follows immediately from 
Lemma \ref{lem:qiCM} which asserts that  $\tH^{l-k}(B_1,((j+d_t)\omega))$ is at most one-dimensional.
\item[(iii)] The long exact sequence \eqref{eq:les} is 2-periodic, up to tensoring with $(2p\omega)$. This follows  from Lemma \ref{lem:periodicity}
applied to the distinguished triangle \eqref{secLCo}.
\end{enumerate}
To show \eqref{eq:logic}, it suffices to show the equivalent statement
\[
\label{eq:morelogic}
\exists l: \gamma_l \neq 0 \Rightarrow \forall m:\beta_m=0.
\]
Assume $\gamma_{l} \neq 0$. By (ii) we get
$\beta_{l-1}=0$ and by \eqref{eq:les} we get $\tH^{l-k-1}(B_1,((j+d_t)\omega))\allowbreak \neq 0$. Thus by (i),
$\tH^{l-k}(B_1,((j+d_t)\omega))=0$, so also $\beta_{l}=0$, again by \eqref{eq:les}. Using (iii) 
this implies that all $\beta_m$ are zero. 

The last claim of the lemma is clear by  \eqref{eq:qiCM}.
\end{proof}

\begin{remark}\label{rmk:prermk}
A more refined version of Lemma \ref{lem:splitH} will be given in Proposition  \ref{rmk:leftright} (whose proof depends however on this lemma). More precisely we will show
that the cases \eqref{gamma_l=0} and \eqref{gamma_l!=0} are distinguished by whether or not $q^\ell_{jt}\le k-[k\le j]$ (see \eqref{eq:logical}) and furthermore we will also show 
that the exact sequences \eqref{eq:sec1}\eqref{eq:sec2a} are in fact split.

Note that the refined version is only used in the ``extra'' \S\ref{sec:extra} which aims to decide \emph{exactly} which summands in the decomposition of $K_{jk}^{G_1}$ appear with nonzero multiplicity.
\end{remark}

\subsubsection{Formality argument}\label{subsec:form}
To prove the formality of $\ulRHom_{B_1}(k,\Ljkt)$ (c.f. Proposition \ref{prop:Tformal}) it is sufficient to construct morphisms 
\begin{equation}\label{eq:map}
\ulRHom_{B_1}(k,\Ljkt)\r \tH^l(B_1,\Ljkt)[-l]
\end{equation}
 in the derived category of $B^{(1)}$-complexes which induce isomorphisms on $H^l(-)$.

 We already have morphisms (with isomorphisms following from Lemma \ref{lem:qiCM})
 \begin{multline}\label{eq:map1}
 \ulRHom_{B_1}(k,\Ljkt)\to \ulRHom_{B_1}(k,\Ljt[-k])\cong \ulRHom_{B_1}(k,((j+d_t)\omega))[-k]\\\cong \bigoplus_l \tH^l(B_1,((j+d_t)\omega))[-l-k]\xrightarrow{\text{projection}} \tH^{l-k}(B_1,((j+d_t)\omega))[-l],
 \end{multline}
which induces $\beta_l$ on $H^l(-)$.

Below we will  construct additional morphisms
\begin{equation}\label{eq:rmap}
\ulRHom_{B_1}(k,\Ljkt)\to \im\alpha_l[-l]
\end{equation}
so that the resulting composition
\[
\im \alpha_l\r \tH^l(B_1,\Ljkt)\r \im\alpha_l
\]
is an isomorphism. Using Lemma \ref{lem:splitH} this yields isomorphisms 
\begin{equation}\label{eq:tHC}
\tH^l(B_1,\Ljkt)\cong\allowbreak\im \alpha_l \oplus (\tH^{l-k}(B_1,((j+d_t)\omega)))^{\oplus m},
\end{equation}
 $m\in \{0,1\}$, and we define \eqref{eq:map} 
as the sum of \eqref{eq:map1} and \eqref{eq:rmap} and hence obtain the formality.

To carry this out we take a natural projection morphism $\pi:\Ljkt\to (\Ljkt)^{0}$ (resp. $\pi:\Ljkt\to (\Ljkt)^{0,1}:=(\Ljkt)^{0}\to (\Ljkt)^{1}$ in the case $k=j+1$). 
 The complex $\ulRHom_{B_1}(k,(\Ljkt)^0)$ is formal by Lemma \ref{lem:tatecohcom}, while (for $k=j+1$) 
 the complex $\ulRHom_{B_1}(k,(\Ljkt)^{0,1})$ is formal 
by Lemma \ref{lem:2form} below. 
 Thus we get a morphism 
 \begin{multline*}
 \theta:\ulRHom_{B_1}(k,\Ljkt)\to \tH^l(B_1,(\Ljkt)^{0})[-l]\\
  (\text{resp. } \theta:\ulRHom_{B_1}(k,\Ljkt)\to \tH^l(B_1,(\Ljkt)^{0,1})[-l]).
  \end{multline*}
 Note that $\tH^l(B_1,(\Ljkt)^{0})$ is a sum of characters (and thus a trivial $U^{(1)}$- representation) by Lemma \ref{lem:tatecohcom}
   and the same holds for $\tH^l(B_1,(\Ljkt)^{0,1})$ by Lemma \ref{lem:2form} below. 
 Thus, we can compose $\theta$ with a projection to a  $H^{(1)}$- (and thus $B^{(1)}$-) invariant subspace of  $\tH^l(B_1,(\Ljkt)^{0})[-l]$ (resp. $\tH^l(B_1,(\Ljkt)^{0,1})[-l]$). The following lemma enables us to view $\im \alpha_l[-l]$ as such a subspace, thus the corresponding composition provides us with the desired map \eqref{eq:rmap}.

\begin{lemma}\label{lem:injection}
We distinguish the following two cases.
\begin{enumerate}
\item If $k\neq j+1$, then the composition 
\begin{equation}
\label{eq:case1}
\im\alpha_l\r \tH^l(B_1,\Ljkt)\to \tH^l(B_1,(\Ljkt)^0),
\end{equation}
where the second map is induced from the projection map  $\Ljkt\to (\Ljkt)^0$, is an injection.
\item If $k= j+1$, then the composition 
\begin{equation}
\label{eq:case2}
\im\alpha_l\r \tH^l(B_1,\Ljkt)\to \tH^l(B_1,(\Ljkt)^{0,1}),
\end{equation}
where the second map is induced from the projection map  $\Ljkt\to (\Ljkt)^{0,1}$, is an injection.
\end{enumerate}

In particular, $\im \alpha_l$ is a sum of characters.
\end{lemma}
\begin{proof}
It follows from Corollary \ref{cor:Kweights} that $\tH^l(B_1,\sLjkt)$ has weights 
\begin{equation}
\label{eq:constraint}
\begin{cases}
\geq lp\omega &\text{if $k\leq j$},\\
\geq (l-1)p\omega &\text{if $k\geq j+1$}.\\
\end{cases}
\end{equation}

Let $0\neq a\in \im\alpha_l$ (whose weights are constrained by \eqref{eq:constraint}). We have to prove that $a$ has nonzero image in the right most groups in \eqref{eq:case1}\eqref{eq:case2}.
We use the short exact sequences of complexes
\[
0\r \sigma_{\ge q+1} \Ljkt \r \Ljkt \r \sigma_{\le q} \Ljkt\r 0,
\]
where $q=0,1$ depending on the case. If the image of $a$ in $\tH^l(B_1, \sigma_{\le q} \Ljkt)$ is zero then $a$ determines a nonzero element of $\tH^l(B_1,\sigma_{\ge q+1} \Ljkt)$  and
hence by the spectral sequence associated to $\Hom_{B_1}(C(k),\sigma_{\ge q+1} \Ljkt)$, a nonzero element in some $\tH^{l-u}(B_1,(\Ljkt)^{u})$ for $u\ge q+1\ge 1$. 
By Lemma \ref{lem:tatecohcom} the weights of $\tH^{l-u}(B_1,(\Ljkt)^{u})$ 
are given by (up to possibly zero multiplicity)
\begin{equation}
\label{eq:constraint2}
\begin{cases}
((l-u)p\omega) &\text{if $u\geq k-j$},\\
0&\text{if $u=k-j-1$},\\
((l-u-1)p\omega) &\text{if $u\le k-j-2$}.
\end{cases}
\end{equation}
The only overlap between \eqref{eq:constraint} and \eqref{eq:constraint2} taken into account $u\ge 1$ is $u=1$, $k=j+1$. Since $u\ge q+1$ this corresponds to $q=0$ which implies case (1).
However $k=j+1$ implies case (2) which is a contradiction.

Thus, $\im \alpha_l$ is a sum of characters, since, as discussed before the lemma, both $\tH^l(B_1,(\Ljkt)^0)$, $\tH^l(B_1,(\Ljkt)^{0,1})$ are sums of characters by respectively \eqref{TateB1formulaA}, Lemma \ref{lem:2form} below.
\end{proof}

We have used the following lemma for $l=-j-1$. The case $l=-j-2$ will only be used in Section \ref{sssec:form-2} below.

\begin{lemma}\label{lem:2form}
If $l=-j-1$ or $l=-j-2$ then 
 $\ulRHom_{B_1}(k,(\Ljt)^l\to (\Ljt)^{l+1})$ (the $(-)^l$ part in cohomological degree zero) is formal in $D(B^{(1)})$, and its cohomology is a sum of characters.
\end{lemma}

\begin{proof}
We write $\Cscr^0=(\Ljt)^l$, $\Cscr^1=(\Ljt)^{l+1}$. 
By Lemma \ref{lem:tatecohcom}, the Tate cohomology of $\Cscr^0$ (resp. $\Cscr^1$) is zero if $l=-j-1$ (resp. $l=-j-2$), thus the inclusion $0\to\Cscr^1$ to $\Cscr^0\to \Cscr^1$ (resp. projection from $\Cscr^0\to \Cscr^1$ to $\Cscr^0\to 0$) induces an isomorphism $\RHom_{B_1}(k,\Cscr^0\to \Cscr^1)\cong \RHom_{B_1}(k,\Cscr^1[-1])$ 
(resp. $\cong \RHom_{B_1}(k,\Cscr^0)$) in $D(B^{(1)})$. The latter complex is
 formal
 and the cohomology is a sum of characters by Lemma \ref{B1cohomology02}.
\end{proof}

\subsection{Formality of $\tau_{\geq 1}\RHom_{B_1}(k,\Ljkt)$}\label{sec:>0formality}
In this section we want to show  
\begin{proposition}\label{prop:formality}
The complex $\tau_{\geq 1}\RHom_{B_1}(k,\Ljkt)$ is formal in $D(B^{(1)})$. Moreover  $H^l(B_1,\Ljkt)$ is a sum of characters for  $l\ge 1$.
\end{proposition}

\begin{proof}
To show that $\tau_{\geq 1}\RHom_{B_1}(k,\Ljkt)$ is formal it suffices to construct for $l\geq 1$ maps
\begin{equation}
\label{eq:fullformality}
\RHom_{B_1}(k,\Ljkt)\to H^l(B_1,\Ljkt)[-l]
\end{equation}
in the derived category of $B^{(1)}$-modules, which induce isomorphisms on $H^l(-)$.

The short exact sequence \eqref{secLCo} induces a commutative diagram of long exact sequences
\begin{equation}
\label{eq:5-lemma}
\tiny
\xymatrix@C=1.12em{
\ar[r] &H^{l-k-1}(B_1,(q_{jt}\omega))\ar[r]\ar[d]&H^l(B_1,\sLjkt)\ar[r]\ar[d]_{\cong}& H^l(B_1,\Ljkt)\ar[r]\ar[d] & H^{l-k}(B_1,(q_{jt}\omega))\ar[r]\ar[d]& H^{l+1}(B_1,\sLjkt)\ar[r]\ar[d]_\cong
&\\
\ar[r]&
{\tH}^{l-k-1}(B_1,(q_{jt}\omega))\ar[r]_-{\gamma_l}&{\tH}^l(B_1,\sLjkt)\ar[r]_-{\alpha_l}& {\tH}^l(B_1,\Ljkt)\ar[r] & {\tH}^{l-k}(B_1,(q_{jt}\omega))\ar[r]
& {\tH}^{l+1}(B_1,\sLjkt)\ar[r]
&
}
\end{equation}
where we have written $q_{jt}:=j+d_t$ for brevity. 
Recall that we assume $k\ge 1$. 

We first deduce some preliminary information concerning the structure of \eqref{eq:5-lemma} for varying values of $l$.
\begin{enumerate}
\item
If $l> k$ or ($l=k$ and $\gamma_l=0$) then the $5$-lemma applies to \eqref{eq:5-lemma} by \eqref{eq:B1char} and we obtain 
\begin{equation}\label{eq:>l}
H^l(B_1,\Ljkt)=\tH^l(B_1,\Ljkt).
\end{equation}

\item
If $l< k$ or ($l=k$ and $\gamma_l \neq 0$) then we claim
\begin{equation}
\label{eq:isoregulartate}
H^l(B_1,\Ljkt)\cong H^l(B_1,\sLjkt).
\end{equation}
For $1 \leq l < k$ this is clear, as we have vanishing $H^{l-k-1}(B_1,((j+d_t)\omega))=\allowbreak H^{l-k}(B_1,((j+d_t)\omega))=0$. For $l=k$ and $\gamma_l \neq 0$, note that $\gamma_l\neq 0$ implies
$\hat{H}^{l-k-1}(B_1,((j+d_t)\omega))\neq 0$ and hence vanishing $0=\hat{H}^{l-k}(B_1,((j+d_t)\omega))\allowbreak =H^{l-k}(B_1,((j+d_t)\omega))$ by (i) in the proof of Lemma \ref{lem:splitH}. 
 Since we also have $H^{l-k-1}(B_1,((j+d_t)\omega))=\allowbreak H^{-1}(B_1,(j+d_t)\omega))=0$   we  get the isomorphism \eqref{eq:isoregulartate} also in this case.
\end{enumerate}

We will now give formality proofs for the three cases in Lemma \ref{lem:splitH}. 

\subsubsection{Formality in the cases (\ref{gamma_l=0,beta_l=0},\ref{gamma_l=0}) in Lemma \ref{lem:splitH}}
In this case $\gamma_l=0$ for all $l$.  
Thus $\alpha_l$ in \eqref{eq:5-lemma} 
is injective and this implies that $H^l(B_1,\Ljkt)\r \tH^l(B_1,\Ljkt)$ is injective.
Indeed if  $l\geq k$  this follows from \eqref{eq:>l}  and if $l<k$ this follows from \eqref{eq:isoregulartate} and some elementary diagram chasing in \eqref{eq:5-lemma}.

Since $\tH^l(B_1,\Ljkt)$ is a sum of characters by Proposition \ref{prop:Tformal} it follows that  $H^l(B_1,\Ljkt)$ is  $B^{(1)}$-equivariant direct summand of $\tH^l(B_1,\Ljkt)$.

 Thus  we can obtain \eqref{eq:fullformality} as a composition 
\begin{multline*}
\RHom_{B_1}(k,\Ljkt)\r \underline{\RHom}_{B_1}(k,\Ljkt)\cong \bigoplus_l\ {\tH}^l(B_1,\Ljkt)[-l]\\
\xrightarrow{\text{projection}}\tH^l(B_1,\Ljkt)[-l]\xrightarrow{\text{projection}} H^l(B_1,\Ljkt)[-l],
\end{multline*}
where we used Proposition \ref{prop:Tformal} for the isomorphism. 

\subsubsection{Formality in the case \eqref{gamma_l!=0} in Lemma \ref{lem:splitH}}
\label{sssec:form-2}
\begin{lemma}
\label{lem:upperlower}
 For $l\geq 1$, there is a commutative diagram with exact rows
\begin{equation}
\label{eq:3-lemma}
\xymatrix{
0\ar[r] &H^{l-k-1}(B_1,((j+d_t)\omega))\ar[r]\ar[d]&H^l(B_1,\sLjkt)\ar[r]\ar[d]^{\cong}& H^l(B_1,\Ljkt)\ar[r]\ar[d] & 0\\
0\ar[r]&
{\tH}^{l-k-1}(B_1,((j+d_t)\omega))\ar[r]_-{\gamma_l}&{\tH}^l(B_1,\sLjkt)\ar[r]_-{\alpha_l}& {\tH}^l(B_1,\Ljkt)\ar[r] & 0 
}
\end{equation}
In addition if $l\le k$ then $H^{l-k-1}(B_1,((j+d_t)\omega))=0$ and if $l>k$ then the vertical maps are all isomorphisms. 
\end{lemma}
\begin{proof}
The lower row is just \eqref{eq:sec2a}. For the upper row we note  for $l\leq k$ we have
$H^l(B_1,\Ljkt)\cong H^l(B_1,\sLjkt)$ by \eqref{eq:isoregulartate} and $H^{l-k-1}(B_1,((j+d_t)\omega))=0$
and thus the top row is exact
while for $l>k$ the left-most
and the right-most vertical maps are isomorphisms by \eqref{eq:B1char} and \eqref{eq:>l} and
thus the top row is isomorphic to the bottom row.
\end{proof}
We claim that it is enough to prove the following two statements
\begin{enumerate}
\item
$\tau_{\geq 1}\ulRHom_{B_1}(k,\sLjkt)$ is formal in $D(B^{(1)})$; 
\item
the upper exact sequence in \eqref{eq:3-lemma} splits.
\end{enumerate}
Indeed, if this is the case then we have a map
\begin{align*}
\bigoplus_{l\geq 1} H^l(B_1,\sLjkt)[-l]\cong
\tau_{\geq 1}\ulRHom_{B_1}(k,\sLjkt)&\cong\tau_{\geq 1}\RHom_{B_1}(k,\sLjkt)\\&\to \tau_{\geq 1}\RHom_{B_1}(k,\Ljkt),
\end{align*}
which yields an isomorphism in $D(B^{(1)})$
\[
\bigoplus_{l\ge 1} H^l(B_1,\Ljkt)[-l]\r \bigoplus_{l\ge 1} H^l(B_1,\sLjkt)[-l]\to \tau_{\geq 1}\RHom_{B_1}(k,\Ljkt) 
\]
using a splitting for the upper exact sequence in \eqref{eq:3-lemma}.

\medskip

We prove now that $\tau_{\geq 1}\ulRHom_{B_1}(k,\sLjkt)$ is formal in $D(B^{(1)})$ and that \emph{both} exact sequences in \eqref{eq:3-lemma} split.
Below we  construct maps for $l\ge 1$
\begin{equation}
\label{eq:Lformality}
\tH^{l}(B_1,\Ljkt)[-l]\to \ulRHom_{B_1}(k,\sLjkt)
\end{equation}
such that the composition 
\[
\tH^l(B_1,\Ljkt)\to \tH^l(B_1,\sLjkt)\xrightarrow{\alpha_l} \tH^l(B_1,\Ljkt)
\]
with the first map induced by applying $H^l(-)$ to \eqref{eq:Lformality} and the second coming from \eqref{eq:3-lemma} is an isomorphism.
This yields that the lower exact sequence in \eqref{eq:3-lemma} splits: 
\begin{equation}
\label{eq:splittingL}
\tH^l(B_1,\sLjkt)=\tH^l(B_1,\Ljkt)\oplus \tH^{l-k-1}(B_1,((j+d_t)\omega)).
\end{equation} 
Moreover the map $\gamma_l$ in \eqref{eq:3-lemma} is induced by the map (using Lemma \ref{lem:qiCM} for the first two isomorphisms) 
\begin{multline}\label{eq:mapK0}
\bigoplus_l \tH^l(B_1,((j+d_t)\omega))[-l-k-1]\cong \ulRHom_{B_1}(k,((j+d_t)\omega))[-k-1]\\\cong\ulRHom_{B_1}(k,\Ljt[-k-1])
\to \ulRHom_{B_1}(k,\Djkt)\cong \ulRHom_{B_1}(k,\sLjkt)
\end{multline}
and it follows that the direct sum of morphism \eqref{eq:mapK0} and \eqref{eq:Lformality} will give a formality morphism $H^l(B_1,\sLjkt)[-l]\to \ulRHom_{B_1}(k,\sLjkt)$. 

Finally note that it follows immediately from Lemma \ref{lem:upperlower} that if the lower exact sequence in  \eqref{eq:3-lemma} splits then so does the upper. 

\medskip

Hence all that remains to be done is the construction of the maps
\eqref{eq:Lformality}. We will do this now.  We look at the maps of complexes
$(\Djkt)^0\to \Djkt$
(resp. $(\Djkt)^{-1,0}:=\allowbreak ((\Djkt)^{-1}\to(\Djkt)^{0})\to \Djkt$) if
$k\neq j$ (resp. $k=j$). By Lemma \ref{lem:tatecohcom} (resp. Lemma
\ref{lem:2form}), $\ulRHom_{B_1}(k,(\Djkt)^0)$ (resp.
$\ulRHom_{B_1}(k,(\Djkt)^{-1,0})$) is formal. Thus we respectively get morphisms 
\begin{multline*}
\phi:\tH^l(B_1,(\Djkt)^0)[-l]\hookrightarrow 
\ulRHom_{B_1}(k,(\Djkt)^0) \to \ulRHom_{B_1}(k,\sLjkt),
 \end{multline*}
 \begin{multline*}
 \phi:\tH^l(B_1,(\Djkt)^{-1,0})[-l] \hookrightarrow 
 \ulRHom_{B_1}(k,(\Djkt)^{-1,0}) \to \ulRHom_{B_1}(k,\sLjkt),
\end{multline*}
where the last morphisms are induced by the respective projections $(\Djkt)^0 \to \sLjkt$, $(\Djkt)^{-1,0} \to \sLjkt$. 
Applying $H^l(-)$ to $\phi$ and composing with $\alpha_l$ we get a map
\begin{align*}
&
\tH^l(B_1,(\Djkt)^0)\xrightarrow{\delta_l} \tH^l(B_1,\sLjkt)\xrightarrow{\alpha_l} \tH^l(B_1,\Ljkt)\\
(\text{resp. }&
 \tH^l(B_1,(\Djkt)^{-1,0})\xrightarrow{\delta_l} \tH^l(B_1,\sLjkt)\xrightarrow{\alpha_l} \tH^l(B_1,\Ljkt))
. 
\end{align*}
The cohomology $\tH^l(B_1,(\Djkt)^0)$ (resp. $\tH^l(B_1,(\Djkt)^{-1,0})$)
is a sum of characters by Lemma \ref{lem:tatecohcom} (resp. Lemma
\ref{lem:2form}).  Thus, we can precompose $\theta$ with an inclusion of a $H^{(1)}$- (and thus $B^{(1)}$-) invariant subspace of $\tH^l(B_1,(\Djkt)^0)[-l]$ (resp. $\tH^l(B_1,(\Djkt)^{-1,0})[-l]$). By the following lemma we can view $\tH^l(B_1,\Ljkt)[-l]$ as such a subspace, obtaining therefore the desired map \eqref{eq:Lformality}.
\end{proof}

\begin{lemma}\label{lem:surjection}
\begin{enumerate} 
\item
If $k\neq j$ then
the composition
\begin{equation}
\label{eq:surj1}
\tH^l(B_1,(\Djkt)^0)\xrightarrow{\delta_l} \tH^l(B_1,\sLjkt)\xrightarrow{\alpha_l} \tH^l(B_1,\Ljkt)
\end{equation}
is a surjection.
\item
If $k=j$
then the composition
\begin{equation}
\label{eq:surj2}
\tH^l(B_1,(\Djkt)^{-1,0})\xrightarrow{\delta_l} \tH^l(B_1,\sLjkt)\xrightarrow{\alpha_l} \tH^l(B_1,\Ljkt)
\end{equation}
is a surjection.
\end{enumerate}
\end{lemma}

\begin{proof}
We argue as in the proof of Lemma \ref{lem:injection}.
By Corollary \ref{lem:Ljktweights}, 
 $\tH^l(B_1,\Ljkt)$ in particular has weights 
\begin{equation}
\label{eq:constraint1a}
\begin{cases}
\leq lp\omega &\text{if $k\leq j$, }\\
\leq (l-1)p\omega &\text{if $k> j$. }
\end{cases}
\end{equation}
On the other hand, $\tH^{l-u}(B_1,(\Djkt)^{u})$ up to multiplicity is equal to
\begin{equation}
\label{eq:constraint3}
\begin{cases}
((l-u)p\omega)&\text{if $u\geq k-j+1$,}\\
0&\text{if $u=k-j$,}\\
((l-1-u)p\omega)&\text{if $u\leq k-j-1$}
\end{cases}
\end{equation}
(see \eqref{eq:tHD}). 

Let $0\neq a\in \tH^l(B_1,\Ljkt)$ be a $H^{(1)}$-homogeneous element. Then the weights of $a$ are constrained by \eqref{eq:constraint1a}. Since $\alpha_l$ is surjective 
(see \eqref{eq:3-lemma})
$a$ is represented by a
$H^{(1)}$-homogeneous element $a'$ in $\tH^l(B_1,\sLjkt)=H^l(\Hom_{B_1}(C(k),\Djkt))$ of the same weight. Consider the exact sequences 
\[
0\r \sigma_{\ge q} \Djkt \r \Djkt\r \sigma_{\le q-1} \Djkt \r0
\]
with $q\in \{0,-1\}$ depending on the case. It is now sufficient to prove that the image of $a'$ in $H^l(\Hom_{B_1}(C(k),\sigma_{\le q-1} \Djkt))$ is zero since
then it is in the image of $H^l(\Hom_{B_1}(C(k),\sigma_{\ge q} \Djk))$. If it is not zero then
it determines some nonzero $H^{(1)}$-invariant element in $\tH^{l-u}(B_1,(\Djkt)^{u})$ for suitable $u\le q-1\le -1$. The only overlap between \eqref{eq:constraint1a} and \eqref{eq:constraint3}
and the condition $u\le -1$ is $u=-1$, $j=k$. However $u=-1$ implies $q=0$ which corresponds to the first case whereas $j=k$ corresponds to the second case. This is a contradiction.
\end{proof}

\subsubsection{$B_1$-cohomology of $\Ljkt$}
To establish Proposition \ref{prop:formality} 
it remains to prove that $H^l(B_1,\Ljkt)$ is a sum of characters for $l\geq 1$. For $l>k$ or $l=k$ and $\gamma_l=0$ this follows from \eqref{eq:>l} combined with Proposition \ref{prop:Tformal}. For $l<k$ or $l=k$ and $\gamma_l \neq 0$, this follows from \eqref{eq:isoregulartate}, combined with \eqref{eq:splittingL}, Proposition \ref{prop:Tformal} and \eqref{TateB1formulaA}. 

\subsection{Proof of Proposition \ref{prop:outform}}
The proposition follows by combining Lemma \ref{lem:qired} and Proposition \ref{prop:formality}.

\section{The cohomology computations}\label{sec:weights}
In this section we compute the cohomology of $H^l(B_1,\Cjk)$ for $l\geq 1$ up to (nonzero) multiplicities. Estimation of weights up  to possibly zero multiplicities is an easy and short calculation, and moreover rough enough to deduce Theorem \ref{thm:higherFr} (up to multiplicities) and consequently the FFRT property.   
Therefore the reader interested only in the latter, might skip the far more technical subsection \S\ref{sec:extra} giving a more precise estimation of the cohomology up to {\em nonzero} multiplicities, which  leads to the decomposition of $K_{jk}^{G_1}$  up to {\em{nonzero}} multiplicities in Proposition \ref{prop:Kjkrefine} below (and eventually to (the complete version of) Theorem \ref{thm:higherFr}).

\subsection{Weight bounds -- first approximation}
By Lemma \ref{lem:qired} it suffices to  bound the weights of $H^l(B_1,\Ljkt)$. Recall that this is a sum of characters by Proposition \ref{prop:formality}.

\begin{proposition}\label{prop:weights}
The weights of $H^l(B_1,\Ljkt)$ for $l\geq 1$ (resp. $l\geq 2$) if $k\leq j$ (resp. $k>j$) lie in the interval $[1,\max\{l,n-3\}]p\omega$. 
\end{proposition}

\begin{proof}
We separate the cases as in \eqref{eq:>l} and \eqref{eq:isoregulartate} in the proof of Proposition \ref{prop:formality}.
\begin{enumerate}
\item
Assume that $l>k$ or $l=k$ and $\gamma_l=0$. By Corollary \ref{lem:Ljktweights}
the weights of 
$H^l(B_1,\Ljkt)\cong \tH^l(B_1,\Ljkt)$ (see \eqref{eq:>l}) are contained in the interval  $[l-k,l]p\omega$.
 If $l > k$ then the desired bounds are clear.
 If $l=k$ we only need to assure that $0$ is not a weight of  $\tH^l(B_1,\Ljkt)$. This follows from \eqref{eq:gamma_l=0,beta_l=0}
 \eqref{eq:sec1} 
  using Corollary \ref{cor:Kweights} (which yields that the lower bound of $\tH(B_1,\sLjkt)$  is $l\ge 1$  if $k\leq j$ and  $l-1=k-1\geq j\geq 1$ if $k>j$)  and the fact that $\tH^0(B_1,((j+d_t)\omega))$ equals $0$ or $(q_{jt}^0p \omega)$ and $q_{jt}^0>0$ by Lemma \ref{rem:coll}.  
\item
Assume that $l<k$ or $l=k$ and $\gamma_l\neq 0$. Then $H^l(B_1,\Ljkt)\cong \tH^l(B_1,\sLjkt)$ by \eqref{eq:isoregulartate}, and by Corollary \ref{cor:Kweights} the weights of $\tH^l(B_1,\sLjkt)$ lie in the interval $[l,l+n\allowbreak -k-2]p\omega$ (resp. $[l-1,l+n-k-2]p\omega$) if $k\leq j$ (resp. $k>j$). This already gives us the lower bound. We also obtain the upper bound if $l\neq k$ as then $l+n-k-2\leq n-3$. 
For $l=k$ it follows from \eqref{eq:sec2a}, using Corollary \ref{lem:Ljktweights} 
and \eqref{eq:qiCM}, 
that $\tH^l(B_1,\sLjkt)$ has weights  $(-1+q_{jt}^{1})p\omega$, $\le lp\omega$ (up to multiplicity). As $-1+q_{jt}^{1}-1\leq n-3$ by Lemma \ref{rem:coll} we obtain the desired upper bound also in this case. \qedhere
\end{enumerate}
\end{proof}

\subsection{Weight bounds -- exact results}\label{sec:extra}
Here we show which non-tilt-free summands in \eqref{eq:list} occur with nonzero multiplicity. 
We rely on Proposition \ref{rmk:leftright}, which we prove in \S\ref{sec:extra2} below. 
\begin{proposition} \label{prop:R2}
Assume $1\le l\le n-3$ and set $\epsilon=[k>j]$ (see \eqref{eq:logical}). There is a decomposition $H^l(B_1,\Ljkt)=\allowbreak R_1\oplus R_2$
as $H^{(1)}$-representations
where $R_1\cong ((l-\epsilon)p\omega)^{\oplus a}$ for $a\in \NN$ and 
\begin{equation}
\label{eq:decomposition}
R_2
=
\begin{cases}
((q_{jt}^\ell+l-k)p\omega) &\text{if  $l\ge k$
and $q^\ell_{jt}$ is defined and satisfies $q^{\ell}_{jt}\le k-\epsilon$,}\\
((q_{jt}^{\bar{\ell}}+l-k-1)p\omega)&\text{if $l\le k$ and 
$q^{\bar{\ell}}_{jt}$ is defined and satisfies $q^{\bar{\ell}}_{jt}>k-\epsilon$,}\\
0&\text{otherwise,}
\end{cases}
\end{equation}
for $\ell\equiv l-k(2)$ and  $\bar{\ell}\equiv 1-\ell(2)$. 
\end{proposition}
\begin{proof}
 This proposition is proved by considering several cases, and determining precisely the possible values of $H^l(B_1,\Cjk)$ in each of these cases.
More precisely, we distinguish the following:

 \begin{enumerate}
 \item \label{case-(1)}
If $q_{jt}^0$, $q_{jt}^1$ are undefined then from 
\eqref{eq:qiCM} combined with \eqref{eq:5-lemma} we see that  
\[
H^l(B_1,\Ljkt)
\cong {H}^l(B_1,\sLjkt)\cong {\tH}^l(B_1,\sLjkt)
 \cong {\tH}^l(B_1,\Ljkt)
\]
This group 
 equals $((l-\epsilon)p\omega)$  up to (possibly zero) multiplicity by Corollary \ref{cor:Kweights} (which restricts the weights of $\tH^l(B_1,\sLjkt)$) combined with Corollary \ref{lem:Ljktweights} (which restricts the weights of $\tH^l(B_1,\Ljkt)$).
  \item
If $q_{jt}^\ell$ is defined for $\ell\in\{0,1\}$ 
 then we have several cases to consider. 
\begin{enumerate}
\item $l>k$: by \eqref{eq:>l}, we have $H^l(B_1,\Ljkt)=\tH^l(B_1,\Ljkt)$. We now consider two further cases. \label{eq:case-i}
\begin{enumerate}
\item\label{eq:case>k1} $q_{jt}^\ell\leq k-\epsilon$. 
Then 
by Proposition \ref{rmk:leftright}\eqref{gamma_l=0},
\begin{equation}
\label{eq:split-sum}
\tH^l(B_1,\Ljkt)\cong 
\begin{cases}
\tH^l(B_1,\sLjkt)\oplus ((q_{jt}^\ell+l-k)p\omega) &\text{if $\ell\equiv l-k\,(2)$},\\
\tH^l(B_1,\sLjkt) &\text{if $\ell\equiv l-k+1\,(2)$}.
\end{cases}
\end{equation}
In particular $\tH^l(B_1,\sLjkt)$ is a summand of $\tH^l(B_1,\Ljkt)$ and by essentially the same reasoning as in the case \eqref{case-(1)} (combining weight restrictions obtained
from Corollary \ref{cor:Kweights} and  Corollary \ref{lem:Ljktweights})
we find that $\tH^l(B_1,\sLjkt)$ 
equals $((l-\epsilon)p\omega)$ up to (possibly zero) multiplicity. 
\item\label{eq:case>k} $q_{jt}^\ell> k-\epsilon$. Then 
we are in the case \eqref{gamma_l!=0} of Proposition \ref{rmk:leftright} which implies 
that 
$\tH^l(B_1,\Ljkt)$ is a direct summand of $\tH^l(B_1,\sLjkt)$. Therefore, it has weight $(l-\epsilon)p\omega$ up to (possibly zero) multiplicity, again reasoning as in the case \eqref{case-(1)}.
\end{enumerate}
\item $l<k$: 
by \eqref{eq:isoregulartate}, $H^l(B_1,\Ljkt)\cong H^l(B_1,\sLjkt)\cong \tH^l(B_1,\sLjkt)$. We again consider two further cases.
\begin{enumerate}
\item 
$q_{jt}^\ell\leq k-\epsilon$. 
Proceeding exactly as in (\ref{eq:case>k1}) we find that all weights of $\tH^l(B_1,\sLjkt)$ are equal to $((l-\epsilon)p\omega)$.
\item\label{eq:case<k2} $q_{jt}^\ell> k-\epsilon$. In this case 
by Proposition \ref{rmk:leftright}\eqref{gamma_l!=0} 
\[
\tH^l(B_1,\sLjkt)\cong 
\begin{cases}
((q_{jt}^\ell+l-k-1)p\omega)
\oplus \tH^l(B_1,\Ljkt) &\text{if $\ell\equiv l-k-1\,(2)$},\\
\tH^l(B_1,\Ljkt)&\text{if $\ell\equiv l-k\,(2)$},
\end{cases}
\]
where (as in the case \eqref{case-(1)}) $\tH^l(B_1,\Ljkt)$ equals $((l-\epsilon)p\omega)$ up to (possibly zero) multiplicity. 
\end{enumerate}
\item $l=k$:
\begin{enumerate}
\item\label{eq:case<k3} $q_{jt}^\ell\leq k-\epsilon$. Then by Proposition \ref{rmk:leftright}\eqref{case_gamma_l=0}, $\gamma_l=0$ and 
 thus  $H^l(B_1,\Ljkt)\cong \tH^l(B_1,\Ljkt)$ by \eqref{eq:5-lemma}
 so we can proceed as in  \eqref{eq:case>k1}.
\item \label{eq:case>k3}$q_{jt}^\ell> k-\epsilon$. Then by Proposition \ref{rmk:leftright}\eqref{case_gamma_l!=0},  $\beta_l=0$ and thus  $H^l(B_1,\Ljkt)\cong H^l(B_1,\sLjkt)\cong \tH^l(B_1,\sLjkt)$ by \eqref{eq:5-lemma} so we can proceed as in \eqref{eq:case<k2}.\qedhere 
\end{enumerate}
\end{enumerate}
 \end{enumerate} 
\end{proof}
\begin{proposition}\label{prop:refKjkdecomposition}
Assume $1\leq r\leq n-3$,  $1+\epsilon\leq l\leq n-3$ for $\epsilon=[k>j]$ (see \eqref{eq:logical}).  
The summand $(rp\omega)\otimes_k S^p_+$  occurs in $H^l(B_1,\Cjk)$  with nonzero  multiplicity if and only if 
\begin{equation}
\label{eq:intervals}
r\in
\begin{cases}
[l-k+m_{lk},l-\epsilon] &\text{if $l>k$},\\
[1,n-3]&\text{if $l=k$},\\
[l-\epsilon,l-k+n-2-m_{lk}] &\text{if $l<k$},
\end{cases}
\end{equation}
where
\begin{equation}
\label{eq:mlk1}
m_{lk}=
\begin{cases}
1&\text{if $l\equiv k\,(2)$,}\\
1&\text{if  $l\equiv k+1\,(2)$,  $l> k$, $j=n-3$, and $p=n-2$,}\\
1&\text{if  $l\equiv k+1\,(2)$, $l< k$, $j=1$, and $p=n-2$,}\\
0&\text{otherwise.}
\end{cases}
\end{equation}
\end{proposition}
\begin{proof}
By \eqref{eq:Cvar1} and \eqref{eq:Cjkbar}, 
\begin{equation}\label{eq:pourdegre}
H^l(B_1,\Cjk)\cong H^l(B_1,\oCjk)\cong \bigoplus_{t\in [0,p-1]^n}H^l(B_1,C_{jk}^{(t)})(-d_t)\otimes_k S^p_+.
\end{equation}
First observe that all the intervals in \eqref{eq:intervals} contain in fact $l-\epsilon$ (the possible contribution from the $R_1$ summand in Proposition \ref{prop:R2}) by the assumption on $l$ and noting that when $\epsilon=1$ then $k>1$ so the first interval in \eqref{eq:intervals} is nonempty. So to prove the proposition it is sufficient that we can get all
weights in the intervals in \eqref{eq:intervals} from the $R_2$ summand in Proposition \ref{prop:R2}.
Varying $t$ and thus $q_{jt}^\ell$ (whose possible values were computed in Lemma \ref{rem:coll}) we obtain the range of legal weights for $R_2$.

Below put
\begin{align*}
u&=[(j,p)=(1,n-2)],\\
\bar{u}&=[(j,p)=(n-3,n-2)].
\end{align*}
It follows from Lemma \ref{rem:coll} that
\begin{align*}
q^0_{jt}&\in [1,n-1-u],\\
q^1_{jt}&\in [\bar{u},n-2],
\end{align*}
and furthermore all possible integer values are attained.
\begin{enumerate}
\item $l-k\equiv 0(2)$. In this case $\ell=0$ and we find that the weights of $R_2$ attain all values in
\begin{equation}
\label{eq:evenbounds}
\begin{cases}
[1+l-k,l-\epsilon]p\omega&\text{if $l\ge k$,}\\
[l-\epsilon,n+l-k-3]p\omega &\text{if $l\le k$.}
\end{cases}
\end{equation}
Since in this case we have $m_{lk}=1$ by \eqref{eq:mlk1} this is compatible with \eqref{eq:intervals} (since in the case $l=k$, $1+l-k=1\leq l-\epsilon$ by the assumption on $l$).
\item $l-k\equiv 1(2)$. In this case $\ell=1$ and we find that the weights of $R_2$ attain all values in
\begin{equation}
\label{eq:firstshot}
\begin{cases}
[\bar{u}+l-k,l-\epsilon]p\omega&\text{if $l>k$,}\\
[l-\epsilon,n+l-k-2-u]p\omega&\text{if $l< k$}
\end{cases}
\end{equation}
as $l=k$ is excluded by $l-k\equiv 1(2)$.
If $l< k$ then we have $m_{lk}=u$ and if $l>k$ then $m_{lk}=\bar{u}$. In both cases the result we get coincides with \eqref{eq:intervals}.
\qedhere\end{enumerate}
\end{proof}
\begin{remark} Note that if $p\ge n-1$ then
\eqref{eq:mlk1} reduces to the much simpler
\[
m_{lk}=
\begin{cases}
1&\text{if $l\equiv k\,(2)$,}\\
0&\text{otherwise.}
\end{cases}
\]
\end{remark}

The proof of Proposition \ref{prop:refKjkdecomposition} also give us degrees of the summands of $H^l(B_1,\Cjk)$ which will be needed in \S\ref{sec:grass}  for the decomposition of $\Fr_*^r\Oscr_\GG$ (where only those summands of $R$ (considered as $R^{p^r}$-module)  which live in degrees divisible by $2p^r$ are visible).  
We record here a rough version sufficient for establishing  Theorem \ref{thm:FrOscr} below. 
\begin{remark}\label{rmk:degrees}
Let $r$ be as in \eqref{eq:intervals}. We claim that as graded $(B^{(1)},S^p)$-modules, $(rp\omega)\otimes_k S^p_+(-d_t)$ is a summand of $H^l(B_1,\Cjk)$ for some $t\in [0,p-1]^n$ such that $q_{jt}^\ell$ is defined and $q_{jt}^\ell+\ell\equiv r \,(2)$. 

To see this note that by \eqref{eq:pourdegre},  $(rp\omega)\otimes_k S^p_+(-d_t)$ is a summand of $H^l(B_1,\Cjk)$ for some $t\in [0,p-1]^n$ such that $(rp\omega)$ is a summand of $H^l(B_1,\Ljkt)$. As noted in the beginning of the proof of 
Proposition \ref{prop:refKjkdecomposition}, $(rp\omega)$ is then a summand of $R_2$ in Proposition \ref{prop:R2}. In particular we either have
$r=q^\ell_{jt}+l-k$ with $\ell\equiv l-k(2)$ or $r=q^{\ell}_{jt}+l-k-1$ with $\ell\equiv 1-l+k(2)$ (replacing $\bar{\ell}$ in the second case of \eqref{eq:decomposition} by $\ell$).
Hence in both cases we get $q_{jt}^\ell+\ell\equiv r \,(2)$. 
\end{remark}
\subsubsection{Refined version of Lemma \ref{lem:splitH}}
\label{sec:extra2}

\begin{proposition}\label{rmk:leftright}
Set $\epsilon=[k>j]$ (see \eqref{eq:logical}). There are the following  (mutually exclusive) cases:
\begin{enumerate}
\item If $q_{jt}^0$, $q_{jt}^1$ are both undefined then $\beta_l=0$, $\gamma_l=0$ for all $l$ and
\begin{equation}
\label{eq:decomp0}
\tH^l(B_1,\Ljkt)\cong \tH^l(B_1,\sLjkt).
\end{equation}
\item\label{case_gamma_l=0}
If $q_{jt}^\ell$ is defined and $q_{jt}^\ell\leq k-\epsilon$ 
then $\gamma_l=0$ for all $l$ and 
\begin{equation}
\label{eq:decomp1}
\begin{aligned}
\tH^l(B_1,\Ljkt)&\cong \tH^l(B_1,\sLjkt)\oplus 
\tH^{l-k}(B_1,((j+d_t)\omega))\\
&\cong 
\begin{cases}
\tH^l(B_1,\sLjkt)\oplus((q_{jt}^\ell+l-k)p\omega) & {\text{if $\ell\equiv l-k\,(2)$}},\\
\tH^l(B_1,\sLjkt) &\text{if $\ell\equiv l-k-1\,(2)$}.\\
\end{cases}
\end{aligned}
\end{equation}
\item\label{case_gamma_l!=0}
If $q_{jt}^\ell$ is defined and $q_{jt}^\ell> k-\epsilon$ 
then $\beta_l=0$ for all $l$ and
\begin{equation}
\label{eq:decomp2}
\begin{aligned}
\tH^{l}(B_1,\sLjkt)&\cong \tH^{l-k-1}(B_1,((j+d_t)\omega))\oplus \tH^{l}(B_1,\Ljkt)\\
&\cong
\begin{cases}
((q_{jt}^\ell+l-k-1)p\omega)\oplus\tH^{l}(B_1,\Ljkt) & \text{if $\ell\equiv l-k-1\,(2)$},\\
\tH^{l}(B_1,\Ljkt)&\text{if $\ell\equiv l-k\,(2)$}.
\end{cases}
\end{aligned}
\end{equation}
\end{enumerate}
\end{proposition}

To prove Proposition \ref{rmk:leftright} we need to better understand the cohomology of $\Ljkt$, $\sLjkt$. The following lemmas are basically a consequence of Lemma \ref{lem:tatecohcom}.
\begin{lemma}\label{lem:specsecdeg3}
  The $\uparrow\rightarrow$ spectral sequence associated to $\Hom_{B_1}(C(k),\Ljkt)$ degenerates at $E_3$ and $d_2^{uv}:E^{uv}_2 \to E^{u+2,v-1}_2$ is $0$ for $u\neq k-j-2$. 
\end{lemma}

\begin{proof} Since the $\u\r$ spectral sequence  associated to $\Hom_{B_1}(C(k),\Ljkt)$ is obtained by restricting the $\u\r$ spectral sequence  associated to $\Hom_{B_1}(C(k),\Ljt)$
to $u\ge -k$ and then shifting everything $k$ places to the right, this follows by inspecting \eqref{eq:prelimE1} taking into accounts that the differentials are $H^{(1)}$-equivariant
and hence  must respect weights.
\end{proof}

\begin{lemma}\label{lem:E30} The $E_3$-term of the $\uparrow\rightarrow$ spectral sequence associated to the double complex $\Hom_{B_1}(C(k),\Ljt)$
is concentrated in a single column. 
\begin{enumerate}
\item If $q^{0}_{jt}$, $q^1_{jt}$ are both undefined then $E^{uv}_3=0$ for all $u,v$.
\item Assume that $q^{\ell}_{jt}$ is defined for $\ell\in \{0,1\}$ and put
\begin{equation}
\label{eq:ultimate}
u=
\begin{cases}
-q^{\ell}_{jt}&\text{if $q^{\ell}_{jt}\le j$,}\\
-q^{\ell}_{jt}-1&\text{if $q^{\ell}_{jt}\ge j+1$.}
\end{cases}
\end{equation}
In that case $u\neq -j-1$ and
the only nonzero cohomology of the $E_3$-page occurs in the $u$'th column. More precisely
\begin{equation}
\label{eq:ultimate1}
E^{uv}_3
=
\begin{cases}
(vp\omega)&\text{if $u\ge -j$, $u+v\equiv\ell (2)$,}\\
((v-1)p\omega)&\text{if $u\le -j-2$, $u+v\equiv \ell(2)$,}\\
0&\text{if $u+v\not\equiv \ell(2)$.}
\end{cases}
\end{equation}
\end{enumerate}
\end{lemma}
\begin{proof}
 As in the proof of Lemma \ref{lem:specsecdeg3} we know that the differentials in the spectral sequence are strongly constrained by Lemma \ref{lem:tatecohcom}. In particular, one checks similarly that it degenerates at $E_3$. Decomposing the pages according to weights we find that the $E_3$-page is given
by taking the cohomology of complexes
\begin{equation}
\label{eq:weightdecomp}
\xymatrix@=1em{
\text{(weights $rp\omega$)} & \cdots \ar[r] &E_1^{-j-3,r+1}\ar[r] & E_1^{-j-2,r+1}\ar[drr]\\
& & & & & E_1^{-j,r}\ar[r] & E^{-j+1,r}_1\ar[r]&\cdots
}
\end{equation}
where the diagonal arrow is the composition
\[
E_1^{-j-2,r+1}\twoheadrightarrow E_2^{-j-2,r+1}\xrightarrow{d_2^{-j-2,r+1}} E_2^{-j,r}\hookrightarrow 
E_1^{-j,r} 
\]
From Lemma \ref{lem:qiCM} and Lemma \ref{rem:coll} we obtain that $E_\infty$ contains at most a single entry with weights $rp\omega$. Since $d_3=0$, it follows that \eqref{eq:weightdecomp} has cohomology in at most a single place. Using Lemma \ref{lem:qiCM} again it follows that (1-dimensional) cohomology $(rp\omega)$  will occur in degree $(u,v)$ in \eqref{eq:weightdecomp} 
if and only if
$
q_{jt}^{(u+v)(2)}
$ is defined, $u\neq -j-1$ and moreover
\begin{equation}
\label{eq:uv}
q_{jt}^{(u+v)(2)}+u+v=r
\end{equation}
where in addition
\begin{equation}
\label{eq:r}
v=
\begin{cases}
r&\text{if $u\ge -j$},\\
r+1&\text{if $u\le -j-2$}.
\end{cases}
\end{equation}
Substituting \eqref{eq:r} in \eqref{eq:uv}
we find that for $\ell\equiv u+v\,(2)$
\[
u=
\begin{cases}
-q^{\ell}_{jt}&\text{if $u\ge -j$,}\\
-q^{\ell}_{jt}-1&\text{if $u\le -j-2$.}\\
\end{cases}
\]
It is easy to see that this condition, combined with $u\neq -j-1$,  is equivalent to \eqref{eq:ultimate}. 

The fact that the spectral sequence is concentrated in a single column (i.e.\ is nonzero
for a single value for $u$) now follows since $q^{0}_{jt}$, $q^1_{jt}$ cannot both be defined (see Lemma \ref{rem:coll}).
\end{proof}

\begin{proof}[Proof of Proposition \ref{rmk:leftright}]
We will first assume the claims about $(\beta_l)_l$, $(\gamma_l)_l$ have been established and we verify case by case that they imply  \eqref{eq:decomp0}
and the top lines in \eqref{eq:decomp1}\eqref{eq:decomp2}.
We may then invoke \eqref{eq:qiCM} to further evaluate $\tH^{l-k}(B_1,((j+d_j)\omega))$ and get the other parts of \eqref{eq:decomp1}\eqref{eq:decomp2}.
\begin{enumerate}
\item Here we have $\beta_l=0$, $\gamma_l=0$ for all $l$. Then \eqref{eq:decomp0} follows immediately from \eqref{eq:gamma_l=0,beta_l=0}.
\item Now we have $\gamma_l=0$ for all $l$. We need to prove that
  \eqref{eq:sec1} splits. This follows from Proposition
  \ref{prop:Tformal} which implies that the middle term
  $\tH^l(B_1,\Ljkt)$ of \eqref{eq:sec1} is sum of characters.
\item Here we  have $\beta_l=0$ for all $l$. Now we need that \eqref{eq:sec2a} splits. This has been established in \eqref{eq:splittingL}.
\end{enumerate}
We now proceed by proving the claims about $(\beta_l)_l$, $(\gamma_l)_l$.
We will consider the fragment of \eqref{eq:les}
\begin{equation}
\label{eq:fragment}
\tH^{l}(B_1,\Ljkt)\xrightarrow{\beta_l} \tH^{l-k}(B_1,((j+d_t)\omega))\xrightarrow{\gamma_{l+1}} \tH^{l+1}(B_1,\sLjkt).
\end{equation}
The case that $q^0_{jt}$ and $q^1_{jt}$ are both undefined follows immediately from \eqref{eq:qiCM}. Hence we now assume that $q^\ell_{jt}$ is defined for some $\ell\in \{0,1\}$.
Note that \eqref{eq:fragment} is obtained by applying $\Hom_{B_1}(C(k),-)$ to
\[
\Ljkt[k]=\sigma_{\ge -k} \Ljt \to \Ljt\to \sigma_{<-k} \Ljt= \Djkt[k+1]
\]
and shifting by $-k$.

Let $(E^{uv}_n)_{u\ge -k}$, $E^{uv}_n$, $(E^{uv}_n)_{u<-k}$ be the $\uparrow\rightarrow$ spectral sequences associated to the double complexes $\Hom_{B_1}(C(k),\sigma_{\ge -k} \Ljt)$, 
$\Hom_{B_1}(C(k), \Ljt)$,  $\Hom_{B_1}(C(k),\allowbreak \sigma_{<-k} \Ljt)$. These spectral sequences degenerate at $E_3$ (see \eqref{eq:prelimE1}, Lemmas \ref{lem:specsecdeg3},\ref{lem:E30}). 

After decomposing
by weights the $E_3$-term of $E_n^{uv}$ consists of the cohomology of the complexes \eqref{eq:weightdecomp} which we will temporarily denote by $\Cscr_{rjt}$. Moreover 
the $E_3$-terms of $(E_n^{uv})_{u\ge -k}$ and $(E^{uv}_n)_{u<-k}$ are given by the cohomology of the brutal truncations  $(\Cscr_{rjt})_{u\ge -k}$,
$(\Cscr_{rjt})_{u< -k}$. In other words the morphisms (obtained by functoriality)
\[
(E^{uv}_\infty)_{u\ge -k}\r E^{uv}_\infty\r (E^{uv}_\infty)_{u<-k}
\]
are obtained from taking the cohomology of 
\[
(\Cscr_{rjt})_{u\ge -k}\r \Cscr_{rjt}\r
(\Cscr_{rjt})_{u< -k}.
\]
If  $\tH^{l-k}(B_1,((j+d_t)\omega))=0$ then $\gamma_{l+1}=0$ and $\beta_l=0$. If $\tH^{l-k}(B_1,((j+d_t)\omega))\neq 0$
then a generator of  $\tH^{l-k}(B_1,((j+d_t)\omega))$ is represented by an element $a\neq 0$ of $H^\ast(\Cscr_{rjt})$ for suitable $r$. Let $\bar{u}$ be the column in $E^{uv}_n$ in which it appears.
The formula for $\bar{u}$  
is given by \eqref{eq:ultimate}. 
We have 
\begin{align*}
\gamma_{l+1}=0&\iff \text{$a$ has zero image in $H^\ast((\Cscr_{rjt})_{u< -k})$}&\iff  \bar{u}\ge -k,\\
\beta_l=0&\iff \text{$a$ has nonzero image in  $H^\ast((\Cscr_{rjt})_{u< -k})$}&\iff \bar{u}< -k.
\end{align*}
Now $\bar{u}$ is in fact independent of $k,l$ (see \eqref{eq:ultimate}).
So we find implications
\begin{align*}
\bar{u}\ge -k&\Rightarrow \forall l:\gamma_l=0,\\
\bar{u}< -k&\Rightarrow \forall l:\beta_l=0.
\end{align*}
Furthermore, using \eqref{eq:ultimate} it easy to see  that we have equivalences
\begin{align*}
\bar{u}\ge -k &\iff q_{jt}^\ell\leq k-\epsilon,\\
\bar{u}< -k &\iff q_{jt}^\ell>k-\epsilon, 
\end{align*}
so that in Case \eqref{case_gamma_l=0} we get $\forall l:\gamma_l=0$ and in
Case \eqref{case_gamma_l!=0} we get $\forall l:\beta_l=0$. 
\end{proof}

\section{Induction from $B$-level to $G$-level}
\label{sec:ind-B-G}
In \S\ref{sec:formality}, \S\ref{sec:weights} we were concerned with complexes of $B^{(1)}$-modules. However, to apply Theorem \ref{prop:main} for the  decomposition of $K_{jk}^{G_1}$ we need their $G^{(1)}$-variants. In this section we show that induction enables a passage between the $B^{(1)}$- and the $G^{(1)}$-level.

We denote the resolution  \eqref{biresolution01} of $M_j$ by $\gCj$ (with the last term 
in degree $0$). 
Now set 
\[
\gCjk=(\sigma_{\geq -k} \gCj)[-k].
\] 
\begin{lemma} 
\label{lem:inductionlemma}
We have
\begin{equation}
\label{eq:BtoG}
\RInd_{B}^G\Cjk=
\begin{cases}
\gCjk&\text{if $k\leq j$,}\\
\tilde{C}_{j,k-1}^\bullet[-1]&\text{if $k>j$.}
\end{cases}
\end{equation}
\end{lemma}
\begin{proof}
We give a characteristic free proof.
First consider the case $k\le j$. In that case $\sigma_{\ge -k} \tCj$ consists of terms
acyclic for $\Ind^G_B$. Hence we obtain the equality $\RInd^G_B\sigma_{\ge -k} \tCj=\Ind^G_B \sigma_{\ge -k}\tCj$.
From the way $\gCj$ was obtained from $\tCj$ in \cite[\S\ref{sec:resol}]{FFRT1} we obtain
$\Ind^G_B \sigma_{\ge -k}\tCj=\sigma_{\ge -k} \gCj$. Applying $?[-k]$ yields \eqref{eq:BtoG}.

The case $k>j$ is a bit more involved. We start from the short exact sequence of complexes
\[
0\r \sigma_{\ge -j}\tCj \r \sigma_{\ge -k} \tCj \r \sigma_{<-j}\sigma_{\ge -k} \tCj \r 0.
\]
The terms of $\sigma_{< -j}\sigma_{\ge -k}\tCj$ satisfy $R^i\Ind^G_B(?)=0$ for $i\neq 1$. 
Using an injective Cartan-Eilenberg resolution for $\sigma_{< -j}\sigma_{\ge -k}\tCj$ one finds 
as in the first paragraph 
\[
\RInd^G_B  \sigma_{<-j}\sigma_{\ge -k} \tCj =R^1\Ind^G_B  \sigma_{<-j}\sigma_{\ge -k} \tCj[-1] =
\sigma_{<-j} \sigma_{\ge -k+1}\gCj.
\]
Hence we get a distinguished triangle (using that  $\RInd_{B}^G \sigma_{\ge -j}\tCj=\sigma_{\ge -j}\gCj $ by the first paragraph)
\[
\sigma_{\ge -j}\gCj \r \RInd^G_B\sigma_{\ge -k} \tCj \r \sigma_{<-j}\sigma_{\ge -k+1} \gCj 
\r 
\]
and hence
\begin{equation}
\label{eq:cone}
\RInd^G_B\sigma_{\ge -k} \tCj=
\cone(\sigma_{<-j}\sigma_{\ge -k+1} \gCj[-1] 
\r \sigma_{\ge -j}\gCj ).
\end{equation}
If $k=j+1$ then $\sigma_{<-j}\sigma_{\ge -k+1} \gCj=0$, and thus \eqref{eq:BtoG} follows from  \eqref{eq:cone} 
by shifting. We assume now that $k>j+1$. 
By  \cite[Proposition \ref{prop:actual}]{FFRT1}, the map $\sigma_{<-j}\sigma_{\ge -k+1} \gCj[-1] 
\r \sigma_{\ge -j}\gCj$ is represented by a map of complexes and hence by a map
$d^{-j-1}:\tilde{C}_j^{-j-1}\r \tilde{C}_j^{-j}$. This map must be nonzero since otherwise
by \eqref{eq:cone}  $\RInd^G_B\sigma_{\ge -k}\tCj = \sigma_{\ge -j}\gCj  \oplus \sigma_{<-j}\sigma_{\ge -k+1} \gCj$, which
would have at least three non-vanishing cohomology
groups, whereas by \eqref{eq:LK} it can only have two. It is easy to see
that then $d^{-j-1}$ is uniquely determined, up to scalar, by $G\times \GL(F)$-equivariance
and compatibility with the grading. In particular it is the same, up to scalar, as the
corresponding differential in $\sigma_{\ge -k+1}\gCj$. We then obtain 
$\RInd^G_B\sigma_{\ge -k} \tCj=\sigma_{\ge -k+1}\gCj$ from 
\eqref{eq:cone} 
and hence \eqref{eq:BtoG}
by shifting.
\end{proof}
\begin{proposition}\label{prop:BGpre}
In $D(G^{(1)},S^p)$ we have 
\begin{equation}
\label{eq:cases}
\tau_{\geq 1}\RHom_{G_1}(k,\gCjk)\cong 
\begin{cases}
\RInd_{B^{(1)}}^{G^{(1)}}\tau_{\ge 1} \RHom_{B_1}(k,\Cjk)
&\text{if $k\leq j$,}\\
(\RInd_{B^{(1)}}^{G^{(1)}}\tau_{\ge 2} \RHom_{B_1}(k,C_{j,k+1}^\bullet))[1]
&\text{if $k\geq j$}.
\end{cases}
\end{equation}
\end{proposition}
\begin{remark} The overlap between the cases in \eqref{eq:cases} is intentional and correct.
\end{remark}
\begin{proof}[Proof of Proposition \ref{prop:BGpre}]
By the Andersen-Jantzen [AJ] spectral sequence \cite[eq. (2)]{JA} (see \cite[\eqref{identity1}]{FFRT1}) and Lemma \ref{lem:inductionlemma} one finds
\begin{equation}
\label{eq:AJourcase}
\RHom_{G_1}(k,\gCjk)\cong 
\begin{cases}
\RInd_{B^{(1)}}^{G^{(1)}}\RHom_{B_1}(k,\Cjk)&\text{if $k\leq j$,}\\
\RInd_{B^{(1)}}^{G^{(1)}}\RHom_{B_1}(k,C_{j,k+1}^\bullet)[1]
&\text{if $k\geq j$.}
\end{cases}
\end{equation}

We first consider the case $k\leq j$. 
Applying Lemma \ref{lem:trunc} below with $F=R\Ind^{G^{(1)}}_{B^{(1)}}$, $X= \RHom_{B_1}(k,\Cjk) $, $Y= \RHom_{G_1}(k,\gCjk)$ and the isomorphism
$Y\cong FX$ being given by \eqref{eq:AJourcase} we obtain a morphism of distinguished triangles 
\begin{equation}\label{eq:trian}
\begin{small}
\begin{tikzcd}
K_{jk}^{G_1} \ar{r} \ar{d} & \RHom_{G_1}(k,\gCjk)\ar{r}\ar{d}{\cong} & \tau_{\geq 1}\RHom_{G_1}(k,\gCjk)\ar{d}\\
\RInd_{B^{(1)}}^{G^{(1)}}L_{jk}^{B_1} \ar{r} & \RInd_{B^{(1)}}^{G^{(1)}}\RHom_{B_1}(k,\Cjk) \ar{r} & \RInd_{B^{(1)}}^{G^{(1)}}\tau_{\geq 1}\RHom_{B_1}(k,\Cjk)
\end{tikzcd}
\end{small}
\end{equation}
It suffices to check that \eqref{eq:trian} is an isomorphism of distinguished triangles. 
By Lemma \ref{lem:trunc} below
we are reduced to showing
\begin{equation}\label{eq:todo}
R^1\Ind_{B^{(1)}}^{G^{(1)}} L_{jk}^{B_1}=0.
\end{equation}
As $\Djkt\to \sLjkt$ is quasi-isomorphism we have by \eqref{eq:tDjkbar}\eqref{eq:Djkbar}\eqref{eq:Djkt}
\begin{align}\label{eq:tLjkvsLjkt}
\tH^l(B_1,L_{jk})&=\bigoplus_{t\in [0,p-1]^n}\tH^l(B_1,\sLjkt)\otimes_k S^p_{+}&\text{for all $l\in\ZZ$},
\\\label{eq:LjkvsLjkt}
H^l(B_1,L_{jk})&=\bigoplus_{t\in [0,p-1]^n}H^l(B_1,\sLjkt)\otimes_k S^p_{+}&\text{for all $l\geq 1$}.
\end{align}
We use the Andersen-Jantzen spectral sequence \cite[\eqref{identity2}]{FFRT1} for the Tate cohomology, \eqref{eq:LK}, and the positivity of the weights of $\tH^l(B_1,L_{jk})$ for $l\geq 0$ (see \eqref{eq:tLjkvsLjkt} and Corollary \ref{cor:Kweights}) to obtain for $l\geq 1$ 
\begin{equation}\label{eq:Ktateind}
\Ind_{B^{(1)}}^{G^{(1)}} \tH^l(B_1,L_{jk})\cong \tH^l(G_1,K_{jk}).
\end{equation}
Since $H^l(G_1,K_{jk}^\un)=\tH^l(G_1,K_{jk}^\un)$ and $H^l(B_1,L_{jk})=\tH^l(B_1,L_{jk})$ for $l\geq 1$,  we in particular have $H^1(B_1,K_{jk})=\Ind_{B^{(1)}}^{G^{(1)}}H^1(B_1,L_{jk})$ by \eqref{eq:Ktateind}, thus \eqref{eq:LK} and the Andersen-Jantzen spectral sequence  \cite[\eqref{identity1}]{FFRT1} for ordinary cohomology imply:
\begin{align*}
\Ind_{B^{(1)}}^{G^{(1)}}H^1(B_1,L_{jk})&=H^1(G_1,K_{jk})\\
&= H^1(G_1,R\Ind_B^GL_{jk})\\
&=H^1(R\Ind_{B^{(1)}}^{G^{(1)}}\RHom_{B_1}(k,L_{jk})),
\end{align*}
which is an extension of $\Ind_{B^{(1)}}^{G^{(1)}}H^1(B_1,L_{jk})$ by $R^1\Ind_{B^{(1)}}^{G^{(1)}} L_{jk}^{B_1}$. Thus, the latter needs to be $0$, yielding \eqref{eq:todo}.

\medskip

We next consider the case $k\ge j$. 
Applying Lemma \ref{lem:trunc} below with $F=R\Ind^{G^{(1)}}_{B^{(1)}}$, $X= \RHom_{B_1}(k,C_{j,k+1}^\bullet)[1] $, $Y= \RHom_{G_1}(k,\gCjk)$ and the isomorphism
$Y\cong FX$ being given by \eqref{eq:AJourcase} we obtain a morphism of distinguished triangles 
\begin{equation}\label{eq:trian2}
\begin{tiny}
\mathclap{\begin{tikzcd}[column sep=tiny, ampersand replacement=\&]
K_{jk}^{G_1} \ar{r} \ar{d} \& \RHom_{G_1}(k,\gCjk)\ar{r} \ar{d}{\cong} \& \tau_{\geq 1}\RHom_{G_1}(k,\gCjk) \ar{d}\\
\RInd_{B^{(1)}}^{G^{(1)}}\tau_{\le 1} \RHom_{B_1}(k,C_{j,k+1}^\bullet)[1] \ar{r} \&\RInd_{B^{(1)}}^{G^{(1)}}\RHom_{B_1}(k,C_{j,k+1}^\bullet)[1] \ar{r}\& \RInd_{B^{(1)}}^{G^{(1)}}\tau_{\geq 2}\RHom_{B_1}(k,C_{j,k+1}^\bullet)[1]
\end{tikzcd}}
\end{tiny}
\end{equation}
It suffices to check that \eqref{eq:trian2} is an isomorphism of distinguished triangles.
By Lemma \ref{lem:trunc} below 
we are reduced to showing (replacing $k+1$ by $k$) that for $k> j$
\begin{equation}\label{eq:todo2}
H^i(\RInd^{G^{(1)}}_{B^{(1)}} \tau_{\le 1}\RHom_{B_1}(k,C^\bullet_{jk}))=0\qquad \text{for $i\ge 2$.}
\end{equation}
First note that
$ \tau_{\le 1} \RHom_{B_1}(k,L_{jk})\r \tau_{\le
  1}\RHom_{B_1}(k,C^\bullet_{jk}) $
is a quasi-is\-omorphism; as the cone of the morphism
$\RHom_{B_1}(k,L_{jk})\r \RHom_{B_1}(k,C^\bullet_{jk})$ is given by
$\RHom_{B_1}(k,\tau_{\ge 1} C^\bullet_{jk})$ and as $\tau_{\ge 1}C^{\bullet}_{jk}$ has
cohomology only in degree $k> j\geq 1$ (i.e.\ in degree $\ge 2$), the
morphism $\RHom_{B_1}(k,L_{jk})\r \RHom_{B_1}(k,\Cjk)$ induces an
isomorphism on cohomology in degrees $\le 1$.  So it remains to show
\[
H^i(\RInd^{G^{(1)}}_{B^{(1)}} \tau_{\le 1}\RHom_{B_1}(k,L_{jk}))=0\qquad \text{for $i\ge 2$.}
\]
In other words by the hypercohomology spectral sequence we have to verify the vanishing $R^1\Ind_{B^{(1)}}^{G^{(1)}}H^1(B_1,L_{jk})=0$. This holds by \eqref{eq:LjkvsLjkt}, 
since the cohomology groups $H^1(B_1,\sLjkt)\cong\tH^1(B_1,\sLjkt)$ only have weights $\geq 0$ by Corollary \ref{cor:Kweights}. 
 \end{proof}

\begin{corollary}
\label{prop:BG}
The object $\tau_{\geq 1}\RHom_{G_1}(k,\gCjk)$ is formal
in $D(G^{(1)},S^p)$. Moreover for $l\ge 1$ we have
\[
H^l(G_1,\gCjk)\cong 
\begin{cases}
\Ind_{B^{(1)}}^{G^{(1)}}H^l(B_1,C^\bullet_{jk})&
\text{if $k\leq j$,}\\
\Ind_{B^{(1)}}^{G^{(1)}}H^{l+1}(B_1,C^\bullet_{j,k+1})&
\text{if $k\geq j$}.
\end{cases}
\]
\end{corollary}
\begin{proof}
Let $\epsilon=[k>j]$ (see \eqref{eq:logical}).
We can conclude from Proposition \ref{prop:BGpre} that 
\begin{align*}
\tau_{\geq 1}\RHom_{G_1}(k,\gCjk) &\cong \RInd_{B^{(1)}}^{G^{(1)}}\tau_{\geq 1+\epsilon}\RHom_{B_1}(k,C^\bullet_{j,k+\epsilon})[\epsilon] \\
& \cong \bigoplus_{l\geq 1+\epsilon}\RInd_{B^{(1)}}^{G^{(1)}}H^l(B_1,C^\bullet_{j,k+\epsilon})[-l+\epsilon] \\
& \cong \bigoplus_{l\geq 1+\epsilon}\Ind_{B^{(1)}}^{G^{(1)}}H^l(B_1,C^\bullet_{j,k+\epsilon})[-l+\epsilon]
\end{align*}
where we used Proposition \ref{prop:outform} for the second isomorphism, and the positivity of the weights of $H^l(B_1,\Cjk)$ (see Lemma \ref{lem:qired} and Proposition \ref{prop:weights}) for the third isomorphism. 
\end{proof}

In the proof of Proposition \ref{prop:BGpre} we have used the following lemma. 
\begin{lemma}
\label{lem:trunc}
Let $F:\Ascr\r\Bscr$ be an exact, left t-exact, functor between triangulated categories with t-structure. 
A morphism $Y\r FX$ in $\Bscr$ for $Y\in \Ob(\Bscr)$ and $X\in \Ob(\Ascr)$ extends uniquely to a morphism of distinguished triangles in $\Bscr$
\begin{equation}
\label{eq:basicdiagram}
\xymatrix{
\tau_{\le 0} Y\ar[r] \ar[d]& Y \ar[r]\ar[d] & \tau_{\ge 1} Y\ar[r]\ar[d]&\\
F\tau_{\le 0} X\ar[r]&FX\ar[r]&F\tau_{\ge 1}X\ar[r]&
}
\end{equation}
If 
the middle arrow is an isomorphism  and $F\tau_{\le 0}X\in \Bscr_{\le 0}$ then this is an isomorphism of distinguished triangles.
\end{lemma}
\begin{proof} Since  $\tau_{\ge 1} X[1] \in \Ascr_{\geq 0}$, we find $F(\tau_{\geq 1}X)\in \Bscr_{\geq 0}[-1]=\Bscr_{\geq 1}$, using the left t-exactness of $F$. Since moreover $\tau_{\le 0} Y\in \Bscr_{\le 0}$, we see that $\Bscr(\tau_{\le 0} Y,F\tau_{\ge 1} X)=0$, $\Bscr(\tau_{\le 0} Y,F\tau_{\ge 1} X[-1])=0$, from which we deduce the existence and uniqueness of the leftmost arrow in a standard way from the properties of triangulated categories.
The rightmost arrow then exists by the TR3 axiom. Its uniqueness follows from $\Bscr(\tau_{\le 0}Y[1],F\tau_{\ge 1} X)=0$.

Now assume that the middle arrow is an isomorphism  and $F\tau_{\le 0}X\in \Bscr_{\le 0}$. The conclusion of the lemma
then follows from the already mentioned fact  $F\tau_{\ge 1} X\in \Bscr_{\ge 1}$ which implies that the lower row in \eqref{eq:basicdiagram} is the
tautological distinguished triangle associated with a t-structure. Hence it must be the same as the tautological distinguished triangle
in the upper row.
\end{proof}

\section{Decomposition of $K_{jk}^{G_1}$}
Here we give the list of indecomposable summands  of the $(G^{(1)},S^p)$-module $K_{jk}^{G_1}$. We first give all possible summands, for which we only need Proposition \ref{prop:weights} (avoiding 
the technical Section \ref{sec:extra}). In Proposition \ref{prop:Kjkrefine} we  determine which summands occur effectively in $K_{jk}^{G_1}$ (applying then Proposition \ref{prop:refKjkdecomposition}). 
\begin{proposition}\label{prop:Kjkdecomposition}
As graded $(G^{(1)},S^p)$-module $K_{jk}^{G_1}$ has   up to (possibly zero) multiplicity and degree the following indecomposable summands:
\begin{equation}\label{eq:list}
\{K_{rs}^{\Fr}\mid 1\leq r,s\leq n-3\}\cup \{T(l)^{\Fr}\otimes_k S^p\mid 0\leq l\leq n-3\}.
\end{equation}
\end{proposition}

\begin{proof} Let $C(k)$ be a left
projective $G_1$-resolution of $k$ consisting of $G$-tilting modules
(see \cite[\S\ref{sec:proofoffinaldecompsub}]{FFRT1}). 
We take $M^\bullet=\Tot(\Hom_{G_1}(C(k),\gCjk))$, and 
verify the hypotheses of Theorem \ref{prop:main}:
\begin{enumerate}
\item follows from \cite[Lemma 13.1]{FFRT1} by the choice of $C(k)^i$ and the choice of $\gCjk$,
\item follows for $H^0(M^\bullet)=K_{jk}^{G_1}$ by \cite[Corollary 2.2]{MR1202803} since  $K_{jk}$ has good filtration by Proposition \ref{prop:propKjk}(2), and for $i>0$ from \eqref{H(tildeCjk)} below and 
the fact that $M_r^{\Fr}$ for $r\geq 1$   have good filtration,
\item 
follows again from  \cite[Corollary 2.2]{MR1202803} using \cite[\eqref{Sdual}]{FFRT1} and Proposition \ref{prop:propKjk}(1)  for the isomorphism 
\begin{equation}\label{dual:KjkG1}
\Hom_{S^p}(K_{jk}^{G_1},S^p)\cong K_{n-j-2,n-k-2}^{G_1}(n(2p-1))\otimes_k \wedge^n F^*,
\end{equation}
\item follows from Corollary \ref{prop:BG}.
\end{enumerate} 

To obtain the summands  we  need by Theorem \ref{prop:main}  to compute (the $(l+1)$-th stable syzygy of) $H^l(G_1,\gCjk)$ for $l\geq 1$. 
Put $\epsilon=[k>j]$ (see \eqref{eq:logical}). Then we find by Corollary \ref{prop:BG} and Lemma \ref{lem:qired}:
\begin{align}\label{eq:cohg1cjk}
H^l(G_1,\gCjk) &= \Ind_{B^{(1)}}^{G^{(1)}}H^{l+\epsilon}(B_1,C_{j,k+\epsilon}^\bullet) \\\nonumber
&= \bigoplus_{t \in [0,p-1]^n}\Ind_{B^{(1)}}^{G^{(1)}}(H^{l+\epsilon}(B_1,C_{j,k+\epsilon}^{(t)}) \otimes_k S_+^p).
\end{align}
By Proposition \ref{prop:weights} we then find 
\begin{align}\label{H(tildeCjk)}
H^l(G_1,\gCjk)&\cong \bigoplus_{t \in [0,p-1]^n} \Ind_{B^{(1)}}^{G^{(1)}} \bigoplus_{1\leq r\leq\max\{l,n-3\}}((r\omega)\otimes_k S^p_+)^{\oplus n_{rt}}\\\nonumber
&\cong \bigoplus_{1\leq r\leq\max\{l,n-3\}}\Ind_{B^{(1)}}^{G^{(1)}}((r\omega)\otimes_k S^p_+)^{\oplus n_r}\\\nonumber
&\cong \bigoplus_{1\leq r\leq\max\{l,n-3\}} (M_r^{\Fr})^{\oplus n_r}
\end{align}
for some $n_{rt},n_r\geq 0$. 
Using Proposition \ref{prop:propKjk}\eqref{item:stabletilde} we then compute the  $(l+1)$-th (stable) syzygies of $M_r^{\Fr}$ (which give us $K_{rl}^{\Fr}$ if $l\leq n-3$ and $0$ otherwise). 
 For the tilt-free summands we  invoke \cite[Proposition A.1]{FFRT1} (see also Remark 14.3 in loc.cit.).
\end{proof}

\begin{proposition}\label{prop:Kjkrefine}
The summand  $K_{rl}^{\Fr}$ 
  occurs 
  with nonzero multiplicity in $K_{jk}^{G_1}$ if and only if 
  \begin{equation}
\label{eq:intervals1}
r\in
\begin{cases}
[l-k+m_{lk},l] &\text{if $l>k$},\\
[1,n-3]&\text{if $l=k$},\\
[l,l-k+n-2-m_{lk}] &\text{if $l<k$},
\end{cases}
\end{equation}
where $m_{lk}$ is as in \eqref{eq:mlk1}.
\end{proposition}
\begin{proof}
We invoke Proposition \ref{prop:refKjkdecomposition} in the first line of \eqref{eq:cohg1cjk}. Then we obtain 
that $H^l(G_1,\gCjk)$ is a direct sum of  $(M_r^{\Fr})^{\oplus n_r}$ for $n_r\geq 1$ and $r$ as in Proposition \ref{prop:refKjkdecomposition} 
for $k,l$ replaced by $k+\epsilon,l+\epsilon$. 
We thus obtain
\[
r\in
\begin{cases}
[(l+\epsilon)-(k+\epsilon)+m_{l+\epsilon,k+\epsilon},(l+\epsilon)-\epsilon] &\text{if $l+\epsilon>k+\epsilon$},\\
[1,n-3]&\text{if $l+\epsilon=k+\epsilon$},\\
[(l+\epsilon)-\epsilon,(l+\epsilon)-(k+\epsilon)+n-2-m_{l+\epsilon,k+\epsilon}] &\text{if $l+\epsilon<k+\epsilon$},
\end{cases}
\]
which is clearly equivalent with \eqref{eq:intervals1}.

Taking $(l+1)$-th stable syzygies we then obtain $K_{rl}^{\Fr}$ for $r,l$ as in 
the  statement of the proposition. 
\end{proof}

\begin{remark}\label{rmk:Kjkdegrees}
Let $1\leq l\leq n-3$ and let $r$ be as in \eqref{eq:intervals1}. Then, as graded $S^p$-modules, $K_{rl}^{\Fr}(-d_t)$ is a summand of $K_{jk}^{G_1}$ for some $t\in [0,p-1]^n$ such that $q_{jt}^\ell$ is defined and $q_{jt}^\ell+\ell\equiv r \,(2)$. 

This follows immediately by the proof of Proposition \ref{prop:Kjkrefine} observing the degrees in the first line of \eqref{eq:cohg1cjk}. By Remark \ref{rmk:degrees}, $H^{l}(G_1,\gCjk)$ then contains as a direct summand  $M_r^{\Fr}(-d_t)$ for some $t\in [0,p-1]^n$, which gives the stable syzygy $K_{rl}^{\Fr}(-d_t)$. 
\end{remark}

\section{Proofs of the main results}\label{sec:proofsmain}
\begin{proof}[Proof of Theorem \ref{thm:higherFr}]
The theorem follows by combining Propositions \ref{prop:mocdecomposition}, 
\ref{prop:Kjkdecomposition}
iteratively (see \S\ref{sec:proofoutline}), which already give us all possible summands. Proposition \ref{prop:mocdecomposition} applied with $j=0$ further implies that  the summands $T(s)^{\Fr}\otimes_k S^p$ for $0\leq s\leq n-3$ and $K_{ll}^{\Fr}$ for $1\leq l\leq n-3$ actually appear,  
Proposition \ref{prop:Kjkrefine} further 
assures that a summand $K_{sl}^{\Fr}$ for all $1\leq s\leq n-3$ occurs in $K_{ll}^{G_1}$.
 \end{proof}

 \begin{proof}[Proof of Theorem \ref{thm:decR}]
 The decomposition is an immediate consequence of Theorem \ref{thm:higherFr}. By Corollary \ref{cor:decomp}, one only needs to apply $(-)^{G^{(r)}}$ to the summands in the decomposition of $S^{G_r}$.   
 \end{proof}

 \begin{proof}[Proof of Theorem \ref{intro:main-app}]
The theorem is an immediate corollary of Theorem \ref{thm:decR} (and Theorem \ref{thm:old}).
 \end{proof}

\begin{proof}[Proof of Theorem \ref{thm:diff}]
According to \cite[Theorem 6]{Hashimoto}, $R$ is strongly $F$-regular. By Theorem \ref{intro:main-app}, $R$ satisfies FFRT. Thus, the theorem follows from \cite[Theorem 4.2.1]{MR1444312}. 
\end{proof}

\section{Pushforwards of the structure sheaf on the Grassmannian}\label{sec:grass}
\subsection{Notation}
\label{ssec:induceda}
In this section we introduce some general notation with regard to Grassmannians.
Let $F,V$ be respectively vector spaces of dimension $n$, $l$ with $n\ge l+1$. We let $\GG$ be the Grassmannian of $l$-dimensional quotients of $F$.
We fix a surjection $[F\r V]\in \GG$ and we use it to write $\GG$ as $\GL(F)/P$ 
where $P$ is the parabolic subgroup of $\GL(F)$ which is the stabilizer of the corresponding flag. Denote by $\W$ the kernel of the chosen surjection $F \to V$. The parabolic~$P$ has $L=\GL(V) \times \GL(\W)$ as a Levi subgroup, and (abelian) radical $R=\Hom(V,\W)$. 

For $U$ a $P$-representation we write $\Lscr_{\GG}(U)$ for the $\GL(F)$-equivariant vector
bundle on~$\GG$ whose fiber in $[P]$ is $U$. If $U$ is an $L$-representation then we consider it as a $P$ representation via the map $P\r L$.

Let 
\begin{equation}
\label{eq:tautseqa}
0 \to \Rscr \to F\otimes_k \Oscr_{\GG} \to \Qscr \to 0
\end{equation} 
be the tautological exact sequence on $\GG$, where $\Qscr$ is the universal quotient bundle, and $\Rscr$ is the universal  subbundle. We then have $\Qscr=\Lscr_{\GG}(V)$, $\Rscr=\Lscr_{\GG}(\W)$. Also, put $\Oscr(1):=\wedge^l\Qscr\cong \wedge^nF\otimes_k (\wedge^{n-l}\Rscr)^\vee$.

\subsection{Associated sheaves}\label{sec:sheafa} 
Remember from \cite[\S \ref{sec:homogeneous}]{FFRT1} that it is useful to consider a $\GL(V)$-representation as a graded $\SL(V)$-representation via
the covering  $m:G_m\times \SL(V)\r \GL(V):(\lambda,g)\mapsto \lambda^{-1}g$. We recall the following proposition from \cite{FFRT1} in order to fix some more notation.
\begin{proposition}\cite[Proposition \ref{prop:diag2}]{FFRT1}
\label{prop:sheaves}
Let $W=F\otimes_k V^\ast$ and put
 $S=\Sym W$, $R=S^{\SL(V)}$ where $R,S$ are both considered as $\NN$-graded rings. 
Then we have a commutative diagram of functors
\[
\xymatrix{
\Rep(\GL(V)) \ar[rr]^{\Lscr_{\GG}(-)} \ar[d]_{m^\ast}&& \coh(\GG) \ar[d]^{\Gamma_*(\GG,-)} \\
\Rep_{\gr}(\SL(V)) \ar[rr]_{M(-)^{[l]}} && \gr(R^{[l]})
}
\]
where 
$
\Gamma_\ast(\GG,\Fscr)=\bigoplus_{m\in \ZZ} 
\Gamma(\GG,\Fscr(m)), 
$
and $(-)^{[l]}$ denotes the $l$-Veronese.
\end{proposition}
We now assume that $\dim V=l=2$ so that $V\cong V^\ast$ as $\SL(V)$-representations and hence in Proposition \ref{prop:sheaves}
we have $S\cong \Sym(F\otimes_k V)$, which is our standard definition for $S$ (see \S\ref{sec:not2}).
This allows us to compare modules of covariants with homogeneous bundles, as in the following proposition. 
\begin{proposition}\cite[Corollary \ref{cor:alggeo}\eqref{ena}]{FFRT1}\label{prop:mocAG}
There is an isomorphism 
\[(M(S^iV)(i))^{[2]}\cong \Gamma_*(\GG,S^i\Qscr).\]
\end{proposition}

The next proposition now gives a geometric analogue of $K_{jk}^G$ needed for $\Fr^r_*\Oscr_\GG$.

\begin{proposition}\label{prop:alggeo}
Assume $1\le j,k\le n-3$.
The following equivariant $\Ext$-groups are one-dimensional 
\begin{align}
\label{eq:exta}\Ext^1_{\GG,\GL(F)}(\wedge^{k-1}\Rscr  \otimes_{\GG} S^{j-k-1}\Qscr(1),\wedge^k\Rscr \otimes_{\GG} S^{j-k}\Qscr)&\quad\text{if $k<j$},\\
\label{eq:extb}\Ext^1_{\GG,\GL(F)}(\wedge^{k}\Rscr \otimes_{\GG} D^{k-j}\Qscr(j-k),\wedge^{k+1}\Rscr \otimes_{\GG} D^{k-j-1}\Qscr(j-k))&\quad\text{if $k>j$}.
\end{align}
Let $\Kscr_{jk}$ be the extension corresponding to a nonzero element in (\ref{eq:exta},\ref{eq:extb}) if $j\neq k$ and also put $\Kscr_{jj}=\wedge^j \Rscr$.
Then there are isomorphisms 
\begin{equation}
\label{eq:uno}
(K_{jk}^{G}(j))^{[2]} \cong \Gamma_*(\GG,\Kscr_{jk})
\end{equation}
as $\GL(F)$-equivariant graded $R^{[2]}$-modules.
\end{proposition}
\begin{remark} We have used divided powers in \eqref{eq:extb} in order to have a characteristic free statement. However with our standing hypothesis
$p\ge \max\{3,n-2\}$ the divided powers are actually equal to the corresponding symmetric powers.
\end{remark}
\begin{proof}[Proof of Proposition \ref{prop:alggeo}]
For $u\in \ZZ$ put $\bar{u}=n-2-u$. The fact that \eqref{eq:exta} is one-dimensional if $k<j$ will be shown in Lemma \ref{lem:1dimext} below.
A straightforward verification yields
\begin{align*}
(\wedge^{k-1}\Rscr  \otimes_{\GG} S^{j-k-1}\Qscr(1))^\vee&=(\wedge^n F)^\ast\otimes_k \wedge^{\bar{k}+1}\Rscr\otimes_{\GG} D^{\bar{k}-\bar{j}-1}\Qscr(\bar{j}-\bar{k}+1),\\
(\wedge^k\Rscr \otimes_{\GG} S^{j-k}\Qscr)^\vee&=(\wedge^n F)^\ast\otimes_k \wedge^{\bar{k}}\Rscr\otimes_{\GG} D^{\bar{k}-\bar{j}} \Qscr(\bar{j}-\bar{k}+1).
\end{align*}
Hence the fact that \eqref{eq:extb} is one-dimensional follows from  the corresponding claim for \eqref{eq:exta}. Moreover we obtain for\footnote{It is easy to see that this formula also holds for $j=k$ and since it is self dual, it also holds for $k>j$.} $k<j$
\[
\Kscr_{jk}^\vee=(\wedge^n F)^\ast\otimes_k \Kscr_{\bar{j}\bar{k}}(1).
\]
 Using Proposition \ref{prop:propKjk}\eqref{lem:Kjkdual} together
with a suitable version of  the ``descent lemma'' \cite[Lemma 4.1.3]{vspenko2015non} we find an analogous formula
\[
(K_{jk}^G)^\vee=(\wedge^nF)^\ast \otimes_k K_{\bar{j}\bar{k}}^G(n).
\]
We claim that it is sufficient to establish \eqref{eq:uno} for $k\leq j$. 
Indeed by dualizing \eqref{eq:uno} in the case $k<j$ we get 
\[
(\wedge^n F)^*\otimes_k (K_{\bar{j}\bar{k}}^{G}(n-j))^{[2]} \cong (\wedge^n F)^*\otimes_k \Gamma_*(\GG,\Kscr_{\bar{j}\bar{k}}(1)),
\]
where we used that $\Gamma_\ast(\GG,\Kscr_{jk})^\vee=\Gamma_\ast(\GG,\Kscr^\vee_{jk})$  by the easily verified fact that $\Gamma_\ast(\GG,-)$ preserves reflexivity. Hence
\[
(K_{\bar{j}\bar{k}}^{G}(\bar{j}))^{[2]} \cong \Gamma_*(\GG,\Kscr_{\bar{j}\bar{k}}),
\]
which establishes the formula \eqref{eq:uno} for $k> j$.

\medskip

So now we assume that $k\leq j$.  
Let $\Cscr_j=S^j(F\otimes_k \Oscr\to \Qscr)$ (with $\Qscr$ in degree $0$) and denote $\bar{\Kscr}_{jk}=Z^{-k}\Cscr_j$. By Lemma \ref{lem:compAG} below  
$\Gamma_*(\GG,\bar{\Kscr}_{jk})=(K_{jk}^G(j))^{[2]}$.

It remains to show that $\Kscr_{jk}=\bar{\Kscr}_{jk}$; i.e. that $\bar{\Kscr}_{jk}$ is a $\GL(F)$-equivariant extension of  $\wedge^{k-1}\Rscr\otimes_\GG S^{j-k-1}\Qscr(1)$ by $\wedge^k\Rscr\otimes_\GG S^{j-k}\Qscr$. In the case $k=j$ this  yields $\bar{\Kscr}_{jk}=\wedge^j\Rscr=\Kscr_{jk}$ and in the case $k<j$ the extension must be non-trivial as $K_{jk}^G(j)$ is indecomposable (using \eqref{eq:reflr} and Proposition \ref{prop:propKjk}\eqref{item:Kjkinde}).

Note that $\Cscr_j$ (and consequently $\sigma_{\geq -k}\Cscr_j$) has a natural filtration coming from the extension of complexes
\[
0\to (\Rscr\to 0)\to(F\otimes_k\Oscr\to \Qscr)\to (\Qscr\xrightarrow{\id}\Qscr)\to 0.
\] 
We need to compute $\bar{\Kscr}_{jk}=H^{-k}(\sigma_{\geq -k}\Cscr)$. 
Thus, we look at the cohomology of the associated graded complex of $\sigma_{\geq -k}\Cscr$. The associated graded complex is of the form:
\begin{align*}
\wedge^{k}\Rscr\otimes_\GG S^{j-k}\Qscr&[k]\\
\wedge^{k-1}\Rscr\otimes_\GG (\Qscr\otimes_\GG S^{j-k}\Qscr\to S^{j-k+1}\Qscr)&[k-1]\\
\wedge^{k-2}\Rscr\otimes_\GG (\wedge^2\Qscr\otimes_\GG S^{j-k}\Qscr\to\Qscr\otimes_\GG S^{j-k+1}\Qscr\to S^{j-k+2}\Qscr)&[k-2]\\
\vdots&\\
\wedge^2\Qscr\otimes_\GG S^{j-2}\Qscr\to \Qscr\otimes_\GG S^{j-1}\Qscr\to S^j\Qscr&
\end{align*}
All rows, except the first two, consist of acyclic complexes. The
cohomology of the first (resp. second) row equals
$\wedge^k\Rscr\otimes_\GG S^{j-k}\Qscr$ (resp.
$\wedge^{k-1}\Rscr\otimes_\GG \wedge^2\Qscr\otimes_\GG
S^{j-k-1}\Qscr$).
Thus, $\bar{\Kscr}_{jk}$ is an extension of the desired sheaves.
\end{proof}
We used the following lemmas.

\begin{lemma}\label{lem:compAG}
Let $\Cscr_j=S^j(F\otimes_k \Oscr\to \Qscr)$, with $\Qscr$ in degree $0$. Then we have a $\GL(F)$-equivariant isomorphism of complexes of graded $R^{[2]}$-modules
\begin{equation}
\label{eq:GammaC}
\Gamma_*(\GG,\Cscr_j) \cong \sigma_{\geq -j}(\tilde{C}_j^G(j)^{[2]}),
\end{equation} 
where $\tilde{C}_j^G$ is obtained by applying $(-)^G$ termwise to $\tilde{C}_j$
(the resolution \eqref{biresolution01} of $M_j$, see \S\ref{sec:ind-B-G}).

For $k\le j$ put $\bar{\Kscr}_{jk}=Z^{-k}\Cscr_j$. Then we have 
\begin{equation}
\label{eq:KK}
\Gamma_\ast(\GG,\bar{\Kscr}_{jk})\cong(K_{jk}^G(j))^{[2]}.
\end{equation}
\end{lemma} 

\begin{proof}
We first show that applying $\Gamma_*(\GG,-)$ to $\Cscr_j$ gives a complex with a single nonzero cohomology group in degree $-j$. We put $\Gamma_{\ast,\ge 0}(\GG,\Fscr)=
\bigoplus_{m\ge 0}\Gamma(\GG,\Fscr(m))$.
As~$\Cscr_j$ has only a single non-vanishing cohomology group, equal to $\wedge^j\Rscr$, in degree $-j$ we have 
\begin{equation}
\label{eq:Rsheaf}
R\Gamma_{\ast,\ge 0} (\GG,\Cscr_j)\cong R\Gamma_{\ast,\ge 0}(\GG,\wedge^j \Rscr)[j]
\end{equation}
and furthermore we also have
\[
\Gamma_{\ast}(\GG,\Cscr_j)=\Gamma_{\ast,\ge 0}(\GG,\Cscr_j)
\]
since $H^0(\GG,S^j\Qscr(l))=0$ for $l<0$ by \cite[Proposition II.2.6]{jantzen2007representations}. 
Hence it suffices to prove that both sides in \eqref{eq:Rsheaf} consist of terms which are acyclic for $\Gamma_{\ast,\ge 0}(\GG,-)$; i.e. 
\begin{align}\label{eq:pourR}
\forall l\ge 0:H^{>0}(\GG,\wedge^j\Rscr(l))=H^{>0}(\GG,\Lscr_{\GG}(l,l;1^j,0^{n-j-2}))=0,\\\label{eq:pourQ}
\forall l\ge 0:H^{>0}(\GG,S^j\Qscr(l))=H^{>0}(\GG,\Lscr_{\GG}(j+l,l;0^{n-2}))=0,
\end{align}
where in the notation $\Lscr_{\GG}(\cdots;\cdots)$ we describe $L=\GL(V)\times\GL(\W)$-representations by their highest weights.
The equalities \eqref{eq:pourR} for $l>0$ and \eqref{eq:pourQ} for $l\geq 0$ follow from \cite[Proposition II.4.5]{jantzen2007representations}, \eqref{eq:pourR} for $l=0$ follows from \cite[Proposition II.5.4.a]{jantzen2007representations}. 
 
So $\Gamma_*(\GG,\Cscr_j)$ is exact (except in degree $-j$). It is easy to see that the differentials in $\Cscr_j$ are nonzero and 
as $\Oscr_{\GG}(1)$ is ample the same is true for the differentials in $\Gamma_\ast(\GG,\Cscr_j)$. 
Furthermore the reader may verify  using Proposition \ref{prop:mocAG}
that $\Gamma_\ast(\GG,\Cscr_j)$  has the same terms as $(\sigma_{\geq -j}\tilde{C}_j^G(j))^{[2]}$. In addition one may check that $\sigma_{\geq -j}\tilde{C}_j^G(j)$ lives in even degrees so
that no information is lost by passing to the 2-Veronese.
Therefore \eqref{eq:GammaC} follows by uniqueness of the $\GL(F)$-equivariant maps in $\tilde{C}_j^G$ using the graded analogue of the ``descent lemma'' \cite[Lemma 4.1.3]{vspenko2015non} and the fact that one can easily see that
the differentials in $\tilde{\Cscr}_j$ are determined by $G\times \GL(F)$-equivariance and compatibility with the grading. 

\medskip

Finally \eqref{eq:KK} follows by applying $Z^{-k}$  to \eqref{eq:GammaC} taking into account that  $\Gamma_*$ and $(-)^G$ are left exact. 
\end{proof}
\begin{lemma}\label{lem:1dimext} 
Assume $1\leq k<j\leq n-3$. Then
\[
\dim\Ext^1_{\GG,\GL(F)}(\wedge^{k-1}\Rscr \otimes_{\GG} S^{j-k-1}\Qscr(1),\allowbreak \wedge^k\Rscr \otimes_{\GG} S^{j-k}\Qscr)=1.
\] 
 \end{lemma}
 \begin{proof}
Denote
 \begin{align*}
 M&=\wedge^{k-1}\W \otimes_k S^{j-k-1}V \otimes_k \wedge^2V,\\
 N&=\wedge^{k}\W \otimes_k S^{j-k}V.
 \end{align*}
 Since the category of homogeneous vector bundles on $\GG$ (with $\Gl(F)$-equivariant maps) is equivalent to the category $\Rep(P)$ of finite dimensional $P$-representations, we are reduced to showing that $\Ext^1_P(M,N)$ is one-dimensional. 
 
   The Lyndon-Hochschild-Serre spectral sequence \cite[I.6.6]{jantzen2007representations} gives
 \begin{equation}
 \label{eq:lhs}
 E_2^{pq}=\Ext^p_L(k,\Ext^q_R(M,N)) \Rightarrow \Ext^{p+q}_P(M,N).
 \end{equation}
 Using the $\ZZ^2$-grading coming from $\GG_m \times \GG_m \subset L$, one sees that (noting that $R$ acts trivially on $M,N$)
 \begin{align*}
 \Ext^p_L(k,\Hom_R(M,N))&=0
 \end{align*} 
 for $p\geq 0$, since $M$ (respectively $N$) has degree $(j-k+1,k-1)$ (respectively $(j-k,k)$). From the spectral sequence \eqref{eq:lhs} one then obtains
 \begin{align*}
 \Ext^1_P(M,N)&=\Hom_L(k,\Ext^1_R(M,N)) \\
 &=\Hom_L(M,H^1(R,k) \otimes_k N).
 \end{align*}
 Now by \cite[I.4.21]{jantzen2007representations}, $H^1(R,k)=\sum_{i=1}^{2(n-2)}\sum_{r=0}^{\infty}kT_i^{p^r}$, where the $T_i$ form generators for $\Sym(\Hom(\W,V))$. Using the grading induced by $\GG_m 
\times \GG_m \subset L$, it suffices to show that 
 \begin{equation}
 \label{eq:homspace}
\dim \Hom_L(M,\Hom(\W,V) \otimes_k N)=1
 \end{equation}
 since the $T_i^{p^r}$ for $r>0$ have degree $(p^r,-p^r)$. As $L=\GL(V) \times \GL(\W)$, \eqref{eq:homspace} now follows immediately since 
 \begin{align*}
 \dim \Hom_{\GL(V)}(S^{j-k-1}V \otimes_k \wedge^2 V,V\otimes_k S^{j-k}V)&=1,\\
\dim \Hom_{\GL(\W)}(\wedge^{k-1}\W ,\W^\ast\otimes_k\wedge^k\W)&=1,
 \end{align*}
by Pieri's formulas (which become filtrations in characteristic $p$ \cite{MR931171}) and e.g.\ \cite[Proposition \ref{basicdeltanabla}]{FFRT1}.
 \end{proof}

\subsection{Decomposition of $\Fr^r_*\Oscr_\GG$}\label{subsec:FrOscr}
The exact decomposition of $\Fr_*\Oscr_\GG$ can be found in \cite[Theorem \ref{thm:grass}]{FFRT1}, but here we will recall only which summands occur. 

\begin{proposition}\cite[Corollary \ref{cor:indecgrass}]{FFRT1}\label{cor:indecgrasstwistf}
  Up to (nonzero) multiplicities and twists, the indecomposable summands of $\Fr_*\Oscr_\GG$ for $p\geq n-1$ 
 are 
\[
\{\Kscr_{j}\mid 1\leq j\leq n-3\}\cup\{S^l\Qscr\mid 0\leq l\leq n-3\},
\]
where $\Kscr_j=\Kscr_{jj}=\wedge^j \Rscr$.
\end{proposition}

\begin{remark}
The assumption $p\geq n-1$ is necessary for $S^{n-3}\Qscr$ to appear (see \cite[Corollary \ref{cor:indecgrass}]{FFRT1}). 
\end{remark}

For $r\geq 2$ we have an analogous decomposition of $\Fr^r_*\Oscr_\GG$ up to multiplicities and twists.

\begin{theorem}\label{thm:FrOscr}
Up to  (possibly zero) multiplicities and twists the indecomposable summands of $\Fr^r_*\Oscr_\GG$ for $r\geq 2$ are:
\[
\{\Kscr_{jk}\mid 1\leq j,k\leq n-3\}\cup\{S^l\Qscr\mid 0\leq l\leq n-3\}.
\]
If $p\geq n$ then all the summands (up to twists) appear with nonzero multiplicities.
\end{theorem}

\begin{proof} 
Since $\Proj(R^{p^r})\cong \Proj((R^{p^r})^{[2p^r]})$, the decomposition of $\Fr^r_*\Oscr_\GG$ is obtained by decomposing $R^{[2p^r]}$ in $\Proj((R^{p^r})^{[2p^r]})$. The theorem follows up to possibly zero multiplicities from Theorem \ref{thm:decR}, Proposition \ref{prop:mocAG} and Proposition \ref{prop:alggeo}.  

It remains to argue for nonzero multiplicities in the case $p \geq n$. 
First consider $S^l\Qscr$ for $0 \leq l \leq n-3$. By Proposition \ref{cor:indecgrasstwistf}, $\Oscr$ is a direct summand of $\Fr_*\Oscr$. By iterating, $\Oscr$ is a direct summand of $\Fr^{r-1}_*\Oscr$, and thus, $\Fr^r_*\Oscr=\Fr_*\Fr^{r-1}_*\Oscr$ has $\Fr_*\Oscr$ as a direct summand. Using Proposition \ref{cor:indecgrasstwistf} again, some twist of $S^l\Qscr$ occurs with nonzero multiplicity in $\Fr_*\Oscr$, and hence in $\Fr^r_*\Oscr$. 

Now we consider summands of the form $\Kscr_{jk}$. 
Let $1\leq j\leq n-4$.  By Lemmas \ref{lem:specsumKjk}, \ref{lem:speccumKjkdeg} below and \eqref{eq:uno}, $\Fr_*(\Kscr_{jj}(-j)\oplus \Kscr_{jj}(-j-1))$ has as summands (twists of) $\Kscr_{sj}$, for $1\leq s\leq n-3$, and (twists of) $\Kscr_{u,n-3}$, for $n-3-j+m_{n-3,j}\leq u\leq n-3$. Varying $1\leq j\leq n-4$, we obtain (twists of) $\Kscr_{sl}$, $1\leq s,l\leq n-3$, in 
\begin{equation*}\label{eq:Kjjj}
\bigoplus_{1\leq j\leq n-4} \Fr_*(\Kscr_{jj}(-j)\oplus \Kscr_{jj}(-j-1)).
\end{equation*}
By a slightly tedious verification using \cite[Corollary \ref{cor:indecgrass}]{FFRT1}, $\Kscr_{jj}(-j)\oplus \Kscr_{jj}(-j-1)$, $1\leq j\leq n-4$, is a summand of $\Fr_*\Oscr_\GG$ (using the assumption $p\geq n$), 
therefore (twists of) $\Kscr_{sl}$, $1\leq s,l\leq n-3$, occur as summands in $\Fr^2_*\Oscr_\GG$. 
For $r>2$ we proceed inductively as above. 
\end{proof}

The following lemmas were used in the proof of Theorem \ref{thm:FrOscr}.  Lemma \ref{lem:specsumKjk} is just an extraction from Proposition \ref{prop:Kjkrefine}, while Lemma \ref{lem:speccumKjkdeg} takes care of the degrees. 
The summands  of the $R^p$-module $(K_{jk}^{G}(j))^{[2]}=\Gamma_*(\GG,\Kscr_{jk})$ (see \eqref{eq:uno}) are only visible in the decomposition of $\Fr_*\Kscr_{jk}$ if 
their degrees are divisible by $2p$ (c.f. the proof of Theorem \ref{thm:FrOscr}). 
Divisibility by $p$ in some form has been used throughout the paper in the embodiment of $q_{jt}^\ell$, which will also be employed in the proof of Lemma \ref{lem:speccumKjkdeg}. 

\begin{lemma}\label{lem:specsumKjk}
If $1\leq j\leq n-4$, then $(G^{(1)},S^p)$-modules $K_{sj}^{\Fr}$, $1\leq s\leq n-3$, $K_{u,n-3}^{\Fr}$, $n-3-j+m_{n-3,j}\leq u\leq n-3$, occur as summands of $K_{jj}^{G_1}$. 
\end{lemma}

\begin{proof}
This is an immediate corollary of Proposition \ref{prop:Kjkrefine}.
\end{proof}

\begin{lemma}\label{lem:speccumKjkdeg}
Let $K_{sl}^{\Fr}$ be a summand of $K_{jk}^{G_1}$. Then there exists an even $d\in \NN$ such that, as graded modules, 
$(K_{sl}^G(s))^{\Fr}(-pd)$ is a summand of $(K_{jk}^G(j))(-2j)$ or a summand of $(K_{jk}^G(j))(-2j-2)$.
\end{lemma}

\begin{proof}
By Remark \ref{rmk:Kjkdegrees}, $K_{sl}^{\Fr}(-d_t)$ for some $t\in [0,p-1]^n$  is a summand of $K_{jk}^{G_1}$ such that $q_{jt}^\ell$ is defined and $q_{jt}^\ell+\ell\equiv s \,(2)$. 
Applying the definition of $q_{jt}^\ell$ (i.e. when defined $pq_{jt}^0=j+d_t$,  $pq_{jt}^1=j+d_t-p+2$), we obtain that in the case $\ell=0$ 
\begin{multline}\label{eq:128a}
(K_{sl}(s))^{\Fr}(-p(s+q_{jt}^0))=(K_{sl})^{\Fr}(-pq_{jt}^0)=(K_{sl})^{\Fr}(-j-d_t)\\ <_{\oplus} K_{jk}^{G_1}(-j) =(K_{jk}^{G_1}(j))(-2j)
\end{multline}
and in the case $\ell=1$
\begin{multline}\label{eq:128b}
(K_{sl}(s))^{\Fr}(-p(s+1+q_{jt}^1))=(K_{sl})^{\Fr}(-p(1+q_{jt}^1))=(K_{sl})^{\Fr}(-j-2-d_t)
\\<_\oplus K_{jk}^{G_1}(-j-2)=(K_{jk}^{G_1}(j))(-2j-2).
\end{multline}
Since $s+q_{jt}^0,s+1+q_{jt}^1$ are even, the lemma follows by applying $(-)^{G^{(1)}}$ to \eqref{eq:128a},\eqref{eq:128b}. 
\end{proof}

\subsection{Proof of Theorem \ref{intro:FrOscr}}\label{subsec:GFFRT}
The theorem 
 follows from Theorem \ref{intro:main-app} combined with Lemma \ref{lem:(G)FFRT} below (since $R$ is Gorenstein by \cite[Proposition \ref{canonical}]{MR611465}).

\begin{lemma}\label{lem:(G)FFRT}
Let $R$ be a connected graded Gorenstein noetherian reduced ring. If $R$ satisfies FFRT then $X=\Proj(R)$ satisfies GFFRT.
\end{lemma} 

\begin{proof}
The property FFRT ensures that the number of higher Frobenius summands of $\Fr^r_*\Oscr_X$ up to twists is finite. Hence, it remains to argue that the summands of $\Fr^r_*\Oscr_X$ can appear with only finitely many twists.  
It is equivalent to show that after dividing by $p^r$, the degrees of the (normalized) summands of $R$ as $R^{p^r}$-module live in a finite interval which is independent of $r$. 

We use a duality argument. Since $R$ is Gorenstein, $\omega_R=R(-e)$ for some $e\in \NN$. By the adjunction formula $\omega_R\cong\Hom_{R^{p^r}}(R,\omega_{R^{p^r}})$, from which we obtain
\begin{equation}\label{eq:duality11}
R(e(p^r-1))\cong\Hom_{R^{p^r}}(R,R^{p^r}).
\end{equation}
Assume that $M$ is normalized (i.e. lives in degrees $\geq 0$ and $M_0\neq 0$) and  that $M^{\Fr^r}(-a)$ is an indecomposable summand of $R$ as graded $R^{p^r}$-module, for some $a\geq 0$ (as $M$ is normalized and $R$ lives in degrees $\geq 0$). 
In particular, $M$ is a reflexive $R$-module. Then $M^\vee=N(c)$ for an  indecomposable graded normalized $R$-module $N$. 
By \eqref{eq:duality11}, we have
\[
N^{\Fr^r}(cp^r+a-e(p^r-1))=(M^{\Fr^r}(-a))^\vee(-e(p^r-1))<_{\oplus} R.
\]
Thus, $cp^r+a-e(p^r-1)\leq 0$ (arguing as above). 
Hence $0\leq a/p^r\leq e-c$.  
Consequently, after dividing by $p^r$, the degrees of every (normalized) summand of $R$ as $R^{p^r}$-module indeed live in a finite interval bounded independently of $r$. 
\end{proof}

 \section{Proof of the NCR property}
 \label{sec:ncrproperty}
 In this section we prove the following result which implies Theorem \ref{thm:ncr2} for $r\geq 2$ using Theorem \ref{thm:higherFr}. The case $r=1$ follows from \cite[Theorem 1.2]{FFRT1}.

 \begin{theorem}\label{thm:preNCR}
 The $R$-module 
 \[
 M=\bigoplus_{i=0}^{n-3} T\{i\}\oplus \bigoplus_{j,k=1}^{n-3} K\{j,k\}
 \]
 defines an NCR for $R$.
 \end{theorem}

 \begin{proof}
 The proof follows very similar lines as the proof of  \cite[Theorem 17.1]{FFRT1}, therefore we concentrate on the necessary modifications.

Let $N_{n-3,0}=\oplus_{j=0}^{n-3}T_i$ with $T_i=S^i V\otimes_k S$.
Then  $\Lambda_{n-3,0}=\End_{G,S}(N_{n-3,0})$ has finite global dimension by Proposition 17.2 in loc. cit. We add the different $K_{jk}$ to $N_{n-3,0}$ in reverse lexicographical order; i.e.
\begin{align*}
N_{1,k+1}&=N_{n-3,k}\oplus K_{1,k+1}&&\text{for $k=0,\dots,n-4$}, \\
N_{j+1,k}&=N_{j,k}\oplus K_{j+1,k} && \text{for $j=1,\dots, n-4$}. 
\end{align*}  
With Lemma \ref{lem:fundex2} below replacing \cite[Lemma \ref{lem:fundex}]{FFRT1} one shows, exactly as in loc. cit.  that $\Lambda_{jk}=\End_{G,S}(N_{jk})$ has finite global dimension
by employing induction on $(j,k)$ (with the reverse lexicographic ordering).
\end{proof}
Instead of using notations
as in \cite{FFRT1} we will use notations more in line with the ones we use in the current paper. 
Note that by our standing hypotheses on the characteristic $\gCj$ is tilt-free so we do not have to deal with tilt-free resolutions. 
We set (and recall)
\begin{align*}
\tilde{D}^\bullet_{jk}&=\sigma_{<-k}\gC^\bullet_j[-k-1],\\
\tilde{C}^\bullet_{jk}&=\sigma_{\ge -k}\gC^\bullet_j[-k]. 
\end{align*}
We have the following duality  (using $\bar{u}=n-2-u$)
\[
(\gC_j^\bullet)^\vee=\gC^\bullet_{\bar{j}}[-n+1](n)\otimes_k \wedge^n F^\ast
\]
(see the proof of Proposition \ref{prop:propKjk}\eqref{lem:Kjkdual}), 
which by some formal manipulations yields
\begin{equation}\label{eq:dualCD}
(\tilde{D}^\bullet_{jk})^\vee=\gC^\bullet_{\bar{j}\bar{k}}(n)\otimes_k \wedge^n F^\ast.
\end{equation}

\begin{lemma} \label{lem:fundex2} 
Let $L$ be either $T_l$ for arbitrary $0\leq l\leq n-3$ or $K_{lm}$ for $(l,m)\le (j,k)$ (for the 
reverse lexicographical ordering). Then $\Hom_{G,S}(L,-)$ 
applied to $\tilde{D}^{\bullet}_{jk}$
yields an exact sequence
\begin{equation}
\label{eq:fundex1}
0\r \Hom_{G,S}(L,\tilde{D}_{jk}^\bullet)\r
\Hom_{G,S}(L,K_{jk})\r \Sscr \r 0
\end{equation}
where $\Sscr=k$ if $L=K_{jk}$ and $\Sscr=0$ in all other cases.
\end{lemma}
\begin{proof}
It is easy to see that
we have to prove
\begin{equation}
\label{eq:havetoprove1}
H^p(\Hom_{G,S}(L,\tilde{C}^\bullet_j))=
\begin{cases}
0&\text{if $p<-k$,}\\
\Sscr&\text{if $p=-k$.}
\end{cases}
\end{equation}
If $L=T_l$ then $\Hom_{G,S}(L,\tilde{C}^\bullet_j)$ is acyclic in degrees $<0$ using \cite[Propositions \ref{prop:exact:good},\ref{tensorgoodfiltration}]{FFRT1} (as $S^lV=D^lV$  by our assumptions on $p$)
and hence there is nothing to prove. We now assume $L=K_{lm}$.

Clearly $\Hom_{G,S}(\tilde{D}_{lm}^\bullet,M_j)$ has zero cohomology in degrees $<0$. 
Hence the same holds for the total complex associated to the double 
complex $\Hom_{G,S}(\tilde{D}^\bullet_{lm},\tilde{C}^\bullet_j)$. We consider the associated $E_1$-spectral sequence (with $\gC^\bullet_j$ horizontally oriented). By the prior discussion we have
$E^{pq}_\infty=0$ for $p+q<0$. To compute the $E_1$-term
 we have to compute the cohomology of 
\[
\Hom_{G,S}(\tilde{D}^\bullet_{lm},T\otimes_k S)=((\tilde{D}^\bullet_{lm})^\vee \otimes_S (T\otimes_k S))^G=(\tilde{C}^\bullet_{\bar{l}\bar{m}}\otimes_k T\otimes_k \wedge^n F^*)^G(n)
\]
for $T$ tilting, where we used \eqref{eq:dualCD} for the second equality. 
We find using the exactness of $(-)^G$ on representations with good filtrations 
\begin{equation}
\label{eq:cohomE1columns}
H^q(\Hom_{G,S}(\tilde{D}_{lm},T\otimes_k S))=
\begin{cases}
\Hom_{G,S}(K_{lm},T\otimes_k S)&\text{if $q=0$},\\
(M_{\bar{l}}\otimes_k T)^G\otimes_k \wedge^n F^*(n)&\text{if $q=\bar{m}$},\\
0&\text{otherwise}.
\end{cases}
\end{equation}
Thus the $E_1$-term of the spectral sequence computing $\Hom_{G,S}(K_{lm},\tilde{C_j}^\bullet)$ consists of two rows respectively given by
\[
\begin{cases}
(M_{\bar{l}}\otimes_S \tilde{C}^\bullet_{j})^G\otimes_k \wedge^n F^\ast(n)&\text{if $q=\bar{m}$},\\
\Hom_{G,S}(K_{lm},\tilde{C}^\bullet_{j})&\text{if $q=0$}.
\end{cases}
\]
Inspection reveals
that $E^{pq}_\infty=0$ for $p+q<0$ implies that the truncated lower row $\sigma_{\le -m}\Hom_{G,S}(K_{lm},\tilde{C}^\bullet_{j})$
is acyclic (except in degree $-m$). 

By the assumption, we either have $m<k$ or else $m=k$ and $l\le j$. If $m<k$ then it follows
that \eqref{eq:havetoprove1} is satisfied. So now assume $m=k$. In that case
$\sigma_{\le -k}\Hom_{G,S}(K_{lm},\tilde{C}^\bullet_{j})$ is acyclic (except in degree $-k$) and 
\[
H^{-k}(\Hom_{G,S}(K_{lm},\tilde{C}^\bullet_{j}))=E_2^{-n+1,\bar{m}}=\ker( E_1^{-n+1,\bar{m}}\r  E_1^{-n+2,\bar{m}}).
\]
We calculate
\begin{align*}
E_1^{-n+1,\bar{m}}&=(M_{\bar{l}}\otimes_k S^{\bar{j}}V\otimes_k\wedge^n F(-n)\otimes \wedge^n F^*(n))^G
&&\\
&=\bigoplus_t (S^{\bar{l}+t}V \otimes_k S^tF(-t)  \otimes_k S^{\bar{j}}V)^G\\
&=\begin{cases}
k&\text{if $l=j$},\\
0&\text{if $l<j$}.
\end{cases}
\end{align*}
So if $l<j$ we are again done. If $l=j$ then it remains to show that $\Hom_{G,S}(K_{jk},\tilde{C}^\bullet_j)$
is \emph{not} exact in degree $-k$. If it were then in \eqref{eq:fundex1}  we should take $\Sscr=0$.
In particular we would have an epimorphism
\[
\Hom_{G,S}(K_{jk},\tilde{D}_{jk}^0)\r \Hom_{G,S}(K_{jk},K_{jk})\r 0.
\]
This would imply that $K_{jk}$ is a direct summand of $\tilde{D}_{jk}^0$, which is impossible. For example because $\tilde{D}_{jk}^0$ is projective and
$K_{jk}$ is not by Proposition \ref{prop:propKjk}\eqref{item:pdim}.
\end{proof}

\bibliographystyle{amsplain}

\end{document}